\documentclass[a4paper,11pt]{amsart}
\usepackage[pagewise]{lineno}
\usepackage{graphicx} 
\usepackage{caption}
\usepackage{subcaption}
\usepackage{tikz-cd}
\usetikzlibrary{shapes} 
\usepackage{pgfplots}
\pgfplotsset{compat=1.15}
\usepackage{mathrsfs}
\usepackage{amsmath}
\usepackage{amsthm}
\usepackage{amsfonts}
\usepackage{amssymb}
\usepackage{mathtools}

\usepackage{thmtools}
\usepackage{thm-restate}

\usepackage[hidelinks]{hyperref}
\hypersetup{ 
  colorlinks   = true,            
  urlcolor     = {blue!50!black}, 
  linkcolor    = {blue!50!black}, 
  citecolor   = {red!50!black}    
}

\newtheorem{theorem}{Theorem}[section]{\bf}{\it}
\newtheorem{lemma}[theorem]{Lemma}{\bf}{\it}
\newtheorem{proposition}[theorem]{Proposition}{\bf}
\newtheorem*{informalthm}{Informal version of main theorem}{\bf}
{\it}
\newtheorem{corollary}[theorem]{Corollary}{\bf}{\it}
{\bf}{\it} 
{\bf}{\it}
\newtheorem*{theorem*}{Theorem}

\newtheorem{remark}[theorem]{Remark}
\newtheorem*{remark*}{Remark}
{\bf}{\it}
{\bf}{\it}

\newtheorem{definition}[theorem]{Definition}
\newtheorem*{definition*}{Definition}

{\bf}{\it}
{\bf}{\it}
{\bf}{\it}

\newtheorem*{example*}{Example}

\theoremstyle{remark}

\theoremstyle{definition}

\theoremstyle{remark}

\usepackage{xcolor}
\definecolor{cbBlue}{HTML}{332288}
\definecolor{cbOrange}{HTML}{882255}
\definecolor{cbPurple}{HTML}{117733}
\definecolor{cbGreen}{HTML}{999933}

\usepackage{enumitem}

\newcommand{\Rips}{\mathrm{Rips}}
\newcommand{\im}{\mathrm{im}}

\newcommand{\dgm}{\mathrm{dgm}}
\newcommand{\Dgm}{\mathrm{Dgm}}
\newcommand{\mdgm}{\mathsf{dgm}}

\newcommand{\R}{\mathbb{R}}
\newcommand{\D}{\mathbb{D}}

\newcommand{\Z}{\mathbb{Z}}
\newcommand{\norm}[1]{\left \lVert #1 \right \rVert}

\newcommand{\spt}{\mathrm{spt}}
\newcommand{\dist}{\mathrm{dist}}
\newcommand{\Haus}{\mathrm{Haus}}
\newcommand{\N}{\mathbb{N}}

\newcommand{\tildeF}{\Tilde{f}_*}
\newcommand{\tildeGf}{\widetilde{(g_f)}_*}
\newcommand{\iX}{\mathsf{i}_X}
\newcommand{\iY}{\mathsf{i}_Y}
\DeclareMathOperator{\rank}{rank}
\DeclareMathOperator{\id}{id}

\usepackage{amsaddr} 

\title[Unveiling Topology in Imaging Problems]{Unveiling Topology in Imaging Problems via Quasi-Isometry and Persistent Homology}

\author{Elli Karvonen$^{*,1}$, Matti Lassas$^{1}$, Pekka Pankka$^{1}$ }
\address{$^{1}$ Department of mathematics and statistics, University of Helsinki, Finland.\\
\textit{E-mail:} \nolinkurl{elli.karvonen@helsinki.fi}, \nolinkurl{matti.lassas@helsinki.fi},\nolinkurl{pekka.pankka@helsinki.fi}}
\thanks{* Corresponding author}

\begin{document}

\maketitle

\section*{Abstract}
We show that the topological structures, such as loops, voids, and higher-dimensional holes
of unknown objects (of flow of an object in space-time) can be recovered from noisy and indirect measurements.
More precisely, we describe how the part of the persistent homology of a space can
be determined from a noise-prone and discretized model space when there is a quasi-isometry between the original space and the space modeling indirect measurements. The result not only guarantees the existence of the structures but also provides size bounds for them. The structure is studied using persistent homology, and the results assume the existence of a quasi-isometry between a model space and the noisy measurements. We explore imaging problems, particularly X-ray imaging and EIT, that are well-suited to this framework.

\section{Introduction}\label{sec:intro}

In inverse problems, the typical task is to determine functions describing an unknown object,
e.g. the density function, from noisy indirect measurements. 
For example, we may consider X-ray images or electrical boundary measurements of the object, and then the problem is to reconstruct the interior structure of the object from these measurements. 
In this paper, we approach inverse problems from a different perspective. Rather than directly solving functions
modeling the materials' parameters, we are determining the topological properties.
We also consider noisy time-dependent observations of a moving object,  like a beating heart or moving water in a piping system. The question is whether we can be sure there is a loop or higher-dimensional structure in the model space if we observe the same-dimensional structure in noisy measurements.
It turns out that under certain limitations and error bounds, the structure of the unknown model space can be observed from incomplete and noisy measurements. Even though the results are presented in the context of inverse problems, we believe that one can find this kind of structural stability result interesting in other application areas also, for example, in dimension reduction.

The observation is achieved using persistent homology with Rips complexes and the $\Z_2$ coefficient field. Persistent homology is a computational approach that detects topological structures, such as components, loops, voids, and higher-dimensional holes. 
Additionally, persistent homology tracks on which scales $r$ the features appear and disappear when one covers a set $X$ by balls of radius $r$, or equivalently, one considers an $r$-neighborhood of the set $X$. One can think that looking at a data set very closely (small scale) allows one to see details, like distinct data points. Looking at the same data from far away (big scale), the details start to merge, and only large structures can be detected. These structural changes of space $X$ are collected to an indexed persistent diagram, $\Dgm_p(X)$, where $p$ denotes the dimension of the homology in question. The indexed persistent diagrams are sets of elements $(k,(a,b))$, where $k$ is the index and $a>0$ presents the scale $r$ when the structure appears (birth time) and $b>0$ is its disappearing scale (death time). Illustratively, the indexed persistence diagram is presented as a simplified persistence diagram, or so to speak, a ``birth-death'' diagram, a subset of $\R^2$ as in Figure \ref{fig:simple_example}. Essentially, a simplified persistence diagram is an indexed persistence diagram without the indices. 
The formal definition of these persistence diagrams is given in Section \ref{sec:per_dgm_stab}, Definition \ref{def:persistence_diagrams}. 

We consider the inverse problem of computing the persistent homology of a metric space $X_{cont.}$ 
(e.g. a manifold)
from the noisy data space $Y_{noisy}$, which is, for example, a finite subset of a Euclidean space. We assume that the model and data are given by mappings
\[
X_{cont.}\to X_{net}\to Y_{clean}\to Y_{noisy}
\]
where $X_{net}$ is the discrete sampling from $X_{cont.}$ or more precisely finite $\frac{1}{2}\varepsilon_0$-net of $X_{cont.}$, and $Y_{clean}$ is noise-free measurements from $X_{net}$. Also, $Y_{noisy}$ can be considered as a discrete, noisy version of $Y_{clean}$.
The mappings between spaces are quasi-isometries (see Section \ref{sec:quasi}) with parameters $(1,\varepsilon_0,0)$, $(L,\varepsilon_1,0)$, and $(1,\varepsilon_2,0)$, respectively, and we call this model an \emph{$(L,\varepsilon_0,\varepsilon_1, \varepsilon_2)$-model}. 

One simple example of the concerned model is presented in Figure \ref{fig:simple_example}, where the model space is a unit circle, $X_{cont.}=S^1\subset \R^2$, $X_{net}$ is the uniform finite sampling of $X_{cont.}$, $Y_{clean}$ is exactly same as $X_{net}$ and the $Y_{noisy}$ is shifted circle points. This corresponds to the simple model where the measurements or the direct map is the identity map, that is, measurements are noise samples from the set $X_{cont.}$.
 The simplified persistence diagrams for each space in the model are also presented. 
More sophisticated examples of inverse problems that fit into this modeling framework are the inverse problems for the Radon transform and for conductivity; see Section \ref{sec:model_inverse_prob} and Section \ref{sec:quasi}. In these cases, the $X_{cont.}$ would be the objects we might measure. Let's say the space of annuli with different inner and outer radii, like in the upcoming computational example  (Subsection \ref{subsec:comp_radon}). The space $X_{net}$ is then some subset of $X_{cont.}$, those elements that one is actually measuring with some device. Ideally, it is a dense sampling of $X_{cont.}$. The space $Y_{clean}$ contains the ideal noise-free measurements. The space $Y_{noisy}$ is the noisy measurements that one observes. 

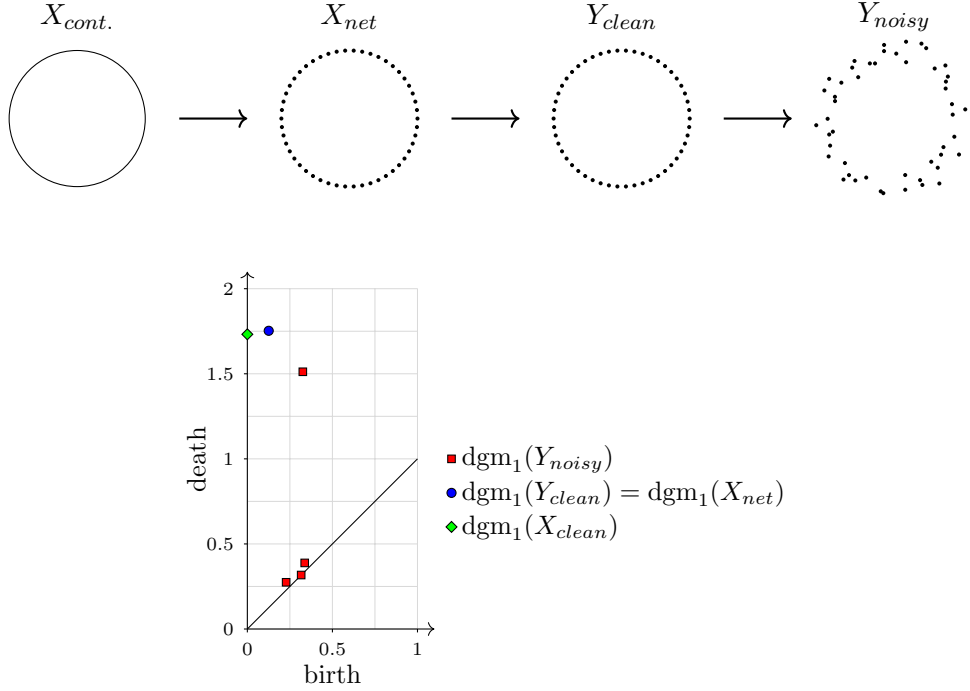
\begin{figure}[!htb]
    \centering
\begin{tikzpicture}[scale=0.9]
\node[] at (0,1.5) {$X_{cont.}$};
\draw (0,0) circle (1);
\draw[->, thick] (1.5,0) -- (2.5,0);
\draw[->, thick] (5.5,0) -- (6.5,0);
\draw[->, thick] (9.5,0) -- (10.5,0);
\begin{scope}[shift={(4,0)}]
   \node[] at (0,1.5) {$X_{net}$};
  \pgfplotstableread[col sep=space]{datafiles/points.dat}\datatable
  \foreach \i in {0,...,49} {
    \pgfplotstablegetelem{\i}{[index]0}\of\datatable
    \let\x\pgfplotsretval
    \pgfplotstablegetelem{\i}{[index]1}\of\datatable
    \let\y\pgfplotsretval
    \fill (\x,\y) circle (0.03);
  }
\end{scope}
\begin{scope}[shift={(8,0)}]
   \node[] at (0,1.5) {$Y_{clean}$};
  \pgfplotstableread[col sep=space]{datafiles/points.dat}\datatable
  \foreach \i in {0,...,49} {
    \pgfplotstablegetelem{\i}{[index]0}\of\datatable
    \let\x\pgfplotsretval
    \pgfplotstablegetelem{\i}{[index]1}\of\datatable
    \let\y\pgfplotsretval
    \fill (\x,\y) circle (0.03);
  }
\end{scope}
\begin{scope}[shift={(12,0)}]
   \node[] at (0,1.5) {$Y_{noisy}$};
  \pgfplotstableread[col sep=space]{datafiles/points_rand.dat}\datatable
  \foreach \i in {0,...,49} {
    \pgfplotstablegetelem{\i}{[index]0}\of\datatable
    \let\x\pgfplotsretval
    \pgfplotstablegetelem{\i}{[index]1}\of\datatable
    \let\y\pgfplotsretval
    \fill (\x,\y) circle (0.03);
  }
\end{scope}

\begin{scope}[scale=2.5,shift={(1,-3)}]
     \node[below] at (0.5,-0.15) {\small{birth}}; 
     \node[left,rotate=90] at (-0.3,1.25) {\small{death}}; 
    \draw[step=0.25, gray!30, very thin] (0,0) grid (1,2);
    \foreach \x in {0,0.5,1}
    \draw (\x,0) -- (\x,-0.02) node[below] {\tiny{\x}};
        \foreach \y in {0,0.5,1,1.5,2}
        \draw (0,\y) -- (-0.02,\y) node[left] {\tiny{\y}};
    \draw[->](0,0)--(0,2.1); 
    \draw[->](0,0)--(1.1,0); 
    \draw[-](0,0)--(1,1); 
  \pgfplotstableread[col sep=space]{datafiles/dgm_rand.dat}\datatable
  \foreach \i in {0,...,3} {
    \pgfplotstablegetelem{\i}{[index]0}\of\datatable
    \let\x\pgfplotsretval
    \pgfplotstablegetelem{\i}{[index]1}\of\datatable
    \let\y\pgfplotsretval
        \node[draw, fill=red, minimum size=1mm, inner sep=0pt, shape=rectangle] at (\x,\y){};
  }
  \pgfplotstableread[col sep=space]{datafiles/dgm.dat}\datatable
  \foreach \i in {0} {
    \pgfplotstablegetelem{\i}{[index]0}\of\datatable
    \let\x\pgfplotsretval
    \pgfplotstablegetelem{\i}{[index]1}\of\datatable
    \let\y\pgfplotsretval
        \node[draw, fill=blue, minimum size=1.2mm, inner sep=0pt, shape=circle] at (\x,\y){};
  }
\node[draw, fill=green, minimum size=1.5mm, inner sep=0pt, shape=diamond] at (0,1.73205081){};
\node[draw, fill=red, minimum size=1mm, inner sep=0pt, shape=rectangle] at (1.2,1){};
\node[right] at (1.2,1) {\small{$\dgm_1(Y_{noisy})$}}; %
\node[draw, fill=blue, minimum size=1.2mm, inner sep=0pt, shape=circle] at (1.2,0.8){};
\node[right] at (1.2,0.8) {\small{$\dgm_1(Y_{clean})=\dgm_1(X_{net})$}}; %
\node[draw, fill=green, minimum size=1.5mm, inner sep=0pt, shape=diamond] at (1.2,0.6){};
\node[right] at (1.2,0.6) {\small{$\dgm_1(X_{clean})$}}; %
\end{scope}
\end{tikzpicture}
\caption{A simplified example when the measurements are obtained using the identity map. {\bf Top.} The (continuous) model space is $X_{cont.}=S^1\subset\R^2$, the sampled space is $X_{net}=\{y_j\mid y_j=(\cos\theta_j,\sin\theta_j),\theta_j=\frac{2\pi j}{50},$ $j=0,\dots,49\}$, the clean measurements are taken by identity map, i.e., $Y_{clean}=X_{net}$, and a bounded random noise is added to the measurements, $Y_{noisy}= \{y_j+e_j\mid y_j\in Y_{clean}, \norm{e_j}_2\leq 0.2, j=0,\dots,49\}$. The $X_{cont.}\to X_{net}\to Y_{clean}\to Y_{noisy}$ is $(1,0.126,0,0.4)$-model. 
{\bf Bottom.}
The topological structures of the above spaces are studied by persistence homology of dimension one. The structures for each of these spaces are presented as a simplified persistence diagram $\dgm_1(\cdot)$. The indexed persistence diagram, as well as the simplified persistence diagram, reveals that $X_{cont.}$ has one one-dimensional hole existing from the scale 0 to the scale $1.73$, see the green point in the birth-death diagram. The space $X_{net}$ (as well as $Y_{clean}$) has a hole existing from scale $0.13$ to  $1.75$ (blue point). The space $Y_{noisy}$ has four holes (red points), three existing for a very short time (diagram points near the diagonal), and one hole that exists longer, from the scale $0.33$ to the scale $1.51$.}
\label{fig:simple_example}
\end{figure}
In this paper, we explore the structural observational stability of this model in terms of persistent homology. Heuristically, we show that the topological properties (in terms of homology) which are observable from the data $Y_{noisy}$ depend stably (in terms of the parameters) on the corresponding topological properties of the space $X_{cont.}$. We emphasize that noisy data $Y_{noisy}$ is obtained from indirect measurements for which the inverse problem of determining the original set $X_{cont.}$ is an ill-posed problem. 

\begin{informalthm}
Let $(a,b)\in \R^2$ be a point in the persistence diagram of the 
space $Y_{noisy}$ of finite and noisy indirect measurements and assume that the noise level is given.
When the point  $(a,b)$ is above a given threshold level, we can determine 
a  rectangle $R\subset \R^2$ that has to contain a point $(c,d)\in R$
in the persistence diagram of the 
the measured object  $X_{cont.}$. That is, an observed 
topological structure in the set of noisy and discrete indirect measurements that is above a threshold level
implies the existence of a similar topological structure in the object from which the measurements
are taken.
\end{informalthm}

In the statement, we denote $\mathbb H_p(X_{cont.})$ the $p${th} persistence module of the space $X_{cont.}$ (see Subsection \ref{sec:persistence_modules} for definition). The assumption that $\mathbb{H}_p(X_{cont.})$ is decomposable guarantees that the indexed persistence diagram of $X_{cont.}$ is well defined (see Section \ref{sec:per_dgm_stab}). 
Rigorously, our main result on the stability of indexed persistence diagrams reads as follows.

\begin{restatable}{theorem}{MainTheorem}
\label{main-thm-intro}
    Suppose that $X_{cont.}\to X_{net}\to Y_{clean} \to Y_{noisy}$ is an $(L,\varepsilon_0,\varepsilon_1,\varepsilon_2)$-model and $\mathbb{H}_p(X_{cont.})$ is decomposable. Let $\vartheta(t)=L(t+\varepsilon_2+\varepsilon_1)+\varepsilon_0$, and $\varphi^{-1}(t)=({t-L\varepsilon_0-\varepsilon_1-\varepsilon_2})/{L}$. Denote 
    $$\Dgm^{(L,\varepsilon_0,\varepsilon_1,\varepsilon_2)}_p(Y_{noisy}):=\{(k,(a,b))\in \Dgm_p(Y_{noisy})\colon \vartheta(a)<\varphi^{-1}(b)\}.$$
     Then there exists an injection
    \[
    \boxdot\colon \Dgm^{(L,\varepsilon_0,\varepsilon_1,\varepsilon_2)}_p(Y_{noisy}) \to \Dgm_p(X_{cont.})
    \] such that for every $(k,(a,b))\in \Dgm^{(L,\varepsilon_0,\varepsilon_1,\varepsilon_2)}_p(Y_{noisy})$,  $ \boxdot\big((k,(a,b))\big)=(l,(c,d))\in \Dgm_p(X_{cont.})$, it holds that 
    \[
(c,d) \in [\varphi^{-1}(a),  \vartheta(a)  ]\times [\varphi^{-1}(b),  \vartheta(b)].
    \]
\end{restatable}

In the theorem, each point $(k,(a,b))\in\Dgm_p(Y_{noisy})$  corresponds to a structure in the finite space $Y_{noisy}$, e.g., a loop in $p=1$, which exists from the positive scale $a$ to the scale $b$, $b>a$. If the pair $(a,b)$ happen to exist long enough, more precisely $\vartheta(a)<\varphi^{-1}(b)$, then we get \emph{an estimation box} $[\varphi^{-1}(a),\vartheta(a)]\times[\varphi^{-1}(b),\vartheta(b)]$. The result states that, in each box, we have a diagram point of $X_{cont.}$, meaning that there is the same-dimensional structure in the space $X_{cont.}$. Furthermore, the box sets limits on the existence of the found structure. Subfigure \ref{fig:estimation_simple} is an illustration of this result applied to the simple example model presented in Figure \ref{fig:simple_example}. The Subfigure \ref{fig:estimation_complex} illustrates the situation where model parameter $L\neq1$ causes the estimation boxes to be rectangles of different sizes.   

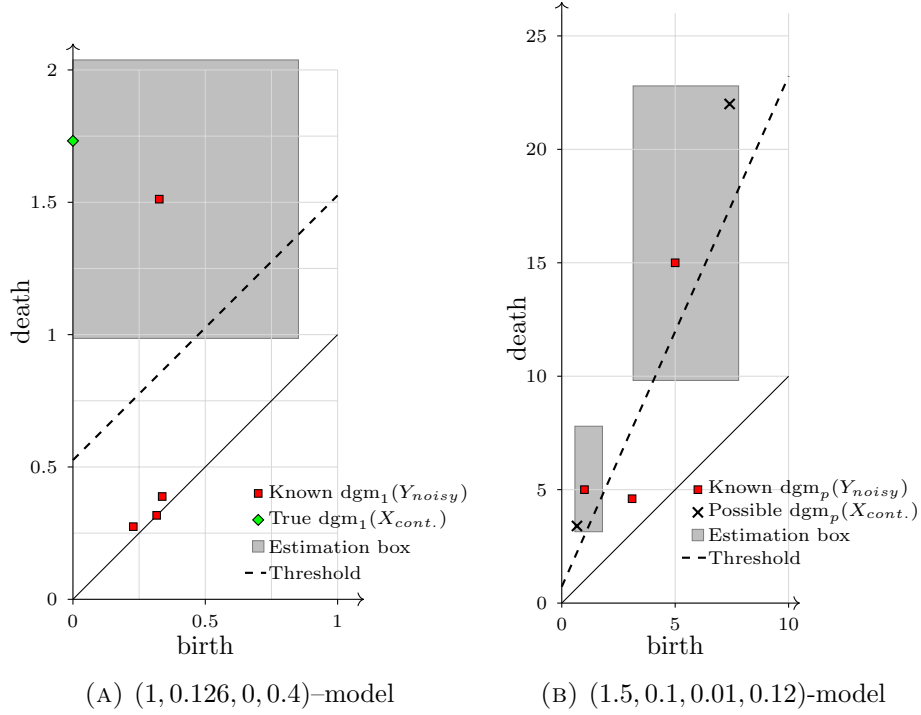
\begin{figure}[tbh!]
\subcaptionbox{ $(1,0.126,0,0.4)$--model\label{fig:estimation_simple}}{%
\begin{tikzpicture}[scale=3.5]
    \draw[color=black, fill=gray, opacity=0.5] (0, 1.5118000507354736-0.526) rectangle (0.32610002160072327+0.526, 1.5118000507354736+0.526); 
    \node[below] at (0.5,-0.1) {\small{birth}}; 
    \node[left,rotate=90] at (-0.2,1.25) {\small{death}}; 
    \draw[step=0.25, gray!30, very thin] (0,0) grid (1,2); 
    \foreach \x in {0,0.5,1}
    \draw (\x,0) -- (\x,-0.02) node[below] {\tiny{\x}};
    \foreach \y in {0,0.5,1,1.5,2}
    \draw (0,\y) -- (-0.02,\y) node[left] {\tiny{\y}};
    \draw[->](0,0)--(0,2.1); 
    \draw[->](0,0)--(1.1,0); 
    \draw[-](0,0)--(1,1); 
    \draw[dashed,thick,color=black](0,0.526)--(1,1.526); 
    \pgfplotstableread[col sep=space]{datafiles/dgm_rand.dat}\datatable 
    \foreach \i in {0,...,3} {
        \pgfplotstablegetelem{\i}{[index]0}\of\datatable
        \let\x\pgfplotsretval
        \pgfplotstablegetelem{\i}{[index]1}\of\datatable
        \let\y\pgfplotsretval
        \node[draw, fill=red, minimum size=1mm, inner sep=0pt, shape=rectangle] at (\x,\y){}; 
    }
    \node[draw, fill=green, minimum size=1.5mm, inner sep=0pt, shape=diamond] at (0,1.73205081){};
    \node[draw, fill=red, minimum size=1mm, inner sep=0pt, shape=rectangle] at (0.7,0.4){};
    \node[right] at (0.7,0.4) {\tiny{Known $\dgm_1(Y_{noisy})$}}; %
    \node[draw, fill=green, minimum size=1.5mm, inner sep=0pt, shape=diamond] at (0.7,0.3){};
    \node[right] at (0.7,0.3) {\tiny{True $\dgm_1(X_{cont.})$}}; %
    \node[draw, fill=gray, minimum size=1.5mm, inner sep=0pt, shape=rectangle,opacity=0.5] at (0.7,0.2){};
    \node[right] at (0.7,0.2) {\tiny{Estimation box}}; %
    \draw[dashed,thick,color=black](0.65,0.1)--(0.73,0.1);;
    \node[right] at (0.7,0.1) {\tiny{Threshold}}; 
    \end{tikzpicture}
    }
%
\subcaptionbox{$(1.5,0.1,0.01,0.12)$-model\label{fig:estimation_complex}}{%
\begin{tikzpicture}[scale=0.3]
    \draw[color=black, fill=gray, opacity=0.5] (3.1467,9.8133) rectangle (7.7950,22.7950); 
    \draw[color=black, fill=gray, opacity=0.5] (0.58,3.1467) rectangle (1.7950,7.7950); 
    \draw[step=5cm, gray!30, very thin] (0,0) grid (10,26); 
    \draw[->] (0,0) -- (10.5,0); 
    \draw[->] (0,0) -- (0,26.5); 
    \node[below] at (5,-1) {\small{birth}}; 
    \node[left,rotate=90] at (-2,14) {\small{death}}; 
    \foreach \x in {0,5,10}
        \draw (\x,0) -- (\x,-0.2) node[below] {\tiny{\x}};
    \foreach \y in {0,5,10,15,20,25}
        \draw (0,\y) -- (-0.2,\y) node[left] {\tiny{\y}};
    \draw[-](0,0)--(10,10); 
    \draw[dashed,thick,color=black](0,0.7225)--(10,23.2225); 
    \node[draw, fill=red, minimum size=1mm, inner sep=0pt, shape=rectangle] at (5,15) {};
    \node[draw, fill=red, minimum size=1mm, inner sep=0pt, shape=rectangle] at (1,5) {};
    \node[draw, fill=red, minimum size=1mm, inner sep=0pt, shape=rectangle] at (3.1,4.6) {};
    \draw (7.4,22) node[cross out, draw=black, minimum size=3pt, inner sep=0pt, thick] {};
    \draw (0.67,3.4) node[cross out, draw=black, minimum size=3pt, inner sep=0pt, thick]{}; 
    \node[draw, fill=red, minimum size=1mm, inner sep=0pt, shape=rectangle] at (6,5){};
    \node[right] at (6,5) {\tiny{Known $\dgm_p(Y_{noisy})$}}; %
    \draw (6,4) node[cross out, draw=black, minimum size=3pt, inner sep=0pt, thick] {};  \node[right] at (6,4) {\tiny{Possible $\dgm_p({X_{cont.}})$}}; %
    \node[draw, fill=gray, minimum size=1.5mm, inner sep=0pt, shape=rectangle,opacity=0.5] at (6,3){};
    \node[right] at (6,3) {\tiny{Estimation box}}; %
    \draw[dashed,thick,color=black] (5.3,2)--(6.4,2) {};
    \node[right] at (6,2) {\tiny{Threshold}};
    \end{tikzpicture}
    }
\caption{Illustration of main Theorem \ref{main-thm-intro} using two models. The red markers denote the known indexed persistence diagram points $\Dgm_p(Y_{noisy})$ in the simplified persistence diagram $\dgm_p(Y_{noisy})$ in both models. Above dashed line (threshold) are the points $(k,(a,b))\in \Dgm_p(Y_{noisy})$ for which $\vartheta(a)<\varphi^{-1}(b)$. The pairs $(a,b)$ define gray estimation boxes. Inside these boxes must be a point of the indexed persistence diagram $\Dgm_p(X_{cont.})$. Possible indexed persistence diagram points of $\Dgm_p(X_{cont.})$ are marked with crosses in the simplified persistence diagrams.} 
\end{figure}

The stability in terms of persistent homology has also been studied previously under different assumptions \cite{pershom_classical_stability,cohen-steiner_lipschitz_2010, chazal:geomcomp, ellipsoid_stability_preprint_Kalisnik}, but to our knowledge, never under quasi-isometry. 
 Results in this paper are highly inspired by the stability result proved by Chazal, Silva, and Oudot \cite{chazal:geomcomp}. Formally, the result by Chazal et. al. is that for totally bounded metric spaces $X$ and $Y$ the distance between (extended) persistence diagrams, $\overline{\Dgm}_p(\cdot)$, is less than two times the Gromov-Hausdorff distance between these spaces; 
\begin{align}\label{eq:bottleneck_result}
d_{\text{bottleneck}}(\overline{\Dgm}_p(X),\overline{\Dgm}_p(Y))\leq 2 d_{\text{GH}}(X,Y),
\end{align}
where $d_{\text{bottleneck}}$ is the bottleneck distance between two persistence diagrams, see \eqref{def:bottleneck}, which are extended to have all points on the diagonal with infinity multiplicity, and $d_{\text{GH}}$ is the Gromov-Hausdorff distance. The results hold for several types of complexes, including Rips complexes, over any coefficient field.

Coarsely speaking, the bottleneck distance matches two diagrams as well as possible by minimizing the $\norm{\cdot}_\infty$-norm between matched points. The diagram points are also allowed to be matched to a diagonal, i.e., to point $(a, a)\in\R^2$. The bottleneck distance is then the largest sup norm of these matched diagram points.

The key concept in the result \eqref{eq:bottleneck_result} from \cite{chazal:geomcomp} is $\varepsilon$-\emph{interleaving}, allowing one to compare and move between the persistent homologies. Notably, if there is $(1,\varepsilon,\delta)$-quasi-isometry from $X$ to $Y$, then  $d_{GH}(X,Y)\leq\frac{1}{2}(\varepsilon+2\delta)$ (similar proof than Burago-Burago-Ivanov Corollary 7.3.28  and related definitions \cite{BurBurIva}). 
The result \eqref{eq:bottleneck_result} 
is global in nature, whereas our result gives point-wise stability.
Even though the formulation of our main theorem (Theorem \ref{main-thm-intro}) and the Chazal et al. result \eqref{eq:bottleneck_result} seem very different, in some cases, one can derive the same information from them. For example, the persistence diagram estimation boxes for $X_{cont.}=S^1$ in the previous simple example model (Figure \ref{fig:estimation_simple}) can be derived also from inequality \eqref{eq:bottleneck_result}, when knowing only the indexed persistence diagram of $Y_{noisy}$ and model parameters $L$, $\varepsilon_0$, $\varepsilon_1$ and $\varepsilon_2$.
However, we found that estimating the Gromov-Hausdorff distance is challenging, or even impossible, in many cases, whereas constructing a quasi-isometry is more feasible and naturally arises in some inverse problems. We noted that we can formalize the same kind of configuration as $\varepsilon$-interleaving between homology groups using quasi-isometry.

This paper is organized as follows. Section \ref{sec:model_inverse_prob} introduces the concerned model in more detail, giving needed terminology. Furthermore, classical stability results of the Radon transform and the conductivity problem are presented, and their connection to previous terminology is given. Section \ref{sec:quasi} focuses on quasi-isometry, the related results, and how the given examples suit the quasi-isometry framework. Section \ref{sec:per_hom_stab} gives an introduction to persistent homology, and the stability result for general metric spaces $X$ and $Y$ (Theorem \ref{thm:sup-inf-interval2}). Notably, Section \ref{sec:per_hom_stab} with \ref{sec:quasi} is mostly independent from previous Sections, and one can find the results of it interesting as such. The theoretical part of persistent homology is continued in Section \ref{sec:per_dgm_stab}, where persistent diagrams are formalized with needed terminology and results. The theoretical part of the paper is finished in Section \ref{sec:stab_model}, where the results in the previous Section are connected to the concerned $(L,\varepsilon_0,\varepsilon_1,\varepsilon_2)$-model. The paper is finished with some computational examples (Section \ref{sec:computational}) and conclusion (Section \ref{sec:conclusion}).

\section{Model inverse problems}\label{sec:model_inverse_prob}

Inverse problems are often formulated as 
\begin{align}\label{eq:inverseproblem}
    m = F(x)+E,
\end{align}
where $m$ is the known indirect measurement, $F$ is the forward model (linear or non-linear), $x$ is the unknown cause and $E$ is additional noise. In this paper, we are interested in observing the structure of a 
unknown model space $X_{cont.}$ from a noisy, incomplete dataset $Y_{noisy}$. The measurements are taken from finite $\frac{1}{2}\varepsilon_0$-net, denoted by $X_{net}$ of $X_{cont.}$. That is,
$X_{net}\subset X_{cont.}$ and there exists a mapping $f^{(0)}\colon X_{cont.}\to X_{net}$ such that $d_{X_{cont.}}(x,f^{(0)}(x))\leq \frac{1}{2}\varepsilon_0$ for every $x\in X_{cont.}$. We denote that a clean dataset is $Y_{clean}=\{F(x) \mid x\in X_{net}\}$, and noisy dataset is \[Y_{noisy}=\{y_j+E_j\mid  y_j\in Y_{clean}, \norm{E_j}\leq \frac{1}{2}\varepsilon_2 \}.\]
The structure observation problem can be divided into subproblems, namely:
\begin{enumerate}[label={(Q\arabic*)}]
    \item Can we observe topological structure of clean data $Y_{clean}$ from noisy data $Y_{noisy}$?\label{q1}
    \item Can we observe topological structure of sampled model space $X_{net}$ from clean data $Y_{clean}$?\label{q2}
    \item Can we observe topological structure of full model space $X_{cont.}$ from incomplete model space $X_{net}$?\label{q3}
\end{enumerate} 

 The mapping properties of $F$ and $\varepsilon$-density of sampled spaces guarantee that the opposite
 directions of the above questions \ref{q1}-\ref{q3} are valid.
That means we have
\begin{align*}
    X_{cont.}\xrightleftharpoons[\text{\ref{q3}}]{\frac{1}{2}\varepsilon_0-\text{net}} X_{net} \xrightleftharpoons[\text{\ref{q2}}]{F}Y_{clean}\xrightleftharpoons[\text{\ref{q1}}]{\text{noise}}Y_{noisy}. 
\end{align*}
In order to answer the questions \ref{q1}-\ref{q3}, we will go through the tools to do so. We will restrict to cases where the model $F$ has certain properties, namely that the function is quasi-up-increasing. 

\begin{definition}
    Let $\omega_1,\omega_2\colon  [0,\infty) \to [0,\infty)$ be twice differentiable functions on $(0,\infty)$, and $\varepsilon_1,\varepsilon_2 \ge 0$. A mapping $f\colon (X,d_X)\to (Y,d_Y)$ between metric spaces is $(t\mapsto \omega_1(t)+\varepsilon_1;t\mapsto \omega_2(t)+\varepsilon_2)$-\textbf{quasi-up-increasing} if 
\begin{enumerate}
    \item $d_Y(f(x_1),f(x_2)) \le \omega_1(d_X(x_1,x_2)) + \varepsilon_1$ for all $x_1,x_2\in X$, \label{quasi-up-increasing:forward}
    \item $d_X(x_1,x_2) \le \omega_2(d_Y(f(x_1),f(x_2))) + \varepsilon_2$ for all $x_1,x_2\in X$, \label{quasi-up-increasing:inverse}
    \item the first derivatives are strictly positive,  $D\omega_1(t)>0$ and $D\omega_2(t)>0$, for every $t\in (0,\infty)$, \label{quasi-up-increasing:firstderivative}
    \item  the second derivatives are negative, $D^2\omega_1(t)\leq 0$ and $D^2\omega_2(t)\leq0$, for every $t\in (0,\infty)$, and \label{quasi-up-increasing:secondderivative}
    \item\label{quasi-up-increasing:zero} $\omega_1(0)=0$ and $\omega_2(0)=0$.
\end{enumerate}
\end{definition}

Now we move to the examples, where the forward map is quasi-up-increasing.

\subsection{Radon transform}\label{sec:radon}

We consider the case in which the mapping $F$ in equation (\ref{eq:inverseproblem}) is the Radon transform. 
It is widely studied since it is related to several imaging problems, such as medical X-ray imaging. The Radon transform gives the mathematical background for the problem. \cite{Cormack1963,Cormack1964,Natterer} Here, we focus on the stability result of the Radon transform.  
At the end of the section, we have a proposition stating that the Radon transform, restricted to a closed ball, is a quasi-up-increasing mapping. We start with some basic definitions and notations.

\subsubsection{Preliminaries of Radon transform}

Let us first introduce the basic notation following  Natterer's classical book \cite{Natterer}. Denote that $\D^n
\subset \mathbb R^n$ is a unit disk of dimension $n$ and $Z$ an unit cylinder $S^{n-1}\times \R$. The Fourier transformation for a function $u\colon \R^n\to \R$, $u\in L^1(\R^n)$, is
\[
\Hat{u}(\xi)=(2\pi)^{-\frac{n}{2}}\int_{\R^n} e^{-ix\cdot\xi}u(x) \,dx
\]
and for the function on $u\colon Z\to \R$, the Fourier transformation is taken with respect to the second variable $s\in \R$
\[
\Hat{u}(\theta,\sigma)=(2\pi)^{-\frac{1}{2}}\int_{\R} e^{-is\sigma}u(\theta,s) \,ds,
\] where $\theta \in S^{n-1}$.

The Sobolev space $H^{{m}}(\R^n)$, ${{m}}\in \N$, is 
$\{u\in L^2(\R^n) \mid D^\beta u\in L^2(\R^n) \text{ for } |\beta|\leq {{m}}\}$. For ${{m}}>0$, the Sobolev space can be written as
\[H^{{{m}}}(\R^n)=\{u\in L^2(\R^n)\colon \norm{u}^2_{H^{{m}}(\R^n)}<\infty \},
\]
where
\[
\norm{u}^2_{H^{{m}}(\R^n)}=\int_{\R^n}(1+|\xi|^2)^{{m}}|\Hat{u}(\xi)|^2\,d\xi.
\]
Furthermore, we denote that 
\[
H_0^{{{m}}}(\D^n)=\{u\in H^{{{m}}}(\R^n) \mid \text{supp}(u)\subseteq \D^n \}.
\]  
For ${{m}}< 0$, the Sobolev space $H^{{m}}$ is a dual space of $H^{-{{m}}}$. 
The norm of $H^{{{m}}}(Z)$ norm gets form
\[
\norm{u}_{H^{{{m}}}(Z)}^2 =\int_{S^{n-1}}\int_{\R}(1+\sigma^2)^{{{m}}}\lvert \Hat{u}(\theta,\sigma)\rvert^2\,d\sigma\,d\theta.
\]  

Now we can define the Radon transform.
The $n$-dimensional Radon transform of a function $u\colon \R^n\to \R$, $u\in L^1(\R^n)$, is 
\[
\mathcal{R}u(\theta,s)=\int_{x\cdot \theta = s}u(x)\,dx,
\]
where $\theta\in S^{n-1}$ and $s\in \R$.
\subsubsection{Stability of the inverse Radon transform}
There is a well-known stability result for the Radon transform (see e.g. \cite[Theorem 5.1 pp. 42--43 ]{Natterer}).
\begin{theorem}\label{thrm:radon_stability}
For each ${{m}}$ there exist $C_1,C_2>0$ such that for $u\in C^{\infty}_0(\D^n)$ \begin{align}\label{ineq:radon}
C_1\norm{u}_{H^{{{m}}}_0(\D^n)}\leq \norm{\mathcal{R}u}_{H^{{{m}}+\frac{n-1}{2}}(Z)}\leq C_2\norm{u}_{H^{{{m}}}_0(\D^n)}. 
\end{align}
\end{theorem}

Since $C^\infty_0(\D^n)$ is dense in $ H^{ {{m}}}_0(\D^n)$, ${{m}}>0$, the result also holds for $u\in H^{ {{m}}}_0(\D^n)$.
One can later see that Theorem \ref{thrm:radon_stability} already yields a quasi-isometry between the metric spaces $H^{{{m}}+\frac{n-1}{2}}(Z)$ and $H^{{{m}}}_0(\D^n)$. However, the metrics are rather complicated from a computational perspective. We prefer to work with metrics induced by the $L^2$ norm. For $L^2$ norm, we get following proposition.

\begin{proposition}
\label{prop:radon_hölder}
 Let $n\ge 2$, $s>0$ and $K>0$. Then there exists constants $C_1,C_2>0$ and $0<p<1$ such that the restriction of the Radon transform $\mathcal R \colon L^2_0(\D^n) \to L^2(Z)$ to the closed ball $B_s(K)=\{ u \in H^s_0(\D^n) \colon \norm{u}_{H^s_0(\D^n)} \le K\}$,
 that is, the map
 \begin{equation}
    \mathcal R\colon   B_s(K)\to L^2(Z),
 \end{equation}
 is a $(t\mapsto C_1t;t\mapsto C_2t^p)$-quasi-up-increasing mapping.  
\end{proposition}

To prove the above proposition, we need the following interpolation result, see e.g.  \cite[pp. 203--204]{Natterer} for details.
\begin{lemma}[Interpolation Lemma]\label{lemma:interpolation}
For $0\leq \beta \leq 1$, ${{m}}_1<{{m}}_2$ and ${{m}}_0 = \beta {{m}}_2+(1-\beta){{m}}_1$, there is a constant $C>0$  such that, for $u\in H^{{{m}}_2}(\D^n)$, the inequality 
 \[
 \norm{u}_{H^{{{m}}_0}(\D^n)}\leq C  \norm{u}_{H^{{{m}}_1}(\D^n)}^{1-\beta}  \norm{u}_{H^{{{m}}_2}(\D^n)}^\beta
\]
holds.
\end{lemma}

\begin{proof}[Proof of Proposition \ref{prop:radon_hölder}]
 Choose ${{m}}=-\frac{n-1}{2}$. By Theorem \ref{thrm:radon_stability} there exist constants $C'_1,C'_2>0$ such that inequality \eqref{ineq:radon} is form
    \begin{align}\label{eq:radon_start_ineq}
    C'_1\norm{u_1-u_2}_{H^{-\frac{n-1}{2}}_0(\D^n)}\leq \norm{\mathcal{R}u_1-\mathcal{R}u_2}_{L^2(Z)}\leq C'_2\norm{u_1-u_2}_{H^{-\frac{n-1}{2}}_0(\D^n)}.
    \end{align}
    Notably $u_1-u_2\in L_0^2(\D^n)$ and $u_1-u_2\in H^{-\frac{n-1}{2}}_0(\D^n)$. Especially there exists constant $C_3'>0$ such that $\norm{u_1-u_2}_{H^{-\frac{n-1}{2}}_0(\D^n)}\leq C_3'\norm{u_1-u_2}_{L^2_0(\D^n)}$. Thus, we get that 
    \[
    \norm{\mathcal{R}u_1-\mathcal{R}u_2}_{L^2(Z)}\leq C'_2\norm{u_1-u_2}_{H^{-\frac{n-1}{2}}_0(\D^n)}\leq C'_2C_3'\norm{u_1-u_2}_{L^2_0(\D^n)}.
    \] Denote that $C_1:=C_2'C_3'$, and $\omega_1\colon[0,\infty)\to [0,\infty)$, $t\mapsto C_1t$. 
    Now condition \eqref{quasi-up-increasing:forward} of quasi-up-increasing mapping, and conditions \eqref{quasi-up-increasing:zero}, \eqref{quasi-up-increasing:firstderivative} and \eqref{quasi-up-increasing:secondderivative} for $\omega_1$ are satisfied. 
    
    Next, we prove the remaining conditions.
    Choose ${{m}}_0=0$, ${{m}}_1=-\frac{n-1}{2}$ and ${{m}}_2=s$. Then by solving ${{m}}_0=\beta{{m}}_2+(1-\beta){{m}}_1$, we obtain that $\beta=\frac{n-1}{2s+n-1}$. Notably $2s+n-1\neq 0$, and furthermore $0< \beta <1$.
    By Lemma \ref{lemma:interpolation} we obtain that there exits a constant $C_4'>0$ such that
    \[
    \norm{u_1-u_2}_{L^2_0(\D^n)}\leq C_4'\norm{u_1-u_2}_{H_0^{-\frac{n-1}{2}}(\D^n)}^{1-\beta}\norm{u_1-u_2}_{H^s_0(\D^n)}^{\beta}.
    \]
    Suppose $u_1,u_2\in B_s(K)$. We have that 
    \[
    \norm{u_1-u_2}_{L^2_0(\D^n)}\leq C_4'(2K)^{\beta}\norm{u_1-u_2}_{H_0^{-\frac{n-1}{2}}(\D^n)}^{1-\beta}.
    \]
    Applying left hand side of inequality \eqref{eq:radon_start_ineq}, we get that
    \[
     \norm{u_1-u_2}_{L^2_0(\D^n)}\leq C_4'(2K)^{\beta} ({C'_1})^{-(1-\beta)}\norm{\mathcal{R}u_1-\mathcal{R}u_2}_{L^2(Z)}^{1-\beta}.
    \]
    Denote that $C_2:=C_4'(2K)^{\beta} (C'_1)^{-(1-\beta)}$, $p:=1-\beta$ and $\omega_2\colon [0,\infty)\to [0,\infty)$, $t\mapsto C_2t^p$. Remaining conditions of quasi-up-increasing mapping are satisfied for $\omega_2$. The claim follows.
\end{proof}

Computational results on Radon transform in the 2D-case are presented in Section \ref{sec:computational}.

\subsection{Inverse conductivity problem}\label{sec:conductivity}
Next, we consider the second example, the inverse conductivity problem. The inverse boundary value problem for the conductivity equation
was originally proposed by Calder\'on \cite{calderon2006inverse} in 1980. The Calder\'on problem and its application to electrical impedance tomography (EIT) have been extensively studied from both analytic and computational perspectives \cite{Borcea}. Recently, the EIT imaging algorithms have been used for stroke classification, see \cite{agnelli2025stroke,sun2025learned,tanyu2023electrical}.

We start with some basic notation.
Let $\Omega$ bounded domain with $C^\infty$ boundary in $\R^n$, $n\geq 3$.
For the conductivity problem, we define first the Dirichlet-to-Neumann map $\Lambda_{\gamma}\colon H^{\frac{1}{2}}(\partial\Omega)\to H^{-\frac{1}{2}}(\partial\Omega)$ on $\Omega$ for  conductivity $\gamma{(x)}$ in $\Omega$.

Let $\gamma\in H^{2+s}(\Omega)$, $s>\frac{n}{2}$. For every voltage potential on
the boundary, $\Phi\in H^\frac{1}{2}(\partial \Omega)$, we denote by $u=u_{\Phi,\gamma}\in H^1(\Omega)$ the (unique) solution to the Dirichlet problem
\[
\left\{ \begin{array}{rl}
\nabla\cdot(\gamma{(x)}\nabla u{(x)})=0 &\text{ in }\Omega, \\
u=\Phi &\text{ on } \partial\Omega.
\end{array}\right.
\]
The \emph{Dirichlet-to-Neumann map} $\Lambda_{\gamma}\colon H^{\frac{1}{2}}(\partial\Omega)\to H^{-\frac{1}{2}}(\partial\Omega)$ is the map 
\[
\Lambda_{\gamma}(\Phi)= \gamma \frac{\partial u_{\Phi,\gamma}}{\partial \upsilon}. 
\]
where $\upsilon$ is the outer normal to $\partial\Omega$.  Notably, the Dirichlet-to-Neumann map is a bounded linear map.

We call mapping \[
\mathcal{F}\colon H^{s+2}(\Omega)\to \mathcal{L}(H^\frac{1}{2}(\partial \Omega),H^{-\frac{1}{2}}(\partial\Omega) ), \gamma\mapsto \Lambda_\gamma\] as \emph{Calder\'on's forward map}. 

Sylvester and Uhlmann \cite{SylvesterUhlmann} proved the unique identifiability of the
conductivity in dimension $d\geq 3$ for isotropic conductivities which
are $C^\infty$-smooth. Analytic reconstruction methods for this problem and the corresponding scattering problem based on the complex geometrical optics are given in \cite{Nachman1988,Novikov1998}.
In $d=2$ dimensions, the first global solution of the inverse conductivity problem is due to Nachman \cite{nachman1996global} for conductivities with two derivatives. In this seminal paper, the $\overline\partial$ technique was used in the study of Calder\'on's inverse problem.
Finally, Astala and P\"aiv\"arinta \cite{astala2006calderon} proved the uniqueness of the inverse 
problem in the form of its original formulation from \cite{calderon2006inverse}, i.e., for general isotropic conductivities in $L^\infty$ which are bounded from below and above by positive constants. For further developments on the uniqueness of the inverse problem, see
\cite{AstalaLassasPaivarinta2016,caro2016global,FerreiraKSU,IdeUhlmannCPAM,
novikov2010global,uhlmann2009electrical}. 

In this work, we are especially interested in stability properties. We denote that for $K>0$ and $s>\frac n2$, the set of measurable and bounded conductivities in $\Omega$ is \[
B^{n,s}(K,\Omega)=\{\gamma\in H^{2+s}(\Omega) \colon K\geq\gamma(x)\geq \frac{1}{K} \text{ for all } x\in \Omega, \norm{\gamma}_{ H^{2+s}(\Omega)}\leq K \}.\] 
The Calderon's forward map restricted to $B^{n,s}(K,\Omega)$ is Lipschitz stable. That is, for $\gamma_1,\gamma_2 \in B^{n,s}(K,\Omega)$ there exists a positive constant $C>0$ such that
\begin{align}
\norm{\Lambda_{\gamma_1}-\Lambda_{\gamma_2}}_{H^{\frac{1}{2}}(\partial\Omega)\to H^{-\frac{1}{2}}(\partial\Omega)} \leq C \norm{\gamma_1-\gamma_2}_{L^2(\Omega)}.   
\end{align}
Here $\norm{\cdot} _{H^{\frac{1}{2}}(\partial\Omega)\to H^{-\frac{1}{2}}(\partial\Omega)}$ is the operator norm for linear operators.
The inverse stability is much more delicate. Classical logarithmic-type stability estimates have
been studied in \cite{alessandrini1988stable,barcelo2001stability,caro2013stability,clop2010stability}. 
Here, we present one version of classical logarithmic-type stability.

\begin{theorem}
    Suppose $n\ge 3$, $s>\frac{n}{2},$ $K>0$ and $\Omega$ bounded domain with $C^\infty$ boundary. Let $\gamma_1,\gamma_2\in B^{n,s}(K,\Omega)$.
  Then, there exist numbers $C>1$  (depending on $K, s , n$, and $\Omega$) and  $0<p<1$ (depending on $n$ and $s$) 
  such that 
    \[
    \norm{\gamma_1-\gamma_2}_{L^\infty(\Omega)}\leq C \frac 1{\big(\ln(1+\norm{\Lambda_{\gamma_1}-\Lambda_{\gamma_2}}^{-1}_{H^{\frac{1}{2}}(\partial\Omega)\to H^{-\frac{1}{2}}(\partial\Omega)})\big)^{p}},
    \]
    when $\norm{\Lambda_{\gamma_1}-\Lambda_{\gamma_2}}_{H^{\frac{1}{2}}(\partial\Omega)\to H^{-\frac{1}{2}}(\partial\Omega)}>0$. 
\end{theorem}

Now, from Lipschitz stability and logarithmic stability, we get the following corollary. 

\begin{corollary}\label{prop:conductivity_quasi-Lip-log}
Let $n\ge 3$, $s>\frac{n}{2},$ $K>0$ and $\Omega$ bounded domain with $C^\infty$ boundary in $\R^n$. Then there exists $C_1,C_2>0$, and  $0<p<1$ such that a Calder\'on's forward map $\mathcal{F}\colon H^{s+2}(\Omega)\to \mathcal{L}(H^\frac{1}{2}(\partial\Omega),H^{-\frac{1}{2}}(\partial\Omega) )$ restricted to $B^{n,s}(K,\Omega)$ is $(t\mapsto C_1t;t\mapsto \omega(t) )$-quasi-up-increasing, where \[
\omega(t)=\begin{cases}
    C_2(\ln(1+t^{-1}))^{-p}, \text{ when } t>0\\
0,   \text{ when } t=0
\end{cases}
.\]
\end{corollary}

Consequences of the above stability results for the inverse problems and some computational
examples are given later in Section \ref{sec:computational}. Next,
we formulate rigorously the objects used in the study of quasi-isometries and persistent homology, and formulate basic results for those to provide tools used in the proof of 
 the main results.

 
\section{Preliminaries on quasi-isometries}\label{sec:quasi}
Quasi-isometry will be one of the key elements when introducing the persistent homology theory in order to observe structures and answer the questions \ref{q1}-\ref{q3}.
A quasi-isometry is a function between two metric spaces.  In this section, we give the definition and related results.
Formally, a quasi-isometry is defined as follows.

\begin{definition}
A mapping $f\colon (X,d_X)\to (Y,d_Y)$ between metric spaces is an \emph{$(L,\varepsilon)$-quasi-isometric embedding} if there exists constants $L\geq 1$ and $\varepsilon\geq 0$ having the property that, for all $x_1,x_2\in X$,
\begin{align*}
\frac{1}{L}d_X(x_1,x_2)-\varepsilon\leq d_Y(f(x_1),f(x_2))\leq Ld_X(x_1,x_2)+\varepsilon.
\end{align*}
\end{definition}

It is immediate from the definition that a quasi-isometric embedding need neither be continuous nor injective. However, we have the following property.
\begin{lemma}
\label{lemma:QI-radius}
Let $f\colon (X,d_X)\to (Y,d_Y)$ be an $(L,\varepsilon)$-quasi-isometric embedding. Then, for each $x\in X$, $f^{-1}(f(x)) \subset \bar B_X(x, L\varepsilon)$.
\end{lemma}
\begin{proof}
Let $x\in X$ and $x'\in f^{-1}(f(x))$. Then
\[
d_X(x,x') \le L\left( d_Y(f(x),f(x')) + \varepsilon\right) = L \varepsilon.
\]
The claim follows.
\end{proof}

A quasi-isometric embedding $X\to Y$, which is coarsely surjective, is called a quasi-isometry. For the definition, recall that the Hausdorff distance of the subsets $A$ and $A'$ in a metric space $(X,d)$ is 
\[
\dist_\Haus(A,A') = \inf\{ r \ge 0 \colon A \subset B_d(A',r) \text{ and } A' \subset B_d(A,r) \},
\]
where $B_d(E,r) = \{ x \in X \colon \dist_d(x,E)<r\}$ is the $r$-neighborhood of a subset $E\subset X$. 

\begin{definition}
A mapping $f \colon X\to Y$ between metric spaces is an \emph{$(L,\varepsilon,\delta)$-quasi-isometry for $L \ge 1$, $\varepsilon\ge 0$, and $\delta \ge 0$} if $f$ is an $(L,\varepsilon)$-quasi-isometric embedding and $\dist_\Haus(fX,Y) \le \delta$.
\end{definition}

\newcommand{\QI}{\mathrm{QI}}
We denote
\[
\QI_{(L,\varepsilon,\delta)}(X,Y) = \{ f \colon X\to Y \colon f \text{ is an $(L,\varepsilon,\delta)$-quasi-isometry}\}.
\]
the family of all $(L,\varepsilon,\delta)$-quasi-isometries.

\subsection{Basic properties of quasi-isometries}

It is well-known that the composition of quasi-isometric embeddings (resp.~quasi-isometries) is a quasi-isometric embedding (resp.~quasi-isometry). We record this as follows.

\begin{lemma}
Let $f \colon (X,d_X)\to (Y,d_Y)$ and $f'\colon (Y,d_Y) \to (Z,d_Z)$ be $(L,\varepsilon)$- and $(L',\varepsilon')$-quasi-isometric embeddings, respectively. Then the composition $f'\circ f\colon X\to Z$ is an $(LL', L'\varepsilon + \varepsilon')$-quasi-isometric embedding. If, in addition, $f\in \QI_{(L,\varepsilon,\delta)}(X,Y)$ and $f'\in \QI_{(L',\varepsilon',\delta')}(Y,Z)$, then $f'\circ f \in \QI_{(LL',L'\varepsilon +\varepsilon',L'\delta+ \varepsilon' + \delta')}(X,Z)$.
\end{lemma}
\begin{proof}
Let $x,y\in X$. Then
\[
d_Z(f'(f(x)),f'(f(y))) \le L' d_Y(f(x),f(y)) + \varepsilon' \le LL' d_X(x,y) + L'\varepsilon + \varepsilon'.
\]
and, similarly, 
\begin{align*}
d_Z(f'(f(x)), f'(f(y)) &\ge \frac{1}{L'} d_Y(f(x),f(y)) - \varepsilon' 
\ge \frac{1}{L'} \left( \frac{1}{L} d_X(x,y) - \varepsilon\right) - \varepsilon' \\
&\ge \frac{1}{LL'} d_X(x,y) - ( L'\varepsilon + \varepsilon').
\end{align*}
Thus $f' \circ f$ is an $(LL', L'\varepsilon + \varepsilon')$ quasi-isometric embedding.

For the second claim, it suffices to show that 
\[
\dist_\Haus((f'\circ f)(X),Z) \le L'\delta +  \varepsilon' + \delta'.
\]
Let $z\in Z$. Since $\dist_\Haus(f'(Y),Z)\le \delta'$, there exists $y \in Y$ for which $d_Z(f'(y),z)\le \delta'$. Since $\dist_\Haus(f(X),Y))\le \delta$, there exists $x\in X$ for which $d_Y(f(x),y)\le \delta$. Thus
\begin{align*}
d_Z(f'(f(x)),z) &\le d_Z(f'(f(x)),f'(y)) + d_Z(f'(y),z) \\
&\le L'd_Y(f(x),y) + \varepsilon' + \delta'
\le L' \delta + \varepsilon' + \delta'.
\end{align*}
The claim follows.
\end{proof}

Each quasi-isometry $X\to Y$ has a quasi-inverse $Y \to X$ in the following sense. We emphasize already at this point that the map $\Phi$, a choice of a quasi-inverse, in the following statement is not unique. 

\begin{proposition}
\label{prop:quasi-inverse}
For $L\ge 1$, $\varepsilon\ge 0$, and $\delta\ge 0$, there exists a map 
\[
\Phi \colon \QI_{(L,\varepsilon,\delta)}(X,Y) \to \QI_{(L,L(\varepsilon+2\delta),L\varepsilon)}(Y,X)
\]
for which $\Phi(f)\circ f \in \QI_{(1,2L\varepsilon,L\varepsilon)}(X,X)$ and $f\circ \Phi(f) \in \QI_{(1,2\delta, \delta)}(Y,Y)$. 
Moreover, $d_X(\Phi(f)\circ f(x),x)\le L\varepsilon$ for each $x\in X$ and $d_Y(f\circ \Phi(f)(y),y)\le \delta$ for each $y\in Y$.
\end{proposition}
\begin{proof}
We define first the map $\Phi$. Let $f\in \QI_{(L,\varepsilon,\delta)}(X,Y)$. We first define a map $g_f \colon Y \to X$ and then show that we may define $\Phi$ by the formula $f\mapsto g_f$.

Let $y\in Y$. Since $\dist_\Haus(fX,Y) \le \delta$, we have that there exists a point $g_f(y)\in X$ for which $d_Y(f(g_f(y)),y))\le \delta$. We show first that $g_f$ is a quasi-isometry.

Let $y,y'\in Y$. Then 
\begin{align*}
d_Y(y,y') &\le d_Y(y,f(g_f(y))) + d_Y(f(g_f(y)), f(g_f(y'))) + d_Y(f(g_f(y')), y') \\
&\le \delta + d_Y(f(g_f(y)), f(g_f(y'))) + \delta \le L d_X(g_f(y),g_f(y')) + \varepsilon + 2\delta.
\end{align*}
Hence 
\[
d_X(g_f(y),g_f(y'))\ge \frac{1}{L} d_Y(y,y') - \frac{(\varepsilon + 2\delta)}{L} \ge \frac{1}{L} d_Y(y,y') - L(\varepsilon + 2\delta).
\]
To the other direction,
\begin{align*}
d_X(g_f(y), g_f(y')) &\le L \left( d_Y(f(g_f(y)), f(g_f(y'))) + \varepsilon\right) \\
&\le L \left( d_Y(f(g_f(y)), y) + d_Y(y,y') + d_Y(y', f(g_f(y'))) + \varepsilon\right) \\
&\le L d_Y(y,y') + L(\varepsilon + 2\delta). 
\end{align*}

Finally, let $x\in X$. Then $g_f(f(x)) \in f^{-1}(f(x))$. Thus, by Lemma \ref{lemma:QI-radius}, $d_X(g_f(f(x),x)\le L \varepsilon$. Hence $\dist_\Haus(g_f(Y),X) \le L\varepsilon$. We have shown that $g_f$ is an $(L,L(\varepsilon+2\delta),L\varepsilon)$-quasi-isometry.

It remains to show that 
we have 
$g_f \circ f \in \QI_{(1,2L\varepsilon,L\varepsilon)}(X,X)$ and $f\circ g_f \in \QI_{(1,2\delta, \delta)}(Y,Y)$. Let $x,x'\in X$. Then, similarly to above,
\begin{align*}
d_X(g_f(f(x)),g_f(f(x'))) &\le d_X(g_f(f(x)), x) + d_X(x,x') + d_X(x',g_f(f(x'))) \\
&\le d_X(x,x') + 2 L \varepsilon
\end{align*}
and
\begin{align*}
d_X(x,x') &\le d_X(x,g_f(f(x))) + d_X(g_f(f(x)),g_f(f(x'))) + d_X(g_f(f(x')), x') \\
&\le d_X(g_f(f(x)), g_f(f(x'))) + 2L \varepsilon. 
\end{align*}
Finally, by Lemma \ref{lemma:QI-radius}, we have, for each $x\in X$, that $d_X(g_f(f(x)), x) \le L \varepsilon$. Thus $g_f\circ f\in \QI_{(1,2L\varepsilon, L \varepsilon)}(X,X)$.
The case of $f\circ g_f \in \QI_{(1,2\delta, \delta)}(Y,Y)$ is similar and left to the interested reader. 

We conclude that we may take $\Phi$ to be the map $f\mapsto g_f$.
\end{proof}

\subsection{Quasi-up-increasing mapping  implies quasi-isometric embedding}

We finish the general discussion on quasi-isometric mappings with an observation that quasi-up-increasing mappings are quasi-isometric embeddings.

\begin{proposition}\label{prop:qui_is_QI}
    Let $F\colon X \to Y$ be $(t\mapsto \omega_1(t)+\varepsilon_1;t\mapsto \omega_2(t)+\varepsilon_2)$-quasi-up-increasing. For $i=1,2$, let $\Tilde{\omega}_i\colon (0,\infty) \to  \R$ be defined by the first derivatives of $\omega_i$ in the following way: 
    \[\Tilde{\omega}_i(r)=\begin{cases}
       D\omega_i(r), \text{ if } D\omega_i(r)\geq 1 \\
       D\omega_i(r)^{-1}, \text{ if } 0<D\omega_i(r)\leq 1 
    \end{cases}.
    \] Then for every $r_1> 0$ and $r_2>0$, the mapping $F\colon X \to Y$ is $(L_{r_1,r_2},\varepsilon_{r_1,r_2})$-quasi-isometric embedding, where \begin{align*}
        L_{r_1,r_2}&=\max\{\Tilde{\omega}_1(r_1),\Tilde{\omega}_2(r_2)\}, \text{ and } \\
        \varepsilon_{r_1,r_2}&=\max\biggl\{-D\omega_1(r_1)r_1+\omega_1(r_1)+\varepsilon_1, \frac{-D\omega_2(r_2)r_2+\omega_2(r_2)+\varepsilon_2}{D\omega_2(r_2)}\biggr\}.
    \end{align*}
\end{proposition}
\begin{proof}
    We can approximate both $\omega_1\colon [0,\infty)\to[0,\infty)$, and $\omega_2\colon [0,\infty)\to[0,\infty)$ up by tangents. That is, for every $r_1,r_2>0$ we get that
    \begin{align*}
    \omega_1(t)+\varepsilon_1&\leq D\omega_1(r_1)t-D\omega_1(r_1)r_1+\omega_1(r_1)+\varepsilon_1 , \text{ and } \\
     \omega_2(t)+\varepsilon_2&\leq D\omega_2(r_2)t-D\omega_2(r_2)r_2+\omega_2(r_2)+\varepsilon_2. 
    \end{align*}
    Since $\omega_1(t),\omega_2(t)\geq 0$ for all $t\in[0,\to\infty)$, we have that $-D\omega_1(r_1)r_1+\omega_1(r_1)+\varepsilon_1>0$ and $-D\omega_2(r_2)r_2+\omega_1(r_2)+\varepsilon_2>0$.
    By the definition of a quasi-up-increasing function and previous estimates we get that 
    \begin{align*}
        d_Y(F(x_1),F(x_2))&\leq D\omega_1(r_1)(d_X(x_1,x_2))-D\omega_1(r_1)r_1+\omega_1(r_1)+\varepsilon_1, 
         \text{ and } \\
     d_X(x_1,x_2)&\leq D\omega_2(r_2) d_Y(F(x_1),F(x_2))-D\omega_2(r_2)r_2+\omega_2(r_2)+\varepsilon_2. 
    \end{align*}
    Furthermore, it implies that 
    \begin{align*}
         \frac{1}{D\omega_2(r_2)}d_X(x_1,x_2)-(\frac{-D\omega_2(r_2)r_2+\omega_2(r_2)+\varepsilon_2}{D\omega_2(r_2)}) \leq d_Y(F(x_1),F(x_2))\\
         \leq D\omega_1(r_1)(d_X(x_1,x_2))-D\omega_1(r_1)r_1+\omega_1(r_1)+\varepsilon_1.
    \end{align*}
    Now by choosing $L_{r_1,r_2}$ and $\varepsilon_{r_1,r_2}$ as in the statement, we obtain that $F\colon X\to Y$ is $(L_{r_1,r_2},\varepsilon_{r_1,r_2})$-quasi-isometric embedding.
\end{proof}

From above Proposition \ref{prop:qui_is_QI}, and
from Proposition \ref{prop:radon_hölder} and Corollary \ref{prop:conductivity_quasi-Lip-log}, we obtained following corollaries.

\begin{corollary}\label{cor:radon_is_QI}
    The Radon transform restricted to a closed ball is a quasi-isometry embedding.
\end{corollary}

\begin{corollary}\label{cor:calderon_is_QI}
        The restricted Calderón's forward map is a quasi-isometry embedding.
\end{corollary}

Furthermore, when having a quasi-up-increasing mapping $F\colon X\to F(X)$, i.e., considering the surjection, we have a quasi-isometry with $\delta=0$. This is reasonable in practical applications, but for the completeness of the theory, the density constant $\delta$ is carried in the upcoming results.


\section{Stability on persistent homology}\label{sec:per_hom_stab}

The persistent homology is a tool to study the structure of data in different scales, revealing the most prominent shapes of the data. 
We will show how quasi-isometry carries the information of the shapes, and furthermore, how to observe these. 
We aim to present a theory such that one can read it without the broad prerequisites of persistent homology. If one would like to get more details, we refer to \cite{Edelsbrunner_shortcourse,Edelsbrunner_survey,Edelsbunner,roadmap_Otter} in the general theory of persistent homology.

\subsection{Rips homology of a metric space}

\newcommand{\diam}{\mathrm{diam}}

Let $(X,d_X)$ be a metric space. In this paper, we consider Rips complexes $\Rips(X,s)$ at scale $s\in\R$ defined as follows.

A map $\sigma \colon \{0,1,\ldots, p\}\to X$ is called a $p$-simplex and, for $i=0,1,\ldots, p$, points $x_i=\sigma(i)$ are called the \emph{vertices of $\sigma$}; we also denote $\sigma = [x_0,\ldots, x_p]$. We denote $\Sigma_p(X)$ the family of all $p$-simplices $\{0,\ldots, p\} \to X$. For $s\ge 0$, we also denote
\[
\Sigma_p(X,s) = \{ [x_0,\ldots, x_p] \in \Sigma_p(X) \colon \diam \{x_0,\ldots, x_p\} \le s\}
\]
the family of all $p$-simplices having diameter at most $s$. Clearly, $\Sigma_p(X,s)=\emptyset$ for $s<0$ and $\Sigma_p(X,0)$ consists of constant maps.

For each $p\in \Z$ and $s\ge 0$, let $\Rips_p(X,s)$ be the free $\Z_2$-vector space having $\Sigma_p(X,s)$ as a basis, that is, $\Rips_p(X,s) = \bigoplus_{\Sigma_p(X,s)} \Z_2$. The boundary maps $\partial_p \colon \Rips_p(X,s) \to \Rips_{p-1}(X,s)$ are given by the linear extension of the usual formula 
\[
\partial_p [x_0,\ldots, x_p] = \sum_{i=0}^p (-1)^i [x_0,\ldots, x_{i-1}, x_{i+1},\ldots, x_p]
\]
for simplices $[x_0,\ldots, x_p]\in \Sigma_p(X,s)$. 

With these boundary maps the sequence $\Rips(X,s) = (\Rips_p(X,s), \partial_p)_{p\in \Z}$ is a complex of vector spaces; see e.g. Rotman \cite{Rotman} for terminology. We call $\Rips(X,s)$ the \emph{Rips complex of $X$ at scale $s$}. The 
\emph{$p$th Rips homology group of $X$ at scale $s$} is 
\[
H_p(\Rips(X,s))=\frac{\ker(\partial_p \colon \Rips_p(X,s) \to \Rips_{p-1}(X,s))}
{\im(\partial_{p+1} \colon \Rips_{p+1}(X,s) \to \Rips_p(X,s))};
\]
recall that $\im \partial_{p+1} \subset \ker \partial_p$, since $\partial_p \circ \partial_{p+1}=0$.

\subsection{Induced maps in persistent homology}
\label{sec:induced-maps}

For $s\le t$, we have the inclusion $\Sigma_p(X,s) \subset \Sigma_p(X,t)$ for simplices and hence $\Rips(X,s) \subset \Rips(X,t)$. Further, inclusion maps $\iota_X^{t,s}\colon \Rips(X,s)\hookrightarrow \Rips(X,t)$ satisfy the composition law $\iota_X^{t,s} = \iota_X^{t,r} \circ \iota_X^{r,s}$ for $s\le r \le t$. These inclusion maps descend to homology and induce linear maps 
\[
\phi_X^{t,s}\colon H_p(\Rips(X,s))\to H_p(\Rips(X,t))
\]
for $s\le t$, which also satisfy the composition law $\phi_X^{t,s}=\phi_X^{t,r}\circ \phi_X^{r,s}$ for all $s\leq r\leq t$.

Let $L\ge 1$ and $\varepsilon\ge 0$. We show that an $(L,\varepsilon)$-quasi-isometric embedding $f\colon X\to Y$ induces a family of linear maps $f_* \colon H_p(\Rips(X,s)) \to H_p(\Rips(Y,Ls+\varepsilon))$. Due to the nature of the statement, we set $\varphi_{(L,\varepsilon)} \colon \R \to \R$ to be the linear map $t\mapsto Lt+\varepsilon$.
The following lemma states that $f\colon X\to Y$ induces, for each $s\ge 0$, the linear map $f_\# \colon \Rips_p(X,s) \to \Rips_p(Y,\varphi_{(L,\varepsilon)}(s))$, satisfying $[x_0,\ldots, x_p] \mapsto [f(x_0),\ldots, f(x_p)]$, is well-defined and 
hence yields 
a linear map $f_*\colon H_p(\Rips(X,s)) \to H_p(\Rips(Y,\varphi_{(L,\varepsilon)}(s)))$.

\begin{lemma}
Let $f\colon X\to Y$ be an $(L,\varepsilon)$-quasi-isometric embedding. Then there exists a well-defined linear map $f_\# \colon \Rips_p(X,s) \to \Rips_p(Y,\varphi_{(L,\varepsilon)}(s))$ satisfying $f_\#[x_0,\ldots, x_p] = [f(x_0),\ldots, f(x_p)]$ for each $[x_0,\ldots, x_p]\in \Sigma_p(X,s)$. In addition, the linear map $f_\#$ descends as a well-defined linear map 
\[
f_* \colon H_p(\Rips(X,s))\to H_p(\Rips(Y,\varphi_{(L,\varepsilon)}(s))), \quad [c] \mapsto [f_\# c].
\]
\end{lemma}
\begin{proof}
The first claim follows immediately from the observation that 
\[
\diam \{ f(x_0),\ldots, f(x_p)\} \le \varphi_{(L,\varepsilon)}(\diam \{ x_0,\ldots, x_p\}) \le \varphi_{(L,\varepsilon)}(s)
\]
and the fact that $\Sigma_p(X,s)$ is a basis of $\Rips_p(X,s)$. The second claim is standard; see e.g.~Rotman \cite{Rotman} for discussion.
\end{proof}

The reader may have already noticed that the induced map of a composition of quasi-isometries may not be formally the same map as the composition of induced maps. The maps $(g_f)_* \circ f_*$ and $(g_f\circ f)_*$, however, agree after suitable post-composition with maps $\phi_X^{t,s}$. For the discussion, we fix the following functions, called \emph{scale functions of $f$}.

Let $\Phi \colon \QI_{(L,\varepsilon,\delta)}(X,Y) \to \QI_{(L,L(\varepsilon+2\delta),L\varepsilon)}(Y,X)$ be a map as in Proposition \ref{prop:quasi-inverse} and for each $f\in \QI_{(L,\varepsilon,\delta)}(X,Y)$ denote $g_f = \Phi(f)$. We denote 
\[
\varphi_f = \varphi_{(L,\varepsilon)} \colon \R\to \R, \ 
\vartheta_f = \varphi_{(L,L(\varepsilon+2\delta))} \colon \R\to \R, \text{ and } \alpha_f = \varphi_{(1,2L\varepsilon)} \colon \R \to \R.
\]
Note that $\vartheta_f(\varphi_f(s)) \ge \alpha_f(s)$ for $s\ge 0$.

\begin{lemma}
\label{lemma:induced-map-composition}
Let $f\in \QI_{(L,\varepsilon,\delta)}(X,Y)$ be a quasi-isometry and let $s\ge 0$. Then 
$(g_f)_* \circ f_* \colon H_p(\Rips(X,s)) \to H_p(\Rips(X,\vartheta_f(\varphi_f(s)))$ and $(g_f \circ f)_* \colon H_p(\Rips(X,s)) \to H_p(\Rips(X,\alpha_f(s)))$. Moreover, 
\[
\phi_X^{\vartheta_f(\varphi_f(s)), \alpha_f(s)} \circ (g_f\circ f)_* = (g_f)_* \circ f_*
\]
as maps $H_p(\Rips(X,s)) \to H_p(\Rips(X,\vartheta_f(\varphi_f(s))))$.
\end{lemma}
\begin{proof}
The first part of the claim follows from Proposition \ref{prop:quasi-inverse}. For the second claim, it suffices to observe that 
\begin{align*}
\iota_X^{\vartheta_f(\varphi_f(s)), \alpha_f(s)}\circ(g_f \circ f)_\# [x_0,\ldots, x_p] 
&= [g_f(f(x_0)), \ldots, g_f(f(x_p))] \\
&=(g_f)_\# f_\# [x_0,\ldots, x_p]
\end{align*}
for $[x_0,\ldots, x_p]\in \Sigma_p(X,s)$.
\end{proof}

Since the composition $g_f\circ f$ has distance $L\varepsilon$ to the identity, we have also that $(g_f \circ f)_*$ agrees, for each $s\ge 0$, with $\phi_X^{\alpha_f(s),s}$. We formalize this as follows. 
\begin{lemma}
\label{lemma:QI-yields-inclusion}
Let $f\in \QI_{(L,\varepsilon,\delta)}(X,Y)$ and $s\ge 0$. Then, for $p\ge 0$,
\[
(g_f \circ f)_* = \phi_X^{\alpha_f(s),s}
\]
as maps $H_p(\Rips(X,s)) \to H_p(\Rips(X,s+2L\varepsilon))$. 
\end{lemma}
\begin{proof}
Indeed, $\alpha_f \colon s \mapsto s+ 2L\varepsilon$. 

We observe first that, for each $[x_0,\ldots, x_p]\in \Sigma_p(X,s)$ and $i \in \{0,\ldots,p\}$, we have that 
\begin{equation}
\label{eq:diameter-change}
\diam \{x_0,\ldots, x_i, (g_f\circ f)(x_i),\ldots, (g_f \circ f)(x_p)\} 
\le \diam \{ x_0, \ldots, x_p\} + 2L\varepsilon.
\end{equation}
This can be verified considering the following three cases.  If $0 \le j < r \le i$, we have that $d_X(x_j,x_r) \le \diam \{ x_0,\dots, x_p\}$. For $0 \le j\le i < r\le p$, we have that $d_X(x_j,(g_f\circ f)(x_r)) \le d_X(x_j,x_r)+d_X(x_r,(g_f\circ f)(x_r)) \le \diam \{x_0,\ldots, x_p\} + L\varepsilon$. Finally, for $i \le j < r\le p$, $d_X((g_f\circ f)(x_i), (g_f\circ f)(x_r)) \le d_X(x_i,x_r) + 2L\varepsilon$. Thus \eqref{eq:diameter-change} holds.

Let now $P \colon \Rips_p(X,s) \to \Rips_{p+1}(X,s+2L\varepsilon)$ be the linear map satisfying 
\[
P[x_0,\ldots, x_p] = \sum_{i=0}^p (-1)^i [x_0,\ldots, x_i, (g_f\circ f)(x_i),\ldots, (g_f\circ f)(x_p)].
\]
By \eqref{eq:diameter-change}, $P$ is well-defined. It is now standard to check that $P$ is a chain homotopy operator from $(g_f \circ f)_\#$ to $\iota_X^{s+2L\varepsilon, s}$; see e.g.~Rotman \cite{Rotman} for terminology and discussion. The claim follows.
\end{proof}

As an immediate consequence of these two observations, we have the following corollary. 
\begin{corollary}
\label{cor:QI-composition}
Let $f\in \QI_{(L,\varepsilon,\delta)}(X,Y)$. Then
\[
(g_f)_* \circ f_* = \phi_X^{\vartheta_f(\varphi_f(s)),s} \colon H_p(\Rips(X,s)) \to 
H_p(\Rips(X,\vartheta_f(\varphi_f(s)))).
\]
In addition,
\[
f_* \circ (g_f)_* = \phi_Y^{\varphi_f(\vartheta_f(s)),s} \colon H_p(\Rips(Y,s)) \to H_p(\Rips(Y,\varphi_f(\vartheta_f(s)))).
\]
\end{corollary}
\begin{proof}
The first claim is a direct application of Lemmas \ref{lemma:induced-map-composition} and \ref{lemma:QI-yields-inclusion}. The second claim follows from the observation that $f$ is the quasi-inverse of $g_f$ and the fact that we may take the function $s\mapsto s+2\delta$ in place of the function $s\mapsto s+2L\varepsilon$ in these lemmas.
\end{proof}

\subsection{Stability of persistent homology under quasi-isometries}

We define two notions of support for non-zero homology classes.

\begin{definition}
The \emph{persistence support $\spt(c)$ of a non-zero element $c\in H_p(\Rips(X,s))$} is the interval
\[
\spt(c)=\{t\geq 0 \colon c\in \im(\phi_X^{s,t}) \}\cup\{t\geq 0 \colon \phi_X^{t,s}(c)\neq0 \}. 
\]
The \emph{total persistence support $\overline{\spt}(c)$ of $c$} is the interval
\[
\overline{\spt}(c) = \bigcup \{ \spt(\phi_X^{t,s}(c)) \colon \phi_X^{t,s}(c)\ne 0\}.
\]
\end{definition}

For the last statement of this section, we introduce our final scale function 
\[
\lambda_{(L,\varepsilon,\delta)} = \varphi_{(L,L(\varepsilon+2\delta))} \circ \varphi_{(L,\varepsilon)} \colon \R\to \R, \quad s\mapsto L^2 s + 2L(\varepsilon + \delta).
\]
Note that, for $f\in \QI_{(L,\varepsilon,\delta)}(X,Y)$, we have that $\lambda_{(L,\varepsilon,\delta)} = \vartheta_f \circ \varphi_f$. 
\begin{definition}
A non-zero class $c\in H_p(\Rips(X,s))$ is $(L,\varepsilon,\delta)$-stable if $\lambda_{(L,\varepsilon,\delta)}(s)<\sup\spt(c)$.
\end{definition}

Images of stable classes with quasi-isometries are non-trivial. We record this fact as the following proposition.

\begin{proposition}
\label{prop:stable-classes2}
Let $f\in \QI_{(L,\varepsilon,\delta)}(X,Y)$ and let $c\in H_p(\Rips(X,s))$ be an $(L,\varepsilon,\delta)$-stable class.  
Then $f_*(c) \ne 0$.
\end{proposition}
\begin{proof}
Since $\vartheta_f \circ \varphi_f = \lambda_{(L,\varepsilon,\delta)}$, we have, by Corollary \ref{cor:QI-composition}, that 
\[
((g_f)_* \circ f_*)(c) = \phi_X^{\lambda_{(L,\varepsilon,\delta)}(s), s}(c)\ne 0.
\]
The claim follows.
\end{proof}

It is not true that images of stable classes under quasi-isometries are stable. However, we have the following estimate for the support of the image.

\begin{theorem}
\label{thm:sup-inf-interval2}
Let $f \colon X\to Y$ be an $(L,\varepsilon,\delta)$-quasi-isometry. Then, for an $(L,\varepsilon,\delta)$-stable element $c\in H_p(\Rips(X,s))$,
\begin{equation}
\label{eq:sup-spt2}
\sup \spt(f_*( c)) \in [\vartheta_f^{-1}(\sup \spt(c)), \varphi_f(\sup \spt(c))]
\end{equation}
and
\begin{equation}
\label{eq:inf-spt2}
\inf \spt(f_*(c)) \in [\vartheta_f^{-1}(\inf \overline \spt(c)), \varphi_f(\inf \spt(c))].
\end{equation}
\end{theorem}

\begin{proof}
Let $g_f = \Phi(f) \in \QI_{(L,L(\varepsilon+2\delta),L\varepsilon)}(Y,X)$ be the fixed quasi-inverse of $f$. 

For the upper bound of \eqref{eq:sup-spt2}, let $t > \sup \spt(c)$. Then 
\[
\phi_Y^{\varphi_f(t),\varphi_f(s)}\left( f_*(c)\right) = f_* \left( \phi_X^{t,s}(c)\right) = f_*(0) = 0.
\]
Thus $\sup \spt(f_* c) \le \varphi_f(t)$. The claim follows.

We move now to the lower bound in \eqref{eq:sup-spt2}.
Suppose that $\sup\spt(f_*(c))<\vartheta_f^{-1}(\sup\spt(c))$. Fix $t'$, such that $\sup\spt(f_*(c))<t'<\vartheta_f^{-1}(\sup\spt(c))$.
Now 
\[
\phi_X^{\vartheta_f(t'),s}(c)=((g_f)_* \circ \phi_Y^{t',\varphi_f(s)})(f_*(c))=(g_f)_*(0)=0. 
\]
Thus $\sup\spt(c)\leq \vartheta_f(t')$ which is contradiction.

For the upper bound in \eqref{eq:inf-spt2} if $s=\inf\spt(c)$, then the claim is clear. Thus, let $s>\inf\spt(c)$. We suppose that $\inf\spt(f_*(c))>\varphi_f(\inf\spt(c))$. Fix $t''$ such that $\inf\spt(c)<t''<\varphi_f^{-1}(\inf\spt(f_*(c)))$. Notably $t''\in\spt(c)$, and $t''<s$, since $\varphi_f^{-1}(\inf\spt(f_*(c)))\leq s$. Let $c'\in H_p(\Rips(X,t''))$  such that $\phi^{s,t''}_X(c')=c$. We have that 
\[
\phi_X^{\varphi_f(s),\varphi_f(t'')}(f_*(c'))=(f_*\circ \phi^{s,t''}_X)(c')=f_*(c).
\]
Thus $\inf\spt(f_*c)\leq \varphi_f(t'')$, which is contradiction.

The lower bound of \eqref{eq:inf-spt2} is also shown by contradiction.
Suppose that $\inf\spt(f_*(c))<\vartheta^{-1}_f(\inf\overline{\spt}(c))$. 
Fix $t'''\in\spt(f_*(c))$ satisfying $\vartheta_f(t''')<\inf\overline{\spt}(c)$. 
Take $c''\in H_p(\Rips(Y,t'''))$ such that $f_*(c)=\phi_Y^{\varphi_f(s),t'''}( c'')$.
Now 
\begin{align*}
(\phi_X^{\vartheta_f(\varphi_f(s)),\vartheta_f(t''')}\circ (g_f)_*)(c'')&=((g_f)_* \circ \phi_Y^{\varphi_f(s),t'''})(c'')\\
&= ((g_f)_*\circ f_*)(c)\\
&=\phi_X^{\vartheta_f(\varphi_f(s)),s}(c)\neq 0.
\end{align*}

Thus $(g_f)_*(c'')$ and $c$ meets, meaning $\inf\overline{\spt}(c)\leq \vartheta_f(t''')$, which is contradiction.
\end{proof}

\begin{remark}
Recall that $H_p(\Rips(X,s))=0$ for $s<0$, and $f_*(c)\neq 0$ for an $(L,\varepsilon,\delta)$-stable element $c$. Thus 
if in Theorem \ref{thm:sup-inf-interval2}, 
$\vartheta_f^{-1}(\inf \spt(c))<0$, we have that 
\[
\inf\spt(f_*(c))\in [0, \varphi_f(\inf \spt(c))].
\] 
\end{remark}

\section{Stability on persistence diagrams}\label{sec:per_dgm_stab}

In the previous section, we studied how homology classes are mapped through the induced map of a quasi-isometry, and we defined so-called stable classes. Now we start to move slightly towards the computational side. Our goal is to give an interpretation of stable classes in terms of a persistence module, which gives us a way to describe the structure of space in different scales, also in an illustrative way. Then we formulate Theorem \ref{thm:sup-inf-interval2} in terms of (indexed) persistent diagrams. 
 
\subsection{Persistence modules}\label{sec:persistence_modules}

Recall that the  persistence module $\mathbb{V}$ is the family of vector spaces \[\{V^s\mid s \in \R\}\] and a double-indexed family of linear maps \[\{v^{t,s}\colon V^s\to V^t \mid s\leq t \}\] for which $v^{t,s}=v^{t,r}\circ v^{r,s}$ for $s\leq r\leq t$ and $v^{t,t}=\id$, see e.g. \cite{Chazal_SSPM_book}. We call the linear maps $v^{t,s}\colon V^s\to V^t$ as persistence maps.

Notably the vector spaces $H_p(\Rips(X,s))$ together with the linear maps $\phi^{t,s}_X\colon H_p(\Rips(X,s))\to  H_p(\Rips(X,t))$ form a persistence module. We denote this module by $\mathbb{H}_p(X)$ and call it more specifically \emph{a homology persistence module of $X$}.

Another class of persistence modules is interval persistence modules.
Let $J\subset\R$ be an interval, and define spaces
\[
I^t_{J}=\begin{cases}
    \Z_2, \text{ if } t\in J \\
    0, \text{ otherwise}
\end{cases}
\]
and linear maps $\textsf{i}^{t,s}_{J}\colon I^s_{J} \to I^t_J$, \[
    \textsf{i}^{t,s}_{J}=\begin{cases}
\id,\text{ if } s,t\in J \\
0, \text{ otherwise}.
\end{cases}
\]
The family of vector spaces $I^t_{J}$ together with maps $\textsf{i}^{t,s}_{J}$ is a persistence module, denoted by $\mathbb{I}_{J}$.

Furthermore, given an indexed family of intervals $\{J_k\mid k\in K\}$ we denote that the direct sum of interval persistence modules, $\bigoplus_{k\in K} \mathbb{I}_{J_k}$, is the vector spaces $\bigoplus_{k\in K} I^t_{J_k}$ together with maps $\bigoplus_{k\in K} \textsf{i}^{t,s}_{J_k}$, where the direct sum of the vector spaces and maps are as usual. Indeed, the direct sum of interval persistence modules (shortly interval modules) is the persistence module.

We start to study the relationship between the homology persistence module and the direct sum of interval persistence modules. For that, we give the definition of the isomorphism of persistence modules.

\begin{definition}
    Two persistence modules $\mathbb{V}$ and $\mathbb{W}$ are said to be isomorphic, denoted by $\mathbb{V}\cong \mathbb{W}$ if for every persistence vector spaces $V^t$ and $W^t$, $t\in \R$, there exists an isomorphism $\tau^t\colon V^t \to W^t$ such that 
    \[
    \tau^t\circ  v^{t,s}= w^{t,s}\circ \tau^s 
    \]
    and
    \[(\tau^{-1})^t\circ w^{t,s} = v^{t,s} \circ (\tau^{-1})^s,
    \] 
    where $v^{t,s}\colon V^s\to V^t$ and $w^{t,s}\colon W^s\to W^t$ are persistence maps for every $s\leq t$. The isomorphism $\tau^t\colon V^t \to W^t$ and its inverse $(\tau^{-1})^t\colon W^t \to V^t$ are called persistence isomorphisms.
\end{definition}

If a homology persistence module $\mathbb{H}_p(X)$ is isomorphic to the interval persistence module $ \bigoplus_{k\in K} \mathbb{I}_{J_k}$, the persistence module $\mathbb{H}_p(X)$ is said to be \emph{decomposable}.
If a homology persistence module is decomposable, then by the Krull-Remak-Schmidt-Azumaya theorem, see e.g. \cite{Chazal_SSPM_book}, the interval persistence module is unique up to indexing.

\begin{theorem}\label{thm:indexing}
    Suppose $\mathbb{H}_p(X)$ is decomposable, and $\mathbb{H}_p(X)\cong \bigoplus_{k\in K} \mathbb{I}_{J_k}$. Furthermore, suppose that $\bigoplus_{k\in K} \mathbb{I}_{J_k}\cong \bigoplus_{m\in M} \mathbb{I}_{J'_m}$. Then there is a bijection $\sigma\colon K\to M$ such that $J_k=J'_{\sigma(k)}$ for all $k\in K$.
\end{theorem}

However, we do not yet have any 
guarantee that $\mathbb{H}_p(X)$ is isomorphic to any persistence interval module. Fortunately, the following theorem provides conditions for the existence of an isomorphism. It is a reformulation of a theorem from \cite[p.30]{oudot:book}, adapted to our setting.
\begin{theorem}\label{thrm:decomposable}
    Let $\mathbb{H}_p(X)$ be a homology persistence module. Then
    $\mathbb{H}_p(X)$ is decomposable if one of the following conditions holds:
    \begin{enumerate}
        \item $H_p(\Rips(X,s))\neq H_p(\Rips(X,t))$ only for finitely many $s,t\in\R$, $s\neq t$,
        \item for every $s\in\R$,  $H_p(\Rips(X,s))$ is finite dimensional, or
        \item for every $s\leq t$, $\rank(\phi_X^{t,s})<\infty$.
    \end{enumerate}
\end{theorem}

\subsection{Persistence natural basis}
We note that, 
for every vector space $\bigoplus_{k\in K}I_{J_k}^t$, there is a natural choice for the basis, which we refer to as \emph{a persistence natural basis},

\begin{align*}
  \mathcal{B}^t= \left\{e^t_k=(a_i)_{i\in K}\ \middle\vert \begin{array}{l}
   a_i=0, \text{ if } i\neq k\\
    a_i=1, \text{ if }i=k, \text{ and } I_{J_k}^t=\Z_2.
  \end{array}\right\}
\end{align*}

We can define persistence support for an element of the interval persistence module similarly to homology classes.
\begin{definition}
    Consider an interval persistence module $\bigoplus_{k\in K}\mathbb{I}_{J_k}$. The persistence support of the vector $v\in \bigoplus_{k\in K}I_{J_k}^t$ is 
    \[
    \spt(v)=\{s\geq0 \mid v\in\im(\bigoplus_{k\in K} \mathsf{i}^{t,s}_{J_{k}}) \}\cup \{s\geq 0 \mid \bigoplus_{k\in K} \mathsf{i}^{s,t}_{J_{k}}(v)\neq 0\}
    \] and the total persistence support of vector $v$ is
    \[
\overline{\spt}(v) = \bigcup \{ \spt(\bigoplus_{k\in K} \mathsf{i}^{s,t}_{J_{k}}(v)) \colon \bigoplus_{k\in K} \mathsf{i}^{s,t}_{J_{k}}(v)\ne 0\}.
\]
\end{definition}
It follows straight from the definition that the support of a natural basis vector is exactly the interval where it exists and is non-zero. We record this observation as a lemma.
\begin{lemma}
      Consider an interval persistence module $\bigoplus_{k\in K}\mathbb{I}_{J_k}$. Let $\mathcal{B}^t$ be the persistence natural basis for the vector space $\bigoplus_{k\in K}I^t_{J_k}$. Then for every $e_k^t\in \mathcal{B}^t$ the persistence support is
      \[
      \spt(e_k^t)=J_k.
      \]
\end{lemma}

The following lemma records that the natural basis elements' persistence support and total persistence support are exactly the same, which is an immediate consequence of the persistence map.

\begin{lemma}{\label{lemma:spt_ospt_same}}
   Consider an interval persistence module $\bigoplus_{k\in K}\mathbb{I}_{J_k}$.  Let $\mathcal{B}^t$ be the natural basis for the vector space $\bigoplus_{k\in K}I^t_{J_k}$. Then for every $e_k^t\in \mathcal{B}^t$ 
\[
\inf\overline{\spt}(e^t_k)=\inf\spt(e^t_k).
\]

\end{lemma}
The support of a general element of $\bigoplus_{k\in K}I_{J_k}^t$ depends on the persistence supports of the persistence natural basis elements. For that, we use the following index notation.
For every $v\in \bigoplus_{k\in K}I_{J_k}^t$, denote 
\[
K_{ \mathcal{B}^t}(v)=\{k \in K\mid v= \bigoplus_{e_k^t\in \mathcal{B}^t}b_ke^t_k, b_k\neq0\}.
\]
\begin{lemma}\label{lemma:gen_elem_spt}
    Consider an interval persistence module $\bigoplus_{k\in K}\mathbb{I}_{J_k}$. Let $\mathcal{B}^t$ be the persistence natural basis for the  vector space $\bigoplus_{k\in K}I^t_{J_k}$. Then for every $v\in\bigoplus_{k\in K}I^t_{J_k} $, the persistence support $\spt(v)$ is an interval $J$ such that 
    \begin{align*}  
    \inf(J)&=\max_{k\in K_{\mathcal{B}^t}(v)}\inf{\spt(e^t_k)}
\end{align*}
and
\begin{align*}
        \sup(J)&=\max_{k\in K_{\mathcal{B}^t}(v)}\sup{\spt(e^t_k)}.
    \end{align*}
\end{lemma}
\begin{proof}
We note that $v\in \im\bigoplus_{k\in K} \mathsf{i}^{t,s}_{J_{k}} $, only if  $e^s_k\in\mathcal{B}^s$ for every $k\in K$ for which $k\in K_{\mathcal{B}^t}(v)$. Thus $\inf(J)=\max_{k\in K_{\mathcal{B}^t}(v)}\inf{\spt(e^t_k)}.$ 
Furthermore $\bigoplus_{k\in K}\mathsf{i}_{J_{k}}^{s,t}(v)\neq 0$ whenever for some $k\in K_{\mathcal{B}^t}(v)$, $\bigoplus_{k\in K}\textsf{i}^{s,t}_{J_k}(e^t_k)\neq 0$. Thus $\sup(J)=\max_{k\in K_{\mathcal{B}^t}(v)}\sup{\spt(e^t_k)}$.
\end{proof}
We finish this subsection with two lemmas, which follow directly from the definition of isomorphism between two persistence modules. 

\begin{lemma}{\label{lemma:decombosable_same_support_1}}
    Suppose $\mathbb{H}_p(X)$ is decomposable, $\mathbb{H}_p(X)\cong \bigoplus_{k\in K} \mathbb{I}_{J_k}$, and $(\tau^{-1})^t\colon \bigoplus_{k\in K}I^t_{J_k}\to H_p(\Rips(X,t))$ is a persistence isomorphism. If $\spt(v)=J$ for $v\in \bigoplus_{k\in K}I^t_{J_k}$, then $\spt((\tau^{-1})^t(v))=J$. Furthermore, if $\overline{\spt}(v)=J'$ for $v\in \bigoplus_{k\in K}I^t_{J_k}$, then $\overline{\spt}((\tau^{-1})^t(v))=J'$.
\end{lemma}

\begin{lemma}{\label{lemma:decombosable_same_support_2}}
      Suppose $\mathbb{H}_p(X)$ is decomposable, $\mathbb{H}_p(X)\cong \bigoplus_{k\in K} \mathbb{I}_{J_k}$, and $\tau^t\colon H_p(\Rips(X,t))\to  \bigoplus_{k\in K}I^t_{J_k}$ is a persistence isomorphism. If $ \spt(c)=J$ for $c\in H_p(\Rips(X,t),$ then $\spt(\tau^t(c))=J$ . 
      Furthermore, if $\overline{\spt}(c)=J'$ for $v\in \bigoplus_{k\in K}I^t_{J_k}$, then $\overline{\spt}(\tau^t(c))=J'$.
\end{lemma}

\subsection{Decomposable persistence modules and quasi-isometry}
Throughout this section, we use the following simplified notation whenever 
$\mathbb{H}_p(X) \cong \bigoplus_{k\in K}\mathbb{I}_{J_k}$, $\mathbb{H}_p(Y) \cong \bigoplus_{l\in L}\mathbb{I}_{J_l}$, and $f \in \QI_{(L,\varepsilon,\delta)}(X, Y)$. Let $g_f = \Phi(f) \in \QI_{(L,L(\varepsilon+2\delta),L\varepsilon)}(Y,X)$ be the fixed quasi-inverse of $f$. We denote that the persistence mappings are 
\begin{align*}
\iX^{t,s}&:=\bigoplus_{k\in K}\mathsf{i}_{J_k}^{t,s} \colon \bigoplus_{k\in K}I^s_{J_k}\to \bigoplus_{k\in K}I^t_{J_k}
\end{align*}
and
\begin{align*}
\iY^{t,s}&:=\bigoplus_{l\in L}\mathsf{i}_{J_l}^{t,s} \colon \bigoplus_{l\in L}I^s_{J_l}\to \bigoplus_{l\in L}I^t_{J_l}.
\end{align*}
Let $\tau_X^t\colon  H_p(\Rips(X,t)) \to \bigoplus_{k\in K}I_{J_k}^t$ and $\tau_Y^t\colon H_p(\Rips(Y,t))\to \bigoplus_{l\in L}I_{J_l}^t$ be the persistence isomorphisms, and $(\tau_X^{-1})^t$, $(\tau_Y^{-1})^t$ their inverses respectively. We denote 
\begin{align*}
    \tildeF&:= \tau_Y^{\varphi_f(t)}\circ f_*\circ (\tau_X^{-1})^t\colon \bigoplus_{k\in K}I^t_{J_k} \to  \bigoplus_{l\in L}I^{\varphi_f(t)}_{J_l}
    \end{align*}
    and
    \begin{align*}
      \tildeGf&:= \tau_X^{\vartheta_f(t)}\circ ({g_f})_*\circ (\tau_Y^{-1})^t\colon  \bigoplus_{l\in L}I^{t}_{J_l} \to  \bigoplus_{k\in K}I^{\vartheta_f(t)}_{J_k}
\end{align*}
the induced homomorphisms in the interval persistence modules.
Moreover, denote that $\mathcal{B}_X^t$ and $\mathcal{B}_Y^t$ are the persistence natural bases with respect to spaces $X$ and $Y$. 

Reader might already notice that the commuting properties of $f_*$ and $(g_f)_*$ are preserved, i.e.
\begin{align} \tildeGf \circ  \tildeF&=\iX^{\vartheta_f(\varphi_f(s)),s}\colon  \bigoplus_{k\in K}I^{s}_{J_k} \to  \bigoplus_{k\in K}I^{\vartheta_f(\varphi_f(s))}_{J_k},
\\
 \tildeF \circ  \tildeGf&=\iY^{\varphi_f(\vartheta_f(s)),s}\colon  \bigoplus_{l\in L}I^{s}_{J_l} \to  \bigoplus_{l\in L}I^{\varphi_f(\vartheta_f(s))}_{J_l}, \\
 \tildeF \circ \iX^{t,s}&=\iY^{\varphi_f(t),\varphi_f(s)}\circ \tildeF\colon \bigoplus_{k\in K}I^{s}_{J_k}\to \bigoplus_{l\in L}I^{\varphi_f(t)}_{J_l}   \text{, and }\\
  \tildeGf \circ \iY^{t,s}&=\iX^{\vartheta_f(t),\vartheta_f(s)}\circ \tildeGf \colon \bigoplus_{l\in L}I^{s}_{J_l} \to  \bigoplus_{k\in K}I^{\vartheta_f(t)}_{J_k}  .
\end{align}
It is natural to define also stable elements for interval persistence families. 

\begin{definition}
   Suppose $\mathbb{H}_p(X)$ is decomposable and $\mathbb{H}_p(X)\cong \bigoplus_{k\in K} \mathbb{I}_{J_k}$. We say that $v\in  \bigoplus_{k\in K}I^{s}_{J_k}$ is $(L,\varepsilon,\delta)$-stable if $\tau_X^s(v) \in H_p(\Rips(X,s))$ is $(L,\varepsilon,\delta)$-stable. 
\end{definition}

The following proposition gives, for quasi-isometric spaces, a correspondence of the stable persistence natural basis elements in the persistence natural bases.

\begin{proposition}\label{lemma:existence_basiselement}
       Suppose $\mathbb{H}_p(X)$ is decomposable, $\mathbb{H}_p(X)\cong \bigoplus_{k\in K} \mathbb{I}_{J_k}$,  and $f \in \QI_{(L,\varepsilon,\delta)}(X, Y)$.
       If $\mathbb{H}_p(Y)$ is decomposable, $\mathbb{H}_p(Y) \cong \bigoplus_{l\in L}\mathbb{I}_{J_l}$, and $e^t_k \in \mathcal{B}^t_X$ is $(L,\varepsilon,\delta)$-stable, then there exists $\overline{e}^{\varphi_f(t)}_l\in \mathcal{B}^{\varphi_f(t)}_Y$ such that \[
       \inf\spt(\overline{e}^{\varphi_f(t)}_l)\in [\vartheta_f^{-1}(\inf\spt(e^t_k)),\varphi_f(\inf\spt(e^t_k))]\]
       and \[
       \sup\spt(\overline{e}^{\varphi_f(t)}_l)\in [\vartheta_f^{-1}(\sup\spt(e^t_k)),\varphi_f(\sup\spt(e^t_k))].\]
       Furthermore, $l\in K_{\mathcal{B}^{\varphi_f(t)}_Y}(\tildeF(e^t_k))$.
\end{proposition}
\begin{proof}
Suppose that 
$e^t_k \in \mathcal{B}^t_X$ is $(L,\varepsilon,\delta)$-stable. By Lemmas \ref{lemma:decombosable_same_support_2} and \ref{lemma:spt_ospt_same} and Theorem \ref{thm:sup-inf-interval2}, 
\begin{align*}
    \inf\spt(\tildeF(e^t_k))\in [\vartheta_f^{-1}(\inf\spt(e^t_k)),\varphi_f(\inf\spt(e^t_k))],
    \end{align*}
and
\begin{align*}    
\sup\spt(\tildeF(e^t_k))\in [\vartheta_f^{-1}(\sup\spt(e^t_k)),\varphi_f(\sup\spt(e^t_k))].\end{align*}

Suppose that $\mathbb{H}_p(Y)\cong \bigoplus_{l\in L}\mathbb{I}_{J_l}$.
Recall that each  $\bigoplus_{l\in L}I^s_{J_l}$, $s\in \R$ is a vector space, and thus \[
v:=\tildeF(e^t_k)=\bigoplus_{\overline{e}^{\varphi_f(t)}_l\in \mathcal{B}_Y^{\varphi_f(t)}}a_l\overline{e}^{\varphi_f(t)}_l,
\]
for finitely many $a_l\neq 0$. 
Also recall that \[
\inf\spt(v)=\max_{l'\in K_{\mathcal{B}_Y^{\varphi_f(t)}}(v)}\inf \spt(\overline{e}_{l'}^{\varphi_f(t)})\]
and 
\[
\sup\spt(v)=\max_{{l'}\in K_{\mathcal{B}_Y^{\varphi_f(t)}}(v)}\sup \spt(\overline{e}_{l'}^{\varphi_f(t)}).\] 

The claim holds if any $\overline{}{e}^{\varphi_f(t)}_{l'}$, $l'\in K_{\mathcal{B}_Y^{\varphi_f(t)}}(v)$ has the property that 
\[
\inf\spt(\overline{e}^{\varphi(t)}_{l'})\in[\vartheta_f^{-1}(\inf\spt(e^t_k)),\varphi_f(\inf\spt(e^t_k))],\]
and
\[
\sup\spt(\overline{e}^{\varphi(t)}_{l'})\in[\vartheta_f^{-1}(\sup\spt(e^t_k)),\varphi_f(\sup\spt(e^t_k))].\]
Also, if $v=\overline{e}_{l'}^{\varphi_f(t)}$, the claim holds by Lemma \ref{lemma:decombosable_same_support_1}.

Thus suppose, that $v=\bigoplus_{l\in L}a_l\overline{e}^{\varphi(t)}_l$, at least for two non-zero coefficients $a_l$. Suppose furthermore that for every basis element 
$\overline{e}^{\varphi_f(t)}_{l'}$, $l'\in K_{\mathcal{B}_Y^{\varphi_f(t)}}(v)$, either the supremum of the support or the infimum of the support does not belong to the desired interval
i.e.,
 \begin{align}\label{assump:inf}
\inf\spt(\overline{e}^{\varphi_f(t)}_{l'})<\vartheta_f^{-1}(\inf\spt(e^t_k))\end{align}
 or
 \begin{align}\label{assump:sup}
   \sup\spt(\overline{e}^{\varphi_f(t)}_{l'})<\vartheta_f^{-1}(\sup\spt(e^t_k)). \end{align}
We will show that this leads to a contradiction.

Denote  
\begin{align*}
    K^1&= \left\{ l'\in K_{\mathcal{B}_Y^{\varphi_f(t)}}(v)\ \middle\vert \begin{array}{l}
    \vartheta_f^{-1}(\sup\spt(e^t_k))\leq\sup\spt(\overline{e}^{\varphi_f(t)}_{l'})\leq\sup\spt(v), \\
    \inf\spt(\overline{e}_{l'}^{\varphi_f(t)})<\vartheta^{-1}_f(\inf\spt(e^t_k))
  \end{array}\right\} \text{ and} \\
K^2&=K_{\mathcal{B}_Y^{\varphi_f(t)}}(v)\setminus K^1.
\end{align*}

By the assumptions \eqref{assump:inf} and \eqref{assump:sup}, and Lemma \ref{lemma:gen_elem_spt} there are at least one index in both $K^1$ and $K^2$. 

Denote that $w^{\varphi_f(t)}=\bigoplus_{l\in K^1}\overline{e}^{\varphi(t)}_l$. Fix $s$ such that
\[
\inf(\spt(w^{\varphi_f(t)}))<s<\vartheta^{-1}_f(\inf\spt(e^t_k)).
\] 
Moreover $\overline{e}^s_{l'}\in \mathcal{B}_Y^s$ for every $l'\in K^1$. Denote that $w^s=\bigoplus_{l\in K^1}\overline{e}^{s}_l$. 

Denote that $z^{\varphi_f(t)}=\bigoplus_{l\in K^2}\overline{e}^{\varphi(t)}_l$. Now $v=w^{\varphi_f(t)}+z^{\varphi_f(t)}$. Also fix $s'$ such that 
\[
\sup\spt(z^{\varphi_f(t)})<s'< \vartheta_f^{-1}(\sup\spt(e^t_k))
.
\]
Notably $\iY^{s',\varphi_f(t)}(z_{\varphi_f(t)})=0.$
Furthermore,
\begin{align*}
v^{s'}&:=\iY^{s',\varphi_f(t)}(v)=\iY^{s',\varphi_f(t)}(w^{\varphi_f(t)})+\iY^{s',\varphi_f(t)}(z^{\varphi_f(t)})\\&=\iY^{s',\varphi_f(t)}(w^{\varphi_f(t)})=\iY^{s',s}(w^{s}).
\end{align*}
 
Recall that $\vartheta_f(s')<\sup\spt(e^t_k)$, meaning that $ \iX^{\vartheta_f(s'),t}(e^t_k)\neq 0$, and
\begin{align*}
    \iX^{\vartheta_f(s'),t}(e^t_k) &=(\tildeGf \circ \iY^{s',\varphi_f(t)}\circ \tildeF)(e^t_k) \\&=(\tildeGf \circ \iY^{s',\varphi_f(t)})(v) =\tildeGf(v^{s'})= (\tildeGf \circ \iY^{s',s})(w^{s})\\
    &=(\iX^{\vartheta_f(s'),\vartheta_f(s)}\circ \tildeGf)(w^s).
\end{align*}
Thus $\inf\spt(e^t_k)\leq \vartheta_f(s)<\inf\spt(e^t_k)$. This is a contradiction.
Thus, the claim holds.

\end{proof}

Before we go to the next proposition, we give the following definition.
\begin{definition}
Suppose that $\mathbb{H}_p(X) \cong \bigoplus_{k\in K}\mathbb{I}_{J_k}$,  and $f \in \QI_{(L,\varepsilon,\delta)}(X, Y)$. We say that natural persistence basis elements $e^t_1,e^t_2\in \mathcal{B}^t_X$, $e^t_1\neq e^t_2$ are separated
if either \[
[\vartheta_f^{-1}(\inf\spt(e^t_1)),\varphi_f(\inf\spt(e^t_1))]\cap [\vartheta_f^{-1}(\inf\spt(e^t_2)),\varphi_f(\inf\spt(e^t_2))]= \emptyset
\]
or
\[
[\vartheta_f^{-1}(\sup\spt(e^t_1)),\varphi_f(\sup\spt(e^t_1))]\cap [\vartheta_f^{-1}(\sup\spt(e^t_2)),\varphi_f(\sup\spt(e^t_2))]= \emptyset
\] holds. Moreover, we say that they are 
coarsely identical if 
\[
[\vartheta_f^{-1}(\inf\spt(e^t_1)),\varphi_f(\inf\spt(e^t_1))]\cap [\vartheta_f^{-1}(\inf\spt(e^t_2)),\varphi_f(\inf\spt(e^t_2))]\neq \emptyset
\]
and
\[
[\vartheta_f^{-1}(\sup\spt(e^t_1)),\varphi_f(\sup\spt(e^t_1))]\cap [\vartheta_f^{-1}(\sup\spt(e^t_2)),\varphi_f(\sup\spt(e^t_2))]\neq \emptyset.
\]
\end{definition}

It is clear from Proposition \ref{lemma:existence_basiselement} that the assignment $e^t_k\mapsto e^{\varphi_f(t)}_l$ is an injection for stable, separated natural persistence basis elements. The same holds for stable, coarsely identical natural persistence basis elements, which we will show in the next proposition. 

\begin{proposition}
    
\label{lemma:two_different_elements}
Suppose $\mathbb{H}_p(X)$ is decomposable, $\mathbb{H}_p(X)\cong \bigoplus_{k\in K} \mathbb{I}_{J_k}$,  and $f \in \QI_{(L,\varepsilon,\delta)}(X, Y)$.      
If $\mathbb{H}_p(Y)$ is decomposable, $\mathbb{H}_p(Y) \cong \bigoplus_{l\in L}\mathbb{I}_{J_l}$, and $e^t_1,e^t_2\in \mathcal{B}^t_X$ are $(L,\varepsilon,\delta)$-stable coarsely identical basis elements,
then there exist $\overline{e}^{\varphi_f(t)}_{1},\overline{e}^{\varphi_f(t)}_{2} \in \mathcal{B}^{\varphi_f(t)}_Y$, $\overline{e}^{\varphi_f(t)}_{1}\neq\overline{e}^{\varphi_f(t)}_{2}$ such that 
\begin{align*}
\inf\spt(\overline{e}^{\varphi_f(t)}_{1})&\in [\vartheta_f^{-1}(\inf\spt(e^t_{1})),\varphi_f(\inf\spt(e^t_{1}))],\\
\sup\spt(\overline{e}^{\varphi_f(t)}_{1})&\in [\vartheta_f^{-1}(\sup\spt(e^t_{1})),\varphi_f(\sup\spt(e^t_{1}))],\\
\inf\spt(\overline{e}^{\varphi_f(t)}_{2})&\in [\vartheta_f^{-1}(\inf\spt(e^t_{2})),\varphi_f(\inf\spt(e^t_{2}))], \text{ and }\\
\sup\spt(\overline{e}^{\varphi_f(t)}_{2})&\in [\vartheta_f^{-1}(\sup\spt(e^t_{2})),\varphi_f(\sup\spt(e^t_{2}))],
 \end{align*}
\end{proposition}

The proof of Proposition \ref{lemma:two_different_elements} is rather elementary and long. It is divided into different cases depending on the natural basis elements $e_1^t$ and $e_2^t$ persistence supports. Since each case follows the same core idea, we present here the main ideas before heading to the actual proof. 

The idea is to study the sum element $e_1^t+e_2^t$, and its image under the map $\tildeF$. Notably, $\tildeF(e_1^t+e_2^t)$ can be written as a sum of natural basis elements of $\mathcal{B}_Y^{\varphi_f(t)}$. We suppose that none of these natural basis elements has the persistence support we want, similarly to what we did in the proof of Proposition \ref{lemma:existence_basiselement}. 
Utilizing the knowledge of $\inf\spt(\tildeF(e_1^t+e_2^t))$ and $\sup\spt( \tildeF(e_1^t+e_2^t))$ respect to $\spt(e_1^t+e_2^t)$, and the knowledge of each natural basis element forming the $\tildeF(e_1^t+e_2^t)$, we will end to the contradiction. Given this heuristic proof, we are ready to present the actual proof.  

\begin{proof} [Proof of Proposition \ref{lemma:two_different_elements}]
Suppose that persistence natural basis elements $e_1^t,e_2^t\in \mathcal{B}_X^t $ are $(L,\varepsilon,\delta)$-stable and coarsely identical. Denote that $a_1=\inf\spt(e^t_1)$, $b_1=\sup\spt(e^t_1)$, and $a_2=\inf\spt(e^t_2)$, $b_2=\sup\spt(e^t_2)$.
By Proposition \ref{lemma:existence_basiselement} we know that for $(L,\varepsilon,\delta)$-stable $e_1^t,e_2^t\in \mathcal{B}_X^t $ there exist persistence basis elements $\overline{e}_{l'}^{\varphi_f(t)},\overline{e}_{l''}^{\varphi_f(t)}\in  \mathcal{B}^{\varphi_f(t)}_Y $ such that 
\begin{align*}
    &\inf\spt(\overline{e}_{l'}^{\varphi_f(t)})\in [\vartheta_f^{-1}(a_1),\varphi_f(a_1)], \\
    &\sup\spt(\overline{e}_{l'}^{\varphi_f(t)}) \in [\vartheta_f^{-1}(b_1),\varphi_f(b_1)], \\
    &\inf\spt(\overline{e}_{l''}^{\varphi_f(t)})\in [\vartheta_f^{-1}(a_2),\varphi_f(a_2)], \text{ and} \\
    &\sup\spt(\overline{e}_{l''}^{\varphi_f(t)}) \in [\vartheta_f^{-1}(b_2),\varphi_f(b_2)].
\end{align*} Furthermore $l'\in K_{ \mathcal{B}^{\varphi_f(t)}_Y }(\Tilde{f}_*(e_1^t))$, and $l''\in K_{ \mathcal{B}^{\varphi_f(t)}_Y }(\Tilde{f}_*(e_2^t))$.
Since intervals defined by $e_1^t,e_2^t$ intersect, we may assume that $l'=l''$. Otherwise, the claim holds.

Since $\Tilde{f_*}(e_1^t)$ is a sum of finitely many persistence basis elements, and $l'\in K_{ \mathcal{B}^{\varphi_f(t)}_Y }(\Tilde{f}_*(e_1^t))$, we write that
\[
\Tilde{f_*}(e_1^t)=\overline{e}_{l'}^{\varphi_f(t)}+\bigoplus_{n\in N}x_n^{\varphi_f(t)},
\] 
where $N\subset L$ is finite, $l'\notin N$, and $x_n^{\varphi_f(t)}\in \mathcal{B}_Y^{\varphi_f(t)}$ for every $n\in N$. Similarly, we write that 
\[
 \Tilde{f_*}(e_2^t)=\overline{e}_{l'}^{\varphi_f(t)}+\bigoplus_{m\in M}x_m^{\varphi_f(t)},\] where $M\subset L$ is finite, $l'\notin M$, and $x_m^{\varphi_f(t)}\in \mathcal{B}_Y^{\varphi_f(t)}$ for every $m\in M$.
Notably $\Tilde{f_*}(e_1^t)\neq \Tilde{f_*}(e_2^t)$, meaning $N\neq M$. Since, $\Tilde{f_*}(e_1^t+e_2^t)\neq0$,
\[
\Tilde{f_*}(e_1^t+e_2^t)=\bigoplus_{n\in N}x_n^{\varphi_f(t)}+ \bigoplus_{m\in M}x_m^{\varphi_f(t)}=\bigoplus_{n\in (N\cup M)\setminus(N\cap M)}x_n^{\varphi_f(t)}\neq0.
\]
We want to show that at least for one basis vector that is $x_n^{\varphi_f(t)}$, $n\in K_{\mathcal{B}_Y^{\varphi_f(t)}}(\tildeF(e_1^t+e_2^t))=(N\cup M)\setminus(N\cap M)$, we have the endpoints of the persistence support of $x_n^{\varphi_f(t)}$ belong to the intervals defined by $e_1^t$ or $e_2^t$. We divide the remaining proof into two parts.

\medskip
\emph{Case 1: $a_1\leq a_2$ and $b_1\leq b_2$}. One can easily verify that \begin{align*}
    &\sup\spt(e^t_1+e^t_2)=b_2,\\
    &\inf \spt(e^t_1+e^t_2)=a_2, \text{ and } \\
    &\inf\overline{\spt}(e^t_1+e^t_2)=a_2.  
\end{align*}
That means by Theorem \ref{thm:sup-inf-interval2} that
\begin{align}
    &\inf \spt(\Tilde{f_*}(e_1^t+e_2^t))\in [\vartheta_f^{-1}(a_2),\varphi_f(a_2)]    
    \end{align}
     and
\begin{align}
    \sup \spt(\Tilde{f_*}(e_1^t+e_2^t))\in [\vartheta_f^{-1}(b_2),\varphi_f(b_2)] \label{eq:sum_sup}.
\end{align}
We suppose that there does not exists natural basis element $x_n^{\varphi_f(t)}\in \mathcal{B}^{\varphi_f(t)}_Y$, for which $n\in K_{\mathcal{B}_Y^{\varphi_f(t)}}(\tildeF(e_1^t+e_2^t))$ has the property that endpoints of the persistence support stay in intervals defined by $e_1^t$ or $e_2^t$. We show that this leads to a contradiction.
Denote that
\begin{align*}
    K^1= \left\{ n \in K_{\mathcal{B}_Y^{\varphi_f(t)}}(\tildeF(e_1^t+e_2^t))\ \middle\vert \begin{array}{l}
    \vartheta_f^{-1}(b_2)\leq\sup\spt(x_n^{\varphi_f(t)})\leq \varphi_f(b_1), \\
    \inf\spt(x_n^{\varphi_f(t)})<\vartheta^{-1}_f(a_1)
  \end{array}\right\},
\end{align*}
\begin{align*}
    K^2= \left\{ n \in K_{\mathcal{B}_Y^{\varphi_f(t)}}(\tildeF(e_1^t+e_2^t))\ \middle\vert \begin{array}{l}
     \varphi_f(b_1)<\sup\spt(x_n^{\varphi_f(t)})\leq \varphi_f(b_2), \\
 \inf\spt(x_n^{\varphi_f(t)})<  \vartheta^{-1}_f(a_2)
 \end{array}\right\},
\end{align*}
and \[
K^3=K_{\mathcal{B}_Y^{\varphi_f(t)}}(\tildeF(e_1^t+e_2^t))\setminus(K^1\cup K^2).\] 
By equation \eqref{eq:sum_sup}, there exists at least one index $m\in K^1\cup K^2$. 
Denote that $w^{\varphi_f(t)}:=\bigoplus_{m\in K^1\cup K^2}x_m^{\varphi_f(t)}$, and $z^{\varphi_f(t)}:=\bigoplus_{m\in K^3}x_m^{\varphi_f(t)}$. That is 
\[
\Tilde{f_*}(e_1^t+e_2^t)=\bigoplus_{n\in  K_{\mathcal{B}_Y^{\varphi_f(t)}}(\tildeF(e_1^t+e_2^t))}x_n^{\varphi_f(t)}=w^{\varphi_f(t)}+ z^{\varphi_f(t)}.
\]
We consider cases $b_1=b_2$ and $b_1<b_2$ separately.

\medskip
\emph{Subcase 1.1: $b_1=b_2$}. Now $K^2=\emptyset$. That means $\inf\spt(w^{\varphi_f(t)})<\vartheta_f^{-1}(a_1)$. That also means that $z^{\varphi_f(t)}\neq 0$, and $\sup\spt(z^{\varphi_f(t)})<\vartheta^{-1}_f(b_1)$. Fix $s$ and $r$ such that $\inf\spt(w^{\varphi_f(t)})<s<\vartheta^{-1}_f(a_1)$ and $\sup\spt(z^{\varphi_f(t)})<r<\vartheta^{-1}_f(b_1)$.  That is $\iY^{r,\varphi_f(t)}(z^{\varphi_f(t)})=0$, and furthermore $\iY^{r,\varphi_f(t)}(w^{\varphi_f(t)}+z^{\varphi_f(t)})=\iY^{r,\varphi_f(t)}(w^{\varphi_f(t)})$.

Notably for every $n\in K^1$, $x_n^s\in \mathcal{B}_Y^s$. Denote that $w^{s}=\bigoplus_{n\in{K^1}} x_n^{s}$. Now $w^{\varphi_f(t)}=\iY^{\varphi_f(t),s}(w^{s})$.
Moreover,
\begin{align*}
(\iX^{\vartheta_f(r),\vartheta_f(s)}\circ\tildeGf)(w^s)&=(\tildeGf\circ\iY^{r,s})(w^s)\\
&=(\tildeGf\circ\iY^{r,\varphi_f(t)}\circ\iY^{\varphi_f(t),s})(w^s)\\
&=(\tildeGf\circ\iY^{r,\varphi_f(t)})(w^{\varphi_f(t)}) \\
&= (\tildeGf\circ\iY^{r,\varphi_f(t)})(w^{\varphi_f(t)}+z^{\varphi_f(t)})\\
&= (\tildeGf\circ\iY^{r,\varphi_f(t)}\circ \Tilde{f}_*) (e_1^t+e_2^t)\\
&=\iX^{\vartheta_f(r),t} (e_1^t+e_2^t)\neq 0.
\end{align*}
That means
$\inf\overline{\spt}(e^t_1+e^t_2)\leq \vartheta_f(s)<a_1$, which is contradiction.  

\medskip
\emph{Subcase 1.2: $b_1<b_2$}. In this case, $\inf(\spt(w^{\varphi_f(t)}))<\vartheta_f^{-1}(a_2)$. Fix $s$ such that  $\inf(\spt(w^{\varphi_f(t)}))<s<\vartheta_f^{-1}(a_2)$. We note that $\sup(\spt(z^{\varphi_f(t)}))<\vartheta_f^{-1}(b_2)$. 
Fix $r$ such that $\max\{\vartheta_f^{-1}(b_1),\sup\spt(z^{\varphi_f(t)})\}<r<\vartheta^{-1}_f(b_2)$. That is $\iY^{r,\varphi_f(t)}(z^{\varphi_f(t)})=0$, and furthermore $\iY^{r,\varphi_f(t)}(w^{\varphi_f(t)}+z^{\varphi_f(t)})=\iY^{r,\varphi_f(t)}(w^{\varphi_f(t)})$.
Again for every $n\in K^1\cup K^2$, $x_n^s\in \mathcal{B}_Y^s$. Denote that $w^{s}=\bigoplus_{n\in{K^1}\cup K^2} x_n^{s}$. $w^{\varphi_f(t)}=\iY^{\varphi_f(t),s}(w^{s})$.
Now \begin{align*}
    (\iX^{\vartheta_f(r),\vartheta_f(s)}\circ\tildeGf)(w^s)=\iX^{\vartheta_f(r),t} (e_1^t+e_2^t)
=\iX^{\vartheta_f(r),t} (e_2^t)\neq 0.
\end{align*}

That means $\inf{\spt}(e^t_2)\leq \vartheta_f(s)<a_2$, which is contradiction.

\medskip
\emph{Case 2: $a_1\leq a_2$ and $b_2< b_1$}. One can now verify that \begin{align*}
    &\sup\spt(e^t_1+e^t_2)=b_1,\\
    &\inf \spt(e^t_1+e^t_2)=a_2, \text{ and } \\
    &\inf\overline{\spt}(e^t_1+e^t_2)=a_1.  
\end{align*}
That means by Theorem \ref{thm:sup-inf-interval2} that
\begin{align}
    &\inf \spt(\Tilde{f_*}(e_1^t+e_2^t))\in [\vartheta_f^{-1}(a_1),\varphi_f(a_2)] \end{align} and \begin{align}
    &\sup \spt(\Tilde{f_*}(e_1^t+e_2^t))\in [\vartheta_f^{-1}(b_1),\varphi_f(b_1)] \label{eq:sum_sup2}.
    \end{align}

Recall that 
$ \Tilde{f_*}(e_1^t)=\overline{e}_{l'}^{\varphi_f(t)}+\bigoplus_{n\in N}x_n^{\varphi_f(t)}$ for finite $N\subset L$, where $x_n^{\varphi_f(t)}\in \mathcal{B}_Y^{\varphi_f(t)}$, and $ \Tilde{f_*}(e_2^t)=\overline{e}_{l'}^{\varphi_f(t)}+\bigoplus_{m\in M}x_m^{\varphi_f(t)}$ for finite $M\subset L$, $x_m^{\varphi_f(t)}\in \mathcal{B}_Y^{\varphi_f(t)}$.

Notably for every $n\in N$, $\inf\spt(x^{\varphi_f(t)}_n)\leq \varphi_f(a_1)$ and for every $m\in M$,  $\sup\spt(x^{\varphi_f(t)}_m)\leq \varphi_f(b_2)$. We denote

\begin{align*}
    K^1= \left\{ n \in  K_{\mathcal{B}_Y^{\varphi_f(t)}}(\tildeF(e_1^t+e_2^t))\ \middle\vert \begin{array}{l}
    \vartheta_f^{-1}(b_1)\leq\sup\spt(x_n^{\varphi_f(t)})\leq \varphi_f(b_1), \\
    \inf\spt(x_n^{\varphi_f(t)})<\vartheta^{-1}_f(a_1)
  \end{array}\right\}
\end{align*}
and $K^2=( K_{\mathcal{B}_Y^{\varphi_f(t)}}(\tildeF(e_1^t+e_2^t)))\setminus K^1$.

By equation \eqref{eq:sum_sup2}, there exists at least one index $m\in K^1$. 
Denote that $w^{\varphi_f(t)}:=\bigoplus_{m\in K^1}x_m^{\varphi_f(t)}$.
Denote that $z^{\varphi_f(t)}:=\bigoplus_{m\in K^2}x_m^{\varphi_f(t)}$. That is 
\[
\Tilde{f_*}(e_1^t+e_2^t)=\bigoplus_{n\in  K_{\mathcal{B}_Y^{\varphi_f(t)}}(\tildeF(e_1^t+e_2^t))}x_n^{\varphi_f(t)}=w^{\varphi_f(t)}+ z^{\varphi_f(t)}.
\]

Notably $\inf(\spt(w^{\varphi_f(t)}))<\vartheta_f^{-1}(a_1)$. Fix $s$ such that  $\inf(\spt(w^{\varphi_f(t)}))<s<\vartheta_f^{-1}(a_1)$. We note that $\sup(\spt(z^{\varphi_f(t)}))<\vartheta_f^{-1}(b_1)$. 
Fix $r$ such that $\max\{\vartheta_f^{-1}(b_2),\sup\spt(z^{\varphi_f(t)})\}<r<\vartheta^{-1}_f(b_1)$. That is $\iY^{r,\varphi_f(t)}(z^{\varphi_f(t)})=0$, and furthermore $\iY^{r,\varphi_f(t)}(w^{\varphi_f(t)}+z^{\varphi_f(t)})=\iY^{r,\varphi_f(t)}(w^{\varphi_f(t)})$.
 For every $n\in K^1$, $x_n^s\in \mathcal{B}_Y^s$. Denote that $w^{s}=\bigoplus_{n\in{K^1}} x_n^{s}$. Notably $w^{\varphi_f(t)}=\iY^{\varphi_f(t),s}(w^{s})$.
Now \begin{align*}
    (\iX^{\vartheta_f(r),\vartheta_f(s)}\circ\tildeGf)(w^s)=\iX^{\vartheta_f(r),t} (e_1^t+e_2^t)
=\iX^{\vartheta_f(r),t} (e_1^t)\neq 0.
\end{align*}

That means $\inf{\spt}(e^t_1)\leq \vartheta_f(s)<a_1$, which is contradiction.
    
\end{proof}

\subsection{Persistence diagram}\label{subsec:persdiagrams}

Now we are ready to define persistence diagrams.

\begin{definition}\label{def:persistence_diagrams}
If $\mathbb{H}_p(X)$ is decomposable, $\mathbb{H}_p(X)\cong \bigoplus_{k\in K} \mathbb{I}_{J_k}$, then the persistence diagram map is $\mdgm_{p,X}\colon K\to \R^2$, $k\mapsto (\inf(J_k),\sup(J_k))$. 
We call the image of the persistence map, $\dgm_p(X):=\im(\mdgm_{p,X}) $, a simplified persistence diagram, and the graph $\Dgm_p(X)=\{(k,\mdgm_{p,X}(k))\colon k\in K\}$ an indexed persistence diagram.
\end{definition}

In this article, indexed persistence diagrams serve an analytical role, whereas simplified persistence diagrams are used for illustration. 
The indexed persistence diagram fully describes the structure of the space $X$ in each scale.
Notably, by Theorem \ref{thm:indexing}, the persistence diagram map and the indexed persistence diagram are unique up to indexing.
Furthermore, each indexed persistence diagram point has a corresponding natural persistence basis element.
Let $(k,(a,b))\in \Dgm_p(X)$, then for every $t\in \R$ satisfying $a<t<b$, there exists the natural persistence basis element $e^t_k\in \mathcal{B}_X^t$, such that $a=\inf(\spt(e^t_k))$ and $b=\sup(\spt(e^t_k))$. Before continuing to the results at the diagram level, let's make a remark about how, in the literature, persistent diagrams are generally defined and the relation to the above definition.

 \begin{remark}
 In the literature, the persistent diagrams of  $\bigoplus_{k\in K}\mathbb{I}_{J_k}$ are the multiset of pairs $(\inf(J_k),\sup(J_k))$, which is essentially the same as an indexed persistence diagram. In the persistence diagrams, $\inf(J_k)$ is often referred to as the birth time, and  $\sup(J_k)$ is the death time. Especially, when $X$ is finite, the intervals are half-open, form $J_k=[a,b)_k$. In more general cases, one can see decorated diagrams, see e.g. \cite{Chazal_SSPM_book}.   

In order to define the bottleneck distance properly, the persistence diagram is extended with diagonal points $(a,a) \in\R^2$ with infinite multiplicity. Denote this persistent diagram of $X$ with all the diagonal points with  infinite multiplicity by
$\overline{\Dgm}_p(X)$.
The bottleneck distance, appearing in \eqref{eq:bottleneck_result} is then
\begin{align}\label{def:bottleneck}
d_{\text{bottleneck}}(\overline{\Dgm}_p(X),\overline{\Dgm}_p(Y))=\inf_{\substack{\text{bijection} \\ \eta \colon \overline{\Dgm}_p(X)\to \overline{\Dgm}_p(Y) }} \sup_{x\in \overline{\Dgm}_p(X)} \norm{x-\eta(x)}_\infty.
\end{align}

 \end{remark}

Next, we define the stable elements of an indexed persistence diagram. 
\begin{definition}
    Suppose that  $\mathbb{H}_p(X) \cong \bigoplus_{k\in K}\mathbb{I}_{J_k}$, we call an indexed diagram point $(k,(a,b))\in \Dgm_p(X)$  a $(L,\varepsilon,\delta)$-stable if $\lambda_{(L,\varepsilon,\delta)}(a)<b$. Furthermore, we call the set 
    \[
    \Dgm_p^{(L,\varepsilon,\delta)}(X):=\{(k,(a,b))\in \Dgm_p(X)\colon \lambda_{(L,\varepsilon,\delta)}(a)<b \} 
    \] 
    the stable indexed persistence diagram of $X$.
\end{definition}

For every $(L,\varepsilon,\delta)$-stable indexed diagram point, one can find a $(L,\varepsilon,\delta)$-stable homology class. Since  $\lambda_{(L,\varepsilon,\delta)}(\varphi^{-1}_f(\vartheta^{-1}_f(s))=s$ this is immediate. We record this in the following lemma.

\begin{lemma}
    Suppose that  $\mathbb{H}_p(X) \cong \bigoplus_{k\in K}\mathbb{I}_{J_k}$. If $(k,(a,b))\in \Dgm^{(L,\varepsilon,\delta)}_p(X)$ then, for every $s$ satisfying $\lambda_{(L,\varepsilon,\delta)}(a)<s<b$, the basis element $e_k^{\varphi^{-1}_f(\vartheta^{-1}_f(s))}\in \mathcal{B}_X^{\varphi^{-1}_f(\vartheta^{-1}_f(s))}$ is an $(L,\varepsilon,\delta)$-stable .  
\end{lemma}

Now there is a corollary from Propositions \ref{lemma:existence_basiselement} and \ref{lemma:two_different_elements}. 
\begin{corollary}\label{cor:persdgm_X_to_Y}
   Suppose that $\mathbb{H}_p(X)$, and  $\mathbb{H}_p(Y)$ are decomposable, and $f\colon X\to Y$ is $(L,\varepsilon,\delta)$-quasi-isometry. There exists an injection
   \[
   \boxdot
_f\colon \Dgm_p^{(L,\varepsilon,\delta)}(X)\to \Dgm_p(Y)
   \] 
   such that for every $\boxdot_f\big((k,(a,b))\big):=(l,(c,d))$ we have that \[(c,d)\in [\vartheta_f^{-1}(a),\varphi_f(a)]\times [\vartheta_f^{-1}(b),\varphi_f(b)].\] 
\end{corollary}

Now, we reverse the roles of the quasi-isometry $f\colon X\to Y$ and its quasi-inverse $g_f\colon Y\to X$. Notably, then $f$ acts as a quasi-inverse of $g_f$, and the scale function $\lambda_{(L,\varepsilon,\delta)}= \vartheta_f\circ \varphi_f$ is reversed, i.e., we consider $\varphi_f \circ \vartheta_f$.  We get the following corollary.

\begin{corollary}\label{cor:persdgm_Y_to_X}
   Suppose that $\mathbb{H}_p(X)$, and  $\mathbb{H}_p(Y)$ are decomposable, and $f\colon X\to Y$ is $(L,\varepsilon,\delta)$-quasi-isometry. There exists an injection
   \[
   \boxdot_{g_f}\colon \Dgm_p^{(L,(L\varepsilon+2\delta),L\varepsilon)}(Y)\to \Dgm_p(X)
   \] 
   such that for every $\boxdot_{g_f}\big((k,(a,b))\big):=(l,(c,d))$ we have that \[(c,d)\in [\varphi_f^{-1}(a),\vartheta_f(a)]\times [\varphi_f^{-1}(b),\vartheta_f(b)].\] 
\end{corollary}

Corollary \ref{cor:persdgm_Y_to_X} says that if $f\colon X\to Y$ is $(L,\varepsilon,\delta)$-quasi-isometry, and $\Dgm_p(Y)$ is known, then one can find the corresponding same dimensional diagram point representing a structure of $X$ for each stable persistence diagram point.

\section{Stability of the model spaces' structures}\label{sec:stab_model}
In Section, \ref{sec:model_inverse_prob}, we presented three questions \ref{q1}-\ref{q3}. Recall that we consider the model 
\[
X_{cont.}\to X_{net}\to Y_{clean} \to Y_{noisy}.
\]

Notably, we have three quasi-isometries:
$f^{(0)}\colon X_{cont.}\to X_{net}$, $x\mapsto x'$, such that $d_X(x,x')\leq \frac{1}{2}\varepsilon_0$, is a $(1,\varepsilon_0,0)$-quasi-isometry, $f^{(1)}\colon X_{net} \to Y_{clean}$ is a $(L,\varepsilon_1,0)$-quasi-isometry, and
$f^{(2)}\colon Y_{clean} \to Y_{noisy}$, $y\mapsto y+E$, $\norm{E}\leq \frac{1}{2}\varepsilon_2$, is a $(1,\varepsilon_2,0)$-quasi-isometry. Moreover, we have the quasi-inverses $g_{f^{(0)}}$, $g_{f^{(1)}}$ and $g_{f^{(2)}}$  respectively. We say that $X_{cont.}\to X_{net}\to Y_{clean} \to Y_{noisy}$ is $(L,\varepsilon_0,\varepsilon_1,\varepsilon_2)$-model. 
Recall $X_{net}\subset X_{cont.}$ is finite in the  $(L,\varepsilon_0,\varepsilon_1,\varepsilon_2)$-model. Since $f^{(1)}$ and $f^{(2)}$ are surjections,  $Y_{clean}$ and $Y_{noisy}$ are finite. By Theorem \ref{thrm:decomposable} persistence modules $\mathbb{H}_p(Y_{noisy})$, $\mathbb{H}_p(Y_{clean})$, $\mathbb{H}_p(X_{net})$ are decomposable, and there exist persistence diagram maps. We are now ready to prove the main theorem of the paper.

\MainTheorem*
\begin{proof}
    Note that $\vartheta(t)=(\vartheta_{f^{(0)}}\circ\vartheta_{f^{(1)}}\circ \vartheta_{f^{(2)}})(t)$ and $\varphi^{-1}(t)=(\varphi^{-1}_{f^{(0)}}\circ\varphi^{-1}_{f^{(1)}}\circ\varphi^{-1}_{f^{(2)}})(t)$. 

    Let $(k,(a,b))\in \Dgm^{(L,\varepsilon_0,\varepsilon_1,\varepsilon_2)}_p(Y_{noisy})$. Since we have \[
    \Dgm^{(L,\varepsilon_0,\varepsilon_1,\varepsilon_2)}_p(Y_{noisy})\subset \Dgm^{(1,\varepsilon_2,\varepsilon_2)}_p(Y_{noisy}),\] by Corollary \ref{cor:persdgm_Y_to_X} there exists an injection 
    \[
    \boxdot_{g_{f^{(2)}}\mid \Dgm^{(L,\varepsilon_0,\varepsilon_1,\varepsilon_2)}_p(Y_{noisy})}\colon \Dgm^{(L,\varepsilon_0,\varepsilon_1,\varepsilon_2)}_p(Y_{noisy})\to \Dgm_p(Y_{clean})
    \] 
    such that for $ \boxdot_{g_{f^{(2)}}\mid \Dgm^{(L,\varepsilon_0,\varepsilon_1,\varepsilon_2)}_p(Y_{noisy})}\big((k,(a,b))\big):=(n,(a',b'))\in \Dgm_p(Y_{clean})$, we have that
    \[(a',b')\in
    [\varphi_{f^{(2)}}^{-1}(a),\vartheta_{f^{(2)}}(a)]\times [\varphi_{f^{(2)}}^{-1}(b),\vartheta_{f^{(2)}}(b)].
    \]
    Then $(n,(a',b'))\in \Dgm_p^{(L,L\varepsilon_1,L\varepsilon_1)}(Y_{clean})$ by Corollary \ref{cor:persdgm_Y_to_X}, there exists an injection  \[
    \boxdot_{g_{f^{(1)}}}\colon \Dgm^{(L,L\varepsilon_1,L\varepsilon_1)}_p(Y_{clean})\to \Dgm_p(X_{net})
    \] 
    such that for $\boxdot_{g_{f^{(1)}}}\big((n,(a',b'))\big):=(m,(a'',b''))\in \Dgm_p(X_{net})$, we have that
    \[(a'',b'')\in[\varphi_{f^{(1)}}^{-1}(\varphi_{f^{(2)}}^{-1}(a)),\vartheta_{f^{(1)}}(\vartheta_{f^{(2)}}(a))]\times [\varphi_{f^{(1)}}^{-1}(\varphi_{f^{(2)}}^{-1}(b)),\vartheta_{f^{(1)}}(\vartheta_{f^{(2)}}(b))].
    \]
    Similarly $(m,(a'',b''))\in \Dgm_p^{(1,\varepsilon_0,\varepsilon_0 )}(X_{net})$, and by Corollary \ref{cor:persdgm_Y_to_X}, there exists an injection  
    \[
    \boxdot_{g_{f^{(0)}}} \colon \Dgm_p^{(1,\varepsilon_0,\varepsilon_0 )}(X_{net}) \to \Dgm_p(X_{cont.}) 
    \]
  such that for $\boxdot_{g_{f^{(0)}}}\big((m,(a'',b''))\big):=(l,(c,d))\in \Dgm_p(X_{cont.})$, we have that
  \[(c,d)\in[\varphi^{-1}(a),  \vartheta(a)  ]\times [\varphi^{-1}(b),  \vartheta(b)].
\]
Since the composition of injections is an injection, we can take \[\boxdot=\boxdot_{g_{f^{(0)}}}\circ \boxdot_{g_{f^{(1)}}}\circ \boxdot_{g_{f^{(2)}}}.\]

\end{proof}

\begin{remark}
    
Each model step increases the uncertainty in the estimate in the persistent diagram, see Figure \ref{fig:estimation_different_areas}. The reader may wonder, why one does not consider straight an $(L, L\varepsilon_0+\varepsilon_1+\varepsilon_2,0)$-quasi-isometry $\tilde{f}=f^{(0)}\circ f^{(1)}\circ f^{(2)}\colon X_{cont.} \to Y_{noisy}.$ The reason is that function $\vartheta_{\tilde{f}}(t)=Lt+L(L\varepsilon_0+\varepsilon_1+\varepsilon_2)$ leads to worse estimate than $\vartheta(t)$ in the main theorem.  
\end{remark}

\begin{remark}    
We want to emphasize that for $L\neq 1$, the so-called lifetime of a structure, i.e., the difference between death and birth times, $b-a$, for $(k,(a,b))\in\Dgm_p(Y_{noisy})$, is not a factor when applying the main theorem. This can be seen in Figure \ref{fig:estimation_different_areas}, where there is indexed diagram points $(1,(1,5))$ and $(2,(3,7))$ and $(3,(3,7))$. For each of these, the lifetime is $4$, but only from the first one, $(1,(1,5))$, we can derive information about the space $X_{cont.}$. Structures appearing on a larger scale must exist longer than those appearing on a smaller scale.   
\end{remark}

\begin{figure}[!htb]
\begin{tikzpicture}[scale=0.8]
    \draw[step=1cm, gray!30, very thin] (0,0) grid (5,8); 
    \draw[color=black, fill=green!60!black!40] (0.5800-0.1,3.2467-0.1) rectangle (1.6950+0.1,7.6950+0.1); 
    \draw[color=black, fill=orange!80!red!20] (0.5800,3.2467) rectangle (1.6950,7.6950); 
    \draw[color=black, fill=blue!60!black!40] (1-0.12,5-0.12) rectangle (1+0.12,5+0.12); 
    \draw[->] (0,0) -- (5.5,0); 
    \draw[->] (0,0) -- (0,8.5); 
    \node[below] at (2.5,-0.8) {\small{birth}}; 
    \node[left,rotate=90] at (-1,5) {\small{death}}; 
    \foreach \x in {0,1,2,3,4,5}
        \draw (\x,0) -- (\x,-0.2) node[below] {\tiny{\x}};
    \foreach \y in {0,1,2,3,4,5,6,7,8}
        \draw (0,\y) -- (-0.2,\y) node[left] {\tiny{\y}};
    \draw[-](0,0)--(5,5); 
    \draw[dashed,thick,color=black](0,0.7225)--(3,7.4725); 
    \draw[dashed,thick,color=black](0,0.7225)--(3,7.4725); 
    \node[draw, fill=red, minimum size=1mm, inner sep=0pt, shape=rectangle] at (1,5) {}; 
    \node[draw, fill=red, minimum size=1mm, inner sep=0pt, shape=rectangle] at (3,7) {};
    \node[draw, fill=red, minimum size=1mm, inner sep=0pt, shape=rectangle] at (1.2,1.4) {};
    \node[draw, fill=red, minimum size=1mm, inner sep=0pt, shape=rectangle] at (1.5,1.8) {};
    \node[draw, fill=red, minimum size=1mm, inner sep=0pt, shape=rectangle] at (5.5,4){};
    \node[right] at (5.5,4) {\tiny{Known $\dgm_p(Y_{noisy})$}}; %
    \node[draw, fill=orange!80!red!20, minimum size=1.5mm, inner sep=0pt, shape=rectangle] at (5.5,3){};
    \node[right] at (5.5,3) {\tiny{Estimation box of $\dgm_p(X_{net})$}}; %
    \node[draw, fill=blue!60!black!40, minimum size=1.5mm, inner sep=0pt, shape=rectangle] at (5.5,3.5){};
    \node[right] at (5.5,3.5) {\tiny{Estimation box of $\dgm_p(Y_{clean})$}};
     \node[draw, fill=green!60!black!40, minimum size=1.5mm, inner sep=0pt, shape=rectangle] at (5.5,2.5){};
    \node[right] at (5.5,2.5) {\tiny{Estimation box of $\dgm_p(X_{cont.})$}};
    \draw[dashed,thick,color=black] (5.3,2)--(5.6,2) {};
    \node[right] at (5.5,2) {\tiny{Threshold}};
    \end{tikzpicture}
\caption{Illustration how each $(1.1,0.1,0.01,0.12)$-model quasi-isometry effect to the estimation box of $X_{cont.}$ Let the known indexed persistence diagram of $Y_{noisy}$ be $
\Dgm_p(Y_{noisy})=
\big\{\big(1,(1,5)\big), \big(2,(3,7)\big), \big(3,(3,7)\big), \big(4,(1.2,1.4)\big) , \big(5,(1.5,1.8)\big)\big\}$
The simplified persistence diagram $\dgm_p(Y_{noisy})$ is illustrated, where red markers are the diagram points. Above dashed line (threshold) is the point $(1,(1,5))\in \Dgm_p(Y_{noisy})$ for which $\vartheta(1)<\varphi^{-1}(5)$, i.e., $(1,(1,5))\in \Dgm_p^{(1.1,0.1,0.01,0.12)}(Y_{noisy})$. The blue square is from $(1,0.12,0)$-quasi-isometry from $Y_{noisy}$ to $Y_{clean}$, the light orange rectangle is from $(1.1,0.01,0)$-quasi-isometry from $Y_{clean}$ to $X_{net}$, and the previous estimation. The light green rectangle is from $(1,0.1,0)$-quasi-isometry from $X_{net}$ to $X_{cont.}$, and the previous estimation. } \label{fig:estimation_different_areas}%

\end{figure}
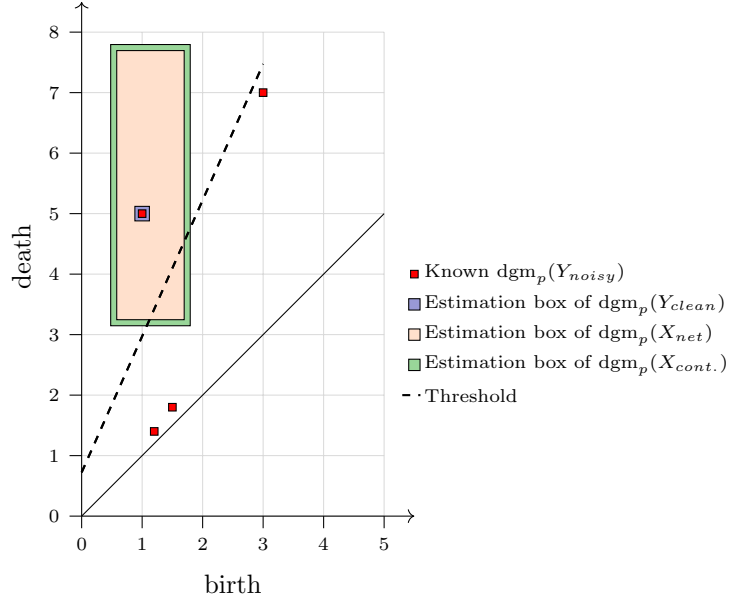

\clearpage

\section{Computational results} \label{sec:computational}
\subsection{Examples on imaging: A computational example with $(4,0.03,0)$-quasi-isometry}
Here we simulate case, where $X_{cont.}\to X_{net} \to Y_{clean}\to Y_{noisy}$ is a $(4,0,0.03,0)$-model. That is we have mappings $\text{id}\colon X_{cont.} \to X_{net}$,  $\text{id}\colon Y_{clean}\to Y_{noisy}$, and $F\colon X_{net}\to Y_{clean}$ is $(4,0.03,0)$-quasi-isometry.

 We created a set $X_{cont.}\subset \R^2$ consisting of points on four disjoint circles with slightly different radii, see Figure \ref{fig:ex1_data}. The radii were 0.9, 0.95, 1, and 1.2, and the center points were $(0,0)$, $(0,-4)$, $(3,-1)$, and $(4,-4)$ correspondingly. Each circle contained 700 evenly distributed points, meaning in total there are 2800 data points in a set $X_{cont.}$. Let simply denote that $X:=X_{cont.}=X_{net}$. We denote an Euclidean distance matrix of $X$ by $\mathsf{D}_X$, $\mathsf{D}_X\in \R^{2800\times2800}$.

Furthermore, we created a distance matrix $\mathsf{D}_Y\in\R^{2800\times2800}$ that adopts the quasi-isometry inequality of $F$. This is done in the following way. We first created a working matrix $\mathsf{D}\in \R^{2800\times2800}$. Each diagonal element is zero, and each non-diagonal element of $\mathsf{D}$ corresponds to an element of $\mathsf{D}_X$ that was first multiplied by a random number between 1.5 and 4, and then a random number between 0.015 and 0.03 was added. That is
\[
\mathsf{D}[i,j]=r_1\mathsf{D}_X[i,j]+r_2,
\]
$r_1\in[1.5,4]$ and $r_2\in[0.015,0.03]$, $i,j=1,\dots,2800$, when $i\neq j$, and $\mathsf{D}[i,i]=0$.
We made sure that a final distance matrix was symmetric, and it respected the triangle inequality. This was done by running $\mathsf{D}$ through the Floyd-Warshall shortest-path algorithm \cite{algorithms}, yielding a matrix $\mathsf{D}_Y$. The matrix $\mathsf{D}_Y$ corresponds to the Euclidean distance matrix of clean (and noisy) measurement data $Y_{clean}=Y_{noisy}$. Let denote that $Y:=Y_{clean}=Y_{noisy}$. Especially, we have that
\[
1.5\mathsf{D}_X[i,j]+0.015\leq \mathsf{D}_Y[i,j]\leq4\mathsf{D}_X[i,j]+0.03.
\] 
So we had $(4,0.03,0)$-quasi-isometry from $X$ to $Y$. The data $Y$ is illustrated by using multidimensional scaling. The overall shape of the data $X$ and $Y$ is similar (up to rotation). However, the scale is different, and we can notice noisiness on the circles'  boundaries (Figures \ref{fig:ex1_dataX} and \ref{fig:ex1_dataY}).

We computed persistent homology of dimension one using Rips complexes from distance matrices $\mathsf{D}_X$ and $\mathsf{D}_Y$. The computations were done using the Ripser package for Python \cite{ctralie2018ripser}. The simplified persistence diagrams $\dgm_1(X)$ and $\dgm_1(Y)$ are shown in the Figures \ref{fig:ex1_phX} and \ref{fig:ex1_phY}. It can be seen that both diagrams have four, clearly off the diagonal, points near the birth time zero, each representing one circle on the data. Both diagrams also have one clearly off the diagonal point, with a greater birth time, representing that the four circles are positioned circularly. The indexed persistence diagram $\Dgm_1(X)$ had total six points, while $\Dgm_1(Y)$ had 523 points. Four points of $\Dgm_1(Y)$ are $(4,0.03,0)$-stable, that is for $(k,(a,b))\in\Dgm_1(Y)$, $k=1,2,3,4$, we have that $L^2a+L^2\varepsilon+\varepsilon=16a+0.51<b$. 
By using Theorem \ref{main-thm-intro} we get the estimation boxes $B_k$ defined by $(k,(a,b))$, $k=1,2,3,4$, of form \[
B_k=[\frac{a-0.03}{4},4a+0.12]\times[\frac{b-0.03}{4},4b+0.12].\] These estimation boxes are illustrated in Figure \ref{fig:synth_diagrams_estmation_boxes}. Furthermore, Theorem \ref{main-thm-intro} guarantees that in each of these boxes there exists an indexed persistence diagram point of $\Dgm_1(X)$. By studying the true indexed persistence diagram $\Dgm_1(X)$, the result can be verified.  

\begin{figure}[!h]
\centering
     \begin{subfigure}[b]{0.49\textwidth}
         \centering
         \begin{tikzpicture}[scale=0.7]
            \begin{axis}[
                 name=full,
                  axis equal,
                ]
                \addplot[
                  only marks,
                  mark=*,
                  mark size=0.1pt,
                  black,
                ]
                table {datafiles/circles_comp_example.dat};
                \draw[red, thin](axis cs:3.8,-1) rectangle (axis cs:4.1,-0.850);
                 \end{axis}
                \begin{axis}[
                  axis equal,
                  width=4cm,
                  at={(full.east)},
                  xshift=-1cm,
                 yshift=0.5cm,
                 restrict x to domain=3.8:4.1,
                  restrict y to domain=-1:-0.85,
                    axis background/.style={fill=white},
                   xtick=\empty,
                ytick=\empty,
                red
                ]
                \addplot[
                  only marks,
                  mark=*,
                  mark size=1pt,
                  black,
                ] table {datafiles/circles_comp_example.dat};
            \end{axis}
        \end{tikzpicture}
         \caption{Data $X$}
         \label{fig:ex1_dataX}
     \end{subfigure}
     \hfill
     \begin{subfigure}[b]{0.49\textwidth}
         \centering
         \begin{tikzpicture}[scale=0.7]
            \begin{axis}[
                 name=full,
                  axis equal,
                ]
                \addplot[
                  only marks,
                  mark=*,
                  mark size=0.1pt,
                  black,
                ]
                table {datafiles/mds_noisy_circles.dat};
                \draw[red, thin](axis cs:3,-3.70) rectangle (axis cs:3.7,-3.50);
                 \end{axis}
                \begin{axis}[
                  axis equal,
                  width=4cm,
                  at={(full.east)},
                  xshift=-1cm,
                 yshift=-2cm,
                 restrict x to domain=3:3.7,
                  restrict y to domain=-3.70:-3.50,
                    axis background/.style={fill=white},
                   xtick=\empty,
                ytick=\empty,
                red
                ]
                \addplot[
                  only marks,
                  mark=*,
                  mark size=1pt,
                  black,
                ] table {datafiles/mds_noisy_circles.dat};
            \end{axis}
        \end{tikzpicture}
         \caption{Data Y illustrated using MDS}
         \label{fig:ex1_dataY}
     \end{subfigure}
     \begin{subfigure}[b]{0.49\textwidth}
         \centering
        \begin{tikzpicture}[xscale=1.5,yscale=1.5]
                \draw[step=1, gray!30, very thin] (0,0) grid (3,3);
                 \node[below] at (2,-0.15) {\small{birth}}; 
                 \node[above,rotate=90] at (-0.15,1.5) {\small{death}}; 
                \foreach \x in {0,1,2,3}
                    \draw (\x,0) -- (\x,-0.03) node[below] {\tiny{\x}};
                \foreach \y in {0,1,2,3}
                    \draw (0,\y) -- (-0.02,\y) node[left] {\tiny{\y}};
                \draw[->](0,0)--(0,3.3); 
                \draw[->](0,0)--(3.3,0); 
                \draw[-](0,0)--(3,3); 
              \pgfplotstableread[col sep=space]{datafiles/X_diagram_data.dat}\datatable
              \foreach \i in {0,...,5} {
                \pgfplotstablegetelem{\i}{[index]0}\of\datatable
                \let\x\pgfplotsretval
                \pgfplotstablegetelem{\i}{[index]1}\of\datatable
                \let\y\pgfplotsretval
                    \node[draw, fill=green, minimum size=1.5mm, inner sep=0pt, shape=diamond] at (\x,\y){};
              }
             \node[above] at (1.5,3) {$\dgm_1(X)$}; %
            
        \end{tikzpicture}
         \caption{Simplified persistence diagram of data $X$ in dimension 1}
         \label{fig:ex1_phX}
     \end{subfigure}
     \hfill
     \begin{subfigure}[b]{0.49\textwidth}
         \centering
                 \begin{tikzpicture}[scale=1.15]
                \draw[step=1, gray!30, very thin] (0,0) grid (4,4);
                 \node[below] at (2,-0.15) {\small{birth}}; 
                 \node[above,rotate=90] at (-0.2,2) {\small{death}}; 
                \foreach \x in {0,1,2,3,4}
                    \draw (\x,0) -- (\x,-0.03) node[below] {\tiny{\x}};
                \foreach \y in {0,1,2,3,4}
                    \draw (0,\y) -- (-0.02,\y) node[left] {\tiny{\y}};
                \draw[->](0,0)--(0,4.3); 
                \draw[->](0,0)--(4.3,0); 
                \draw[-](0,0)--(4,4); 
              \pgfplotstableread[col sep=space]{datafiles/Y_diagram_data.dat}\datatable
              \foreach \i in {0,...,522} {
                \pgfplotstablegetelem{\i}{[index]0}\of\datatable
                \let\x\pgfplotsretval
                \pgfplotstablegetelem{\i}{[index]1}\of\datatable
                \let\y\pgfplotsretval
                    \node[draw, fill=red, minimum size=1mm, inner sep=0pt, shape=rectangle] at (\x,\y){};
              }
             \node[above] at (2,4) {$\dgm_1(Y)$}; %
            
        \end{tikzpicture}
         \caption{Simplified persistence diagram of data $Y$ in dimension 1}
         \label{fig:ex1_phY}
     \end{subfigure}
        \caption{The data $X$ consists of four disjoint circles, each having 700 evenly distributed points. The data $Y$ distance matrix $\mathsf{D}_Y$ is created from the distance matrix of $X$, $\mathsf{D}_X$, such that the quasi-isometry inequality and metric properties hold. The data $Y$ is illustrated using multidimensional scaling (MDS), where $\mathsf{D}_Y$ acts as a dissimilarity matrix. The indexed persistence diagrams are computed from $\mathsf{D}_X$ and $\mathsf{D}_Y$.}
        \label{fig:ex1_data}
\end{figure}
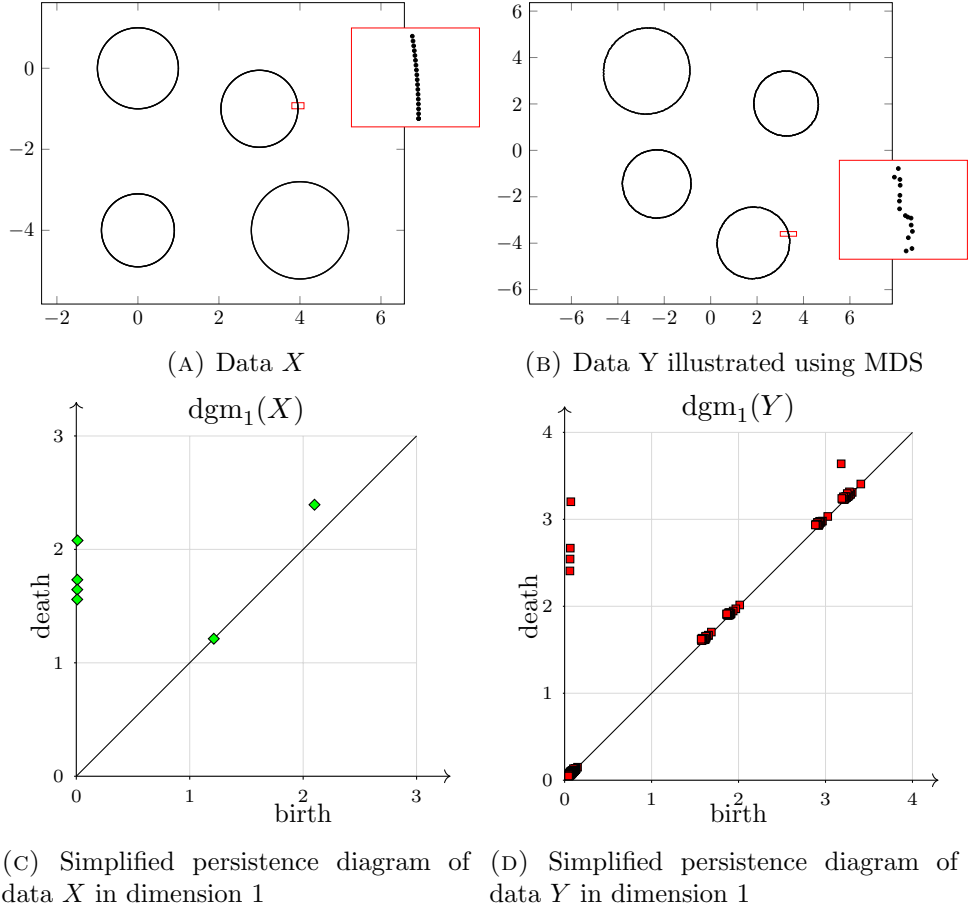

\begin{figure}[tbh!]
    \centering
\begin{tikzpicture}[]
\begin{scope}[xscale=2.5,yscale=1.5]
    \draw[step=1, gray!30, very thin] (0,0) grid (4,4);
  \pgfplotstableread[col sep=space]{datafiles/estbox.dat}\datatable
  \foreach \i in {0,1,2,3} {
    \pgfplotstablegetelem{\i}{[index]0}\of\datatable
    \let\x\pgfplotsretval
    \pgfplotstablegetelem{\i}{[index]1}\of\datatable
    \let\y\pgfplotsretval
    \pgfplotstablegetelem{\i}{[index]2}\of\datatable
    \let\u\pgfplotsretval
    \pgfplotstablegetelem{\i}{[index]3}\of\datatable
    \let\v\pgfplotsretval
     \fill[color=gray, opacity=0.2] (\x,\y) rectangle (\u,{min(4,\v)});
     \draw[dashed,thick,color=gray] (\x,4)--(\u,4);
  }
     \node[below] at (2,-0.15) {\small{birth}}; 
     \node[above,rotate=90] at (-0.15,2) {\small{death}}; 
    \foreach \x in {0,1,2,3,4}
        \draw (\x,0) -- (\x,-0.03) node[below] {\tiny{\x}};
    \foreach \y in {0,1,2,3,4}
        \draw (0,\y) -- (-0.02,\y) node[left] {\tiny{\y}};
    \draw[->](0,0)--(0,4.3); 
    \draw[->](0,0)--(4.3,0); 
    \draw[-](0,0)--(4,4); 
    \draw[dashed,thick,color=black](0,0.51)--(0.218125,4); %

  \pgfplotstableread[col sep=space]{datafiles/Y_diagram_data.dat}\datatable
  \foreach \i in {0,...,522} {
    \pgfplotstablegetelem{\i}{[index]0}\of\datatable
    \let\x\pgfplotsretval
    \pgfplotstablegetelem{\i}{[index]1}\of\datatable
    \let\y\pgfplotsretval
        \node[draw, fill=red, minimum size=1mm, inner sep=0pt, shape=rectangle] at (\x,\y){};
  }
  \pgfplotstableread[col sep=space]{datafiles/X_diagram_data.dat}\datatable
  \foreach \i in {0,...,5} {
    \pgfplotstablegetelem{\i}{[index]0}\of\datatable
    \let\x\pgfplotsretval
    \pgfplotstablegetelem{\i}{[index]1}\of\datatable
    \let\y\pgfplotsretval
        \node[draw, fill=green, minimum size=1.5mm, inner sep=0pt, shape=diamond] at (\x,\y){};
  }
 \node[above] at (2,4) {Simplified persistence diagrams with estimation boxes}; %
\node[draw, fill=red, minimum size=1mm, inner sep=0pt, shape=rectangle] at (4.2,2){};
\node[right] at (4.2,2) {\small{$\dgm_1(Y)$}}; %
\node[draw, fill=green, minimum size=1.5mm, inner sep=0pt, shape=diamond] at (4.2,1.75){};
\node[right] at (4.2,1.75) {\small{$\dgm_1(X)$}}; %
\node[draw, fill=gray, minimum size=1.5mm, inner sep=0pt, shape=rectangle,opacity=0.5] at (4.2,1.5){};
\node[right] at (4.2,1.5) {\small{Estimation box}}; %
\draw[dashed,thick,color=black] (4.05,1.25)--(4.2,1.25) {};
\node[right] at (4.2,1.25) {\small{Threshold}};
\end{scope}
\begin{scope}[yscale=10,xscale=800,shift={(-0.005,-2.3)}]
  \pgfplotstableread[col sep=space]{datafiles/estbox.dat}\datatable
  \foreach \i in {0,1,2,3} {
    \pgfplotstablegetelem{\i}{[index]0}\of\datatable
    \let\x\pgfplotsretval
    \pgfplotstablegetelem{\i}{[index]1}\of\datatable
    \let\y\pgfplotsretval
    \pgfplotstablegetelem{\i}{[index]2}\of\datatable
    \let\u\pgfplotsretval
    \pgfplotstablegetelem{\i}{[index]3}\of\datatable
    \let\v\pgfplotsretval
     \fill[color=gray, opacity=0.2] (\x,{max(1.5,\y)}) rectangle ({min(0.015,\u},{min(2.1,\v)});
     }
    \draw[dashed,thick,color=gray] (0.007702371403574944,2.1)--(0.015,2.1);
    \draw[dashed,thick,color=gray] (0.015,1.5)--(0.015,2.1);
  \pgfplotstableread[col sep=space]{datafiles/X_diagram_data.dat}\datatable
  \foreach \i in {2,...,5} {
    \pgfplotstablegetelem{\i}{[index]0}\of\datatable
    \let\x\pgfplotsretval
    \pgfplotstablegetelem{\i}{[index]1}\of\datatable
    \let\y\pgfplotsretval
        \node[draw, fill=green, minimum size=1.5mm, inner sep=0pt, shape=diamond] at (\x,\y){};
  }
    \draw[->](0.005,1.5)--(0.005,2.15); 
    \draw[->](0.005,1.5)--(0.016,1.5); 
    \foreach \x in {0.005,0.0075,0.01,0.015}
        \draw (\x,1.5) -- (\x,1.5-0.02) node[below] {\tiny{\x}};
    \foreach \y in {1.5,1.6,1.7,1.8,1.9,2,2.1}
        \draw (0.005,\y) -- (0.005-0.0001,\y) node[left] {\tiny{\y}};
     \node[above] at (0.01,2.15) {\large{Zoomed simplified persistence diagrams with estimation boxes}}; %
        \node[below] at (0.01,1.45) {\small{birth}}; 
     \node[above,rotate=90] at (0.0043,1.8) {\small{death}}; 
\end{scope}
\end{tikzpicture}
\caption{We have a $(4,0.03,0)$-quasi-isometry between finite space $X\subset \R^2$ and measurements $Y$. The simplified persistence diagrams of $X$ and $Y$ are presented with green diamonds and red squares, respectively. The dashed line represents the threshold $d=\varphi(\vartheta(b))=16b+0.51$; the four points of $\dgm_1(Y)$  above the threshold are $(4,0.03,0)$-stable. By Theorem \ref{main-thm-intro}, each stable diagram point defines an estimation box (gray boxes). Inside each box, there is a diagram point of $\dgm_1(X)$. The zoomed part confirms that there are indeed points of $\dgm_1(X)$ situated within these boxes.  }
\label{fig:synth_diagrams_estmation_boxes}
\end{figure}
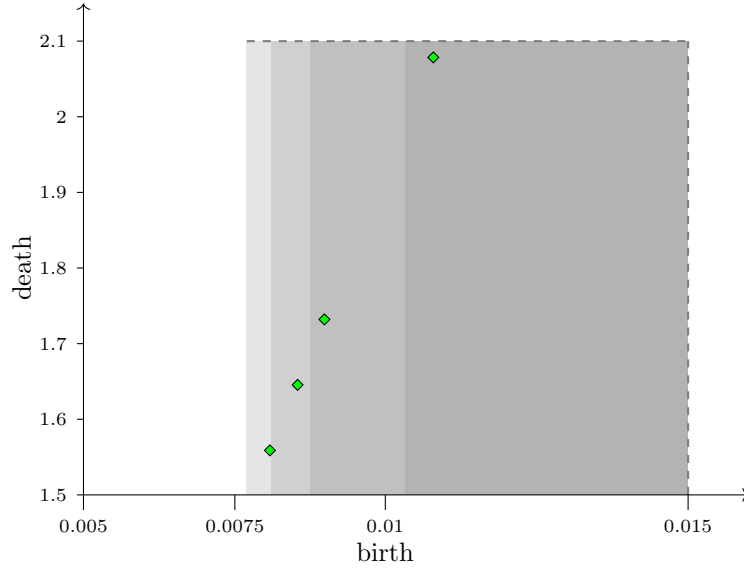

\newpage
\subsection{Computational example on the Radon transform of a time-dependent, periodically changing 
object.}\label{subsec:comp_radon}

In this computational example, we consider an object that changes over time in a periodic manner. The object we are considering is annulus-shaped, which is composed of a single material. The change happens in the following manner over time. First, the inner radius of the annulus increases, and the outer radius stays put. Then the inner radius of the annulus does not change while the outer radius starts to increase. Then the outer radius stays in the maximum position, and the inner radius starts to increase. When the inner radius has hit the minimum, it stays there, and the outer radius starts to decrease, see the Figure \ref{fig:ex2_obj}.

In the model setting, the (discretized) object in different states, with different inner and outer radii, forms the model space $X_{cont.}$
Furthermore, suppose that we can only observe this object by taking X-ray projections. We have the discretized Radon transform from a subspace  $X_{net}$ of $X_{cont.}$, 
\[
\mathsf{R}\colon X_{net}\to Y_{clean},
\] where $Y_{clean}$ is the set of clean measurements. The $Y_{noisy}$ is then the clean measurements affected by some noise.

In the continuous setting, let $t\in [-2,2]$ and $f_t\colon \R^2\to \{0,1\}$ such that \begin{align}
    f_t &= 1_{B(3+3t,0) \setminus B(1,0)}, \text{ when } 0<t<1  \\
    f_t & = 1_{B(6,0) \setminus B(1+(t-1),0)}, \text{ when } 1\leq t \leq 2   \\
    f_t & = 1_{B(3,0) \setminus B(1-t,0)}, \text{ when } 0 \leq -t < 1 \\
    f_t & = 1_{B(3+3(-t-1),0) \setminus B(2,0)}, \text{ when } 1 \leq -t \leq 2.
\end{align}
The compact support of each mapping $f_t$ is an annulus centered on the origo of form $B(r_1)\setminus B(r_2)$, where $r_1\in[3,6]$ and $r_2 \in [1,2]$.
The Radon transforms $\mathcal{R}f_t(\theta,\cdot)$ are the same for all measuring angles $\theta \in [0,2\pi)$ due to the object symmetry, and thus it is enough to consider a fixed angle $\theta$. Also, the data behave nicely in the sense that $\mathcal{R}f_t\neq\mathcal{R}f_{t'}$, when $t\neq t'$.
We have shown in Section \ref{sec:quasi} that there exist constants $L\geq1$ and $\varepsilon>0$ such that. 
\[
\frac{1}{L}\norm{f_{t}-f_{t'}}_{L^2_0}-\varepsilon\leq \norm{\mathcal{R}f_{t}-\mathcal{R}f_{t'}}_{L^2}\leq L\norm{f_{t}-f_{t'}}_{L^2_0}+\varepsilon,
\]
where $f_t,f_{t'}\in L^2_0$, and $\norm{f_t}_{H^s_0}\leq K,$ $K>0$, $0<s<\frac{1}{2}$. 

The result that the Radon transform adapts the quasi-isometry form is more theoretical than practical. In practice, it's difficult to know or even estimate the constants $L$ and $\varepsilon$. Thus, the experiment we present is here more to highlight the similarity of the persistence diagrams when having nicely behaving data, as here. Thus, it would be reasonable to assume that having more Radon transformation of the data would give us information about the object space, even if some assumptions are violated. We will demonstrate how similar the persistence diagrams are.  

\begin{figure}
    \centering
    \begin{tikzpicture}[scale=.6]
        \node[inner sep=0pt] (russell) at (0,0)
    {\includegraphics[width=.15\textwidth]{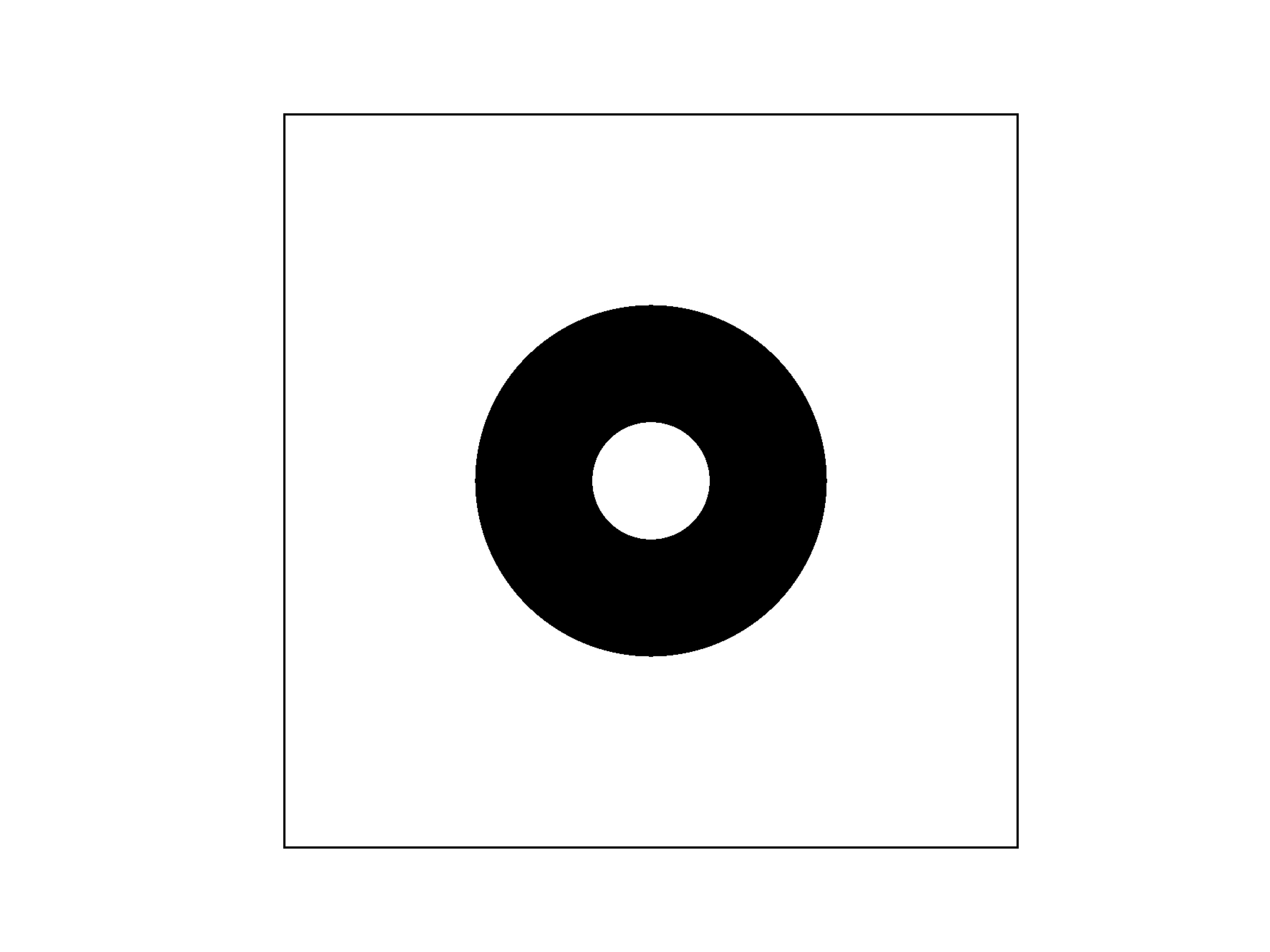}};
    \node[inner sep=0pt] (whitehead) at (3,0)
    {\includegraphics[width=.15\textwidth]{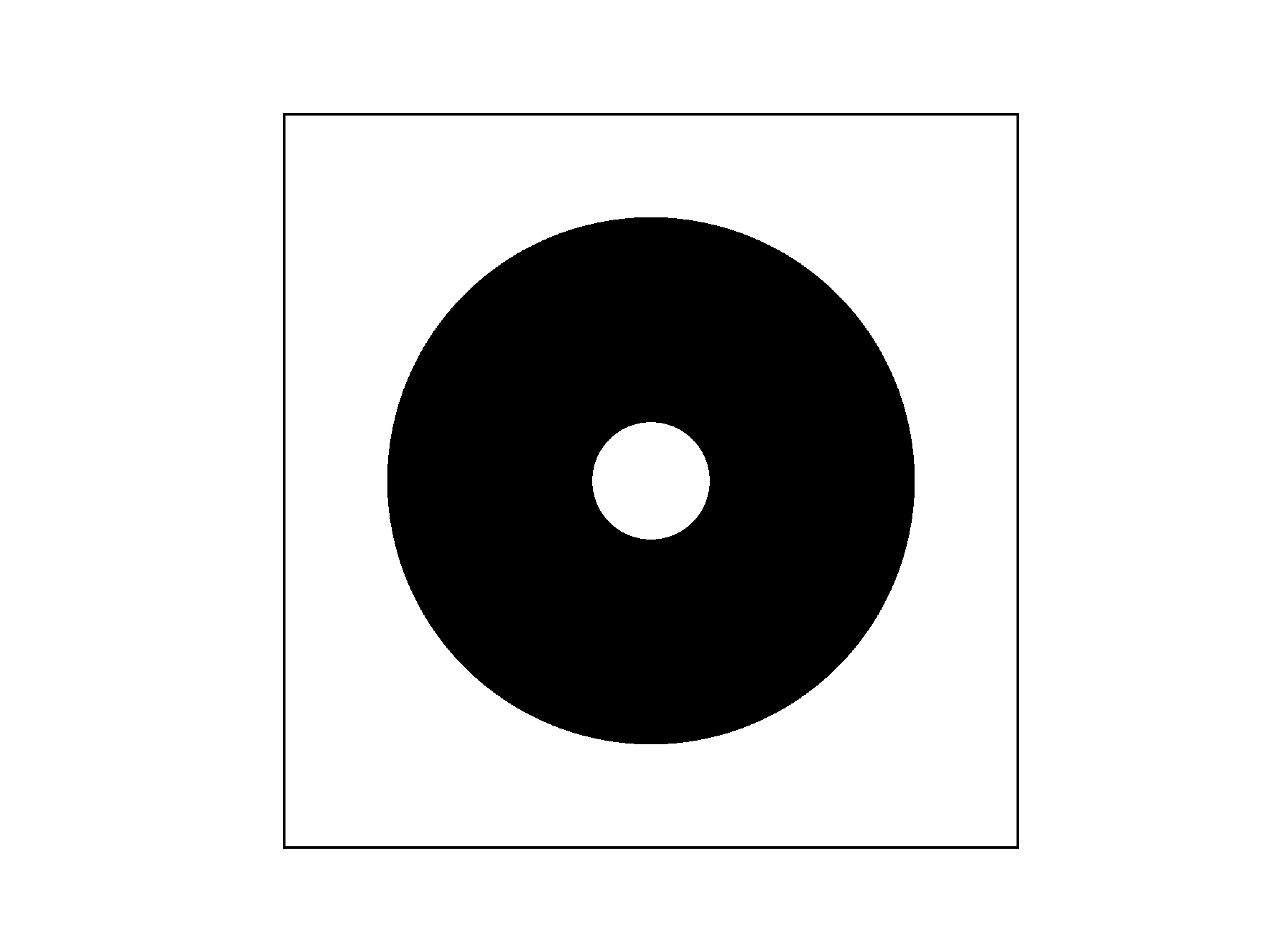}};
    \node[inner sep=0pt] (whitehead) at (6,0)
    {\includegraphics[width=.15\textwidth]{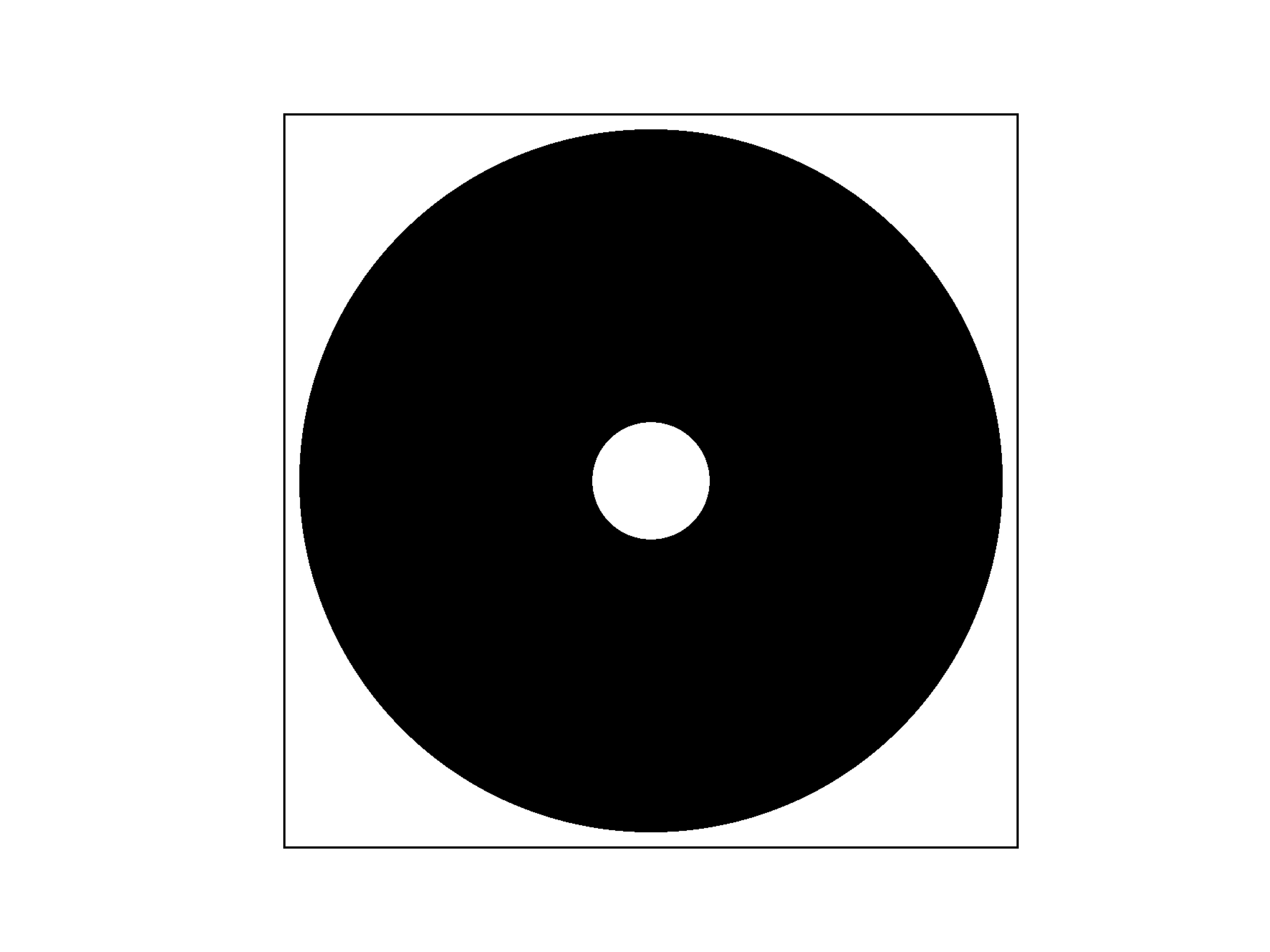}};
    \node[inner sep=0pt] (whitehead) at (6,6)
    {\includegraphics[width=.15\textwidth]{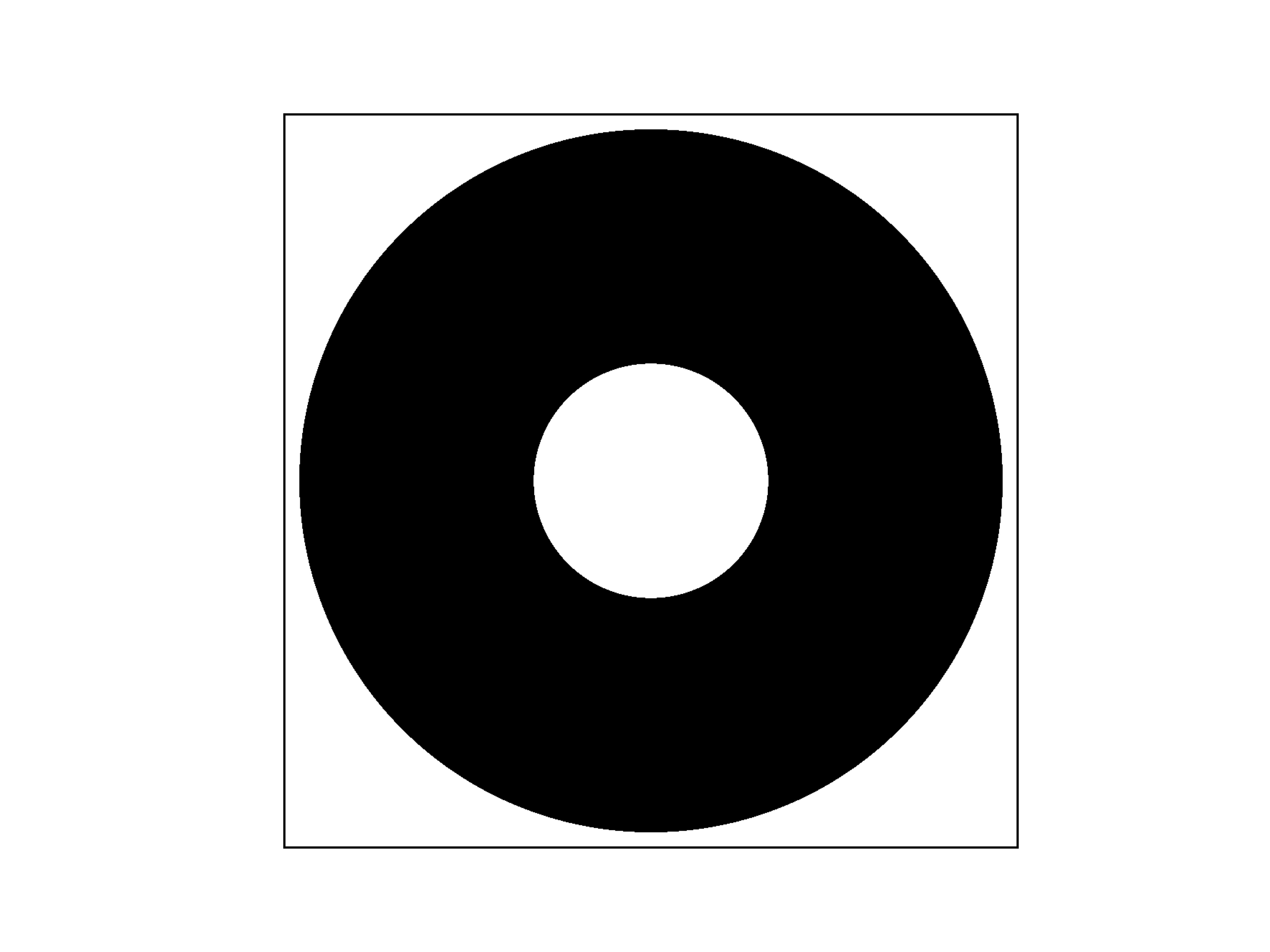}};
    \node[inner sep=0pt] (whitehead) at (0,6)
    {\includegraphics[width=.15\textwidth]{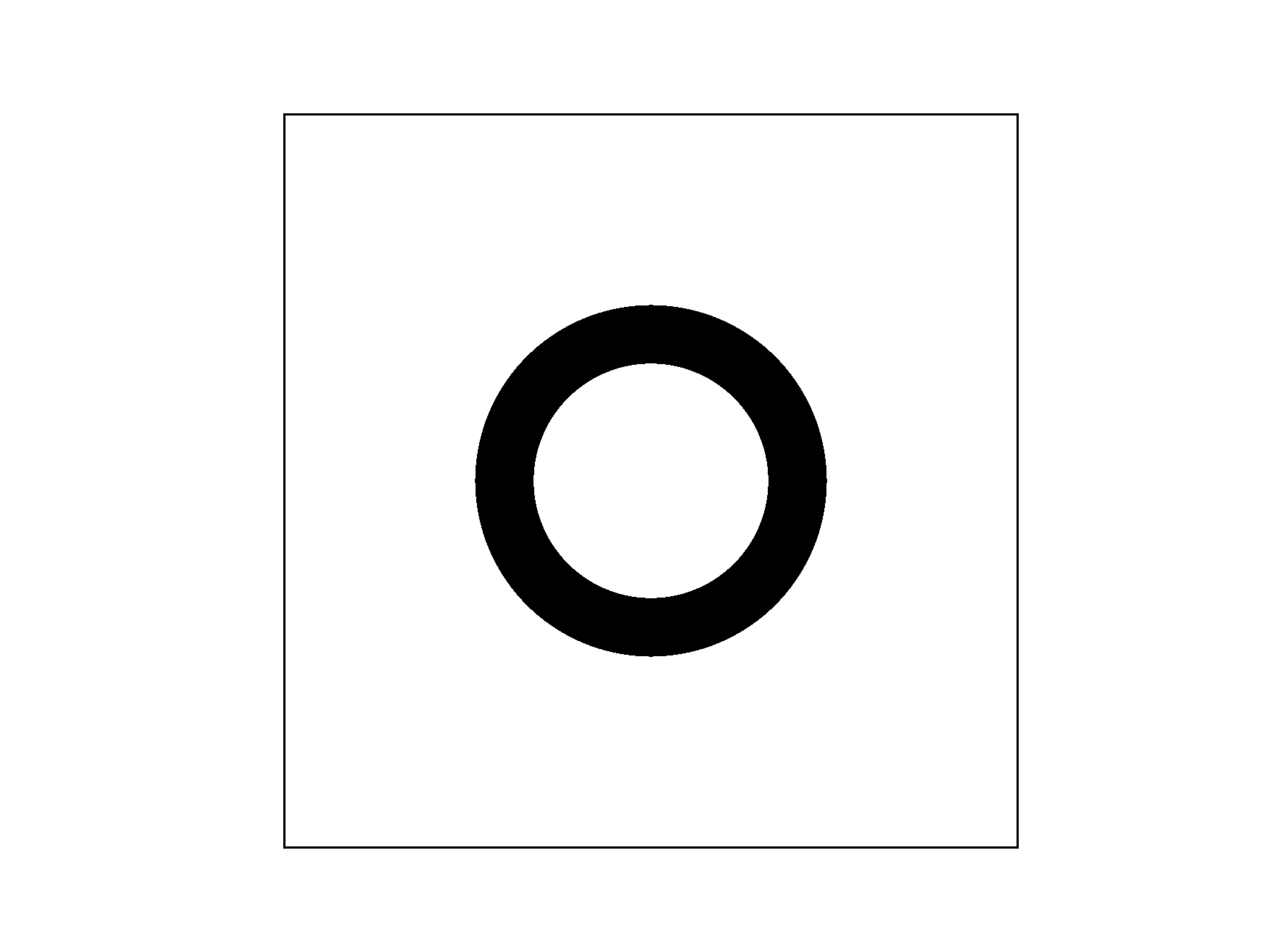}};
 \node[inner sep=0pt] (whitehead) at (3,6)
    {\includegraphics[width=.15\textwidth]{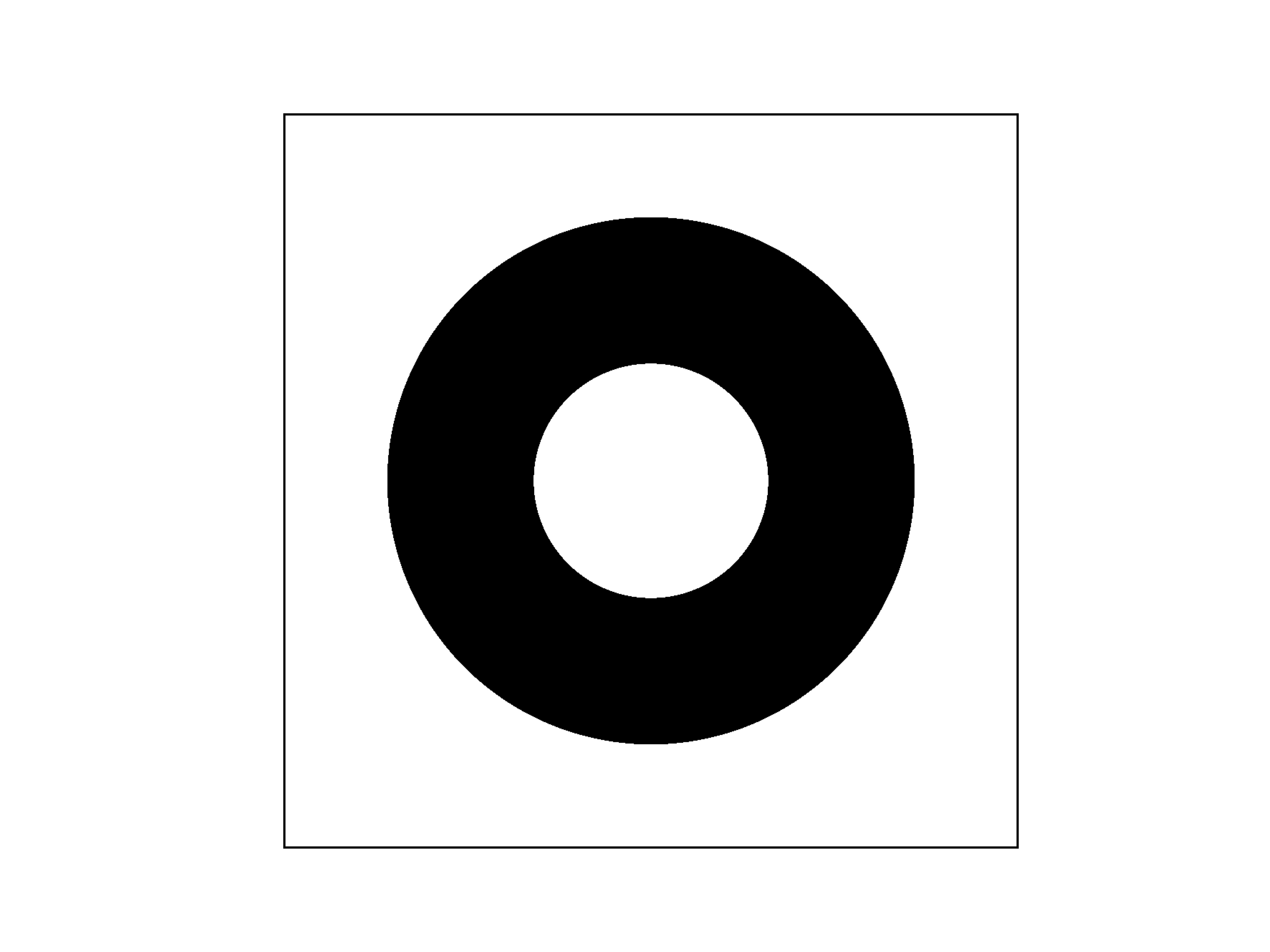}};
     \node[inner sep=0pt] (whitehead) at (6,3)
    {\includegraphics[width=.15\textwidth]{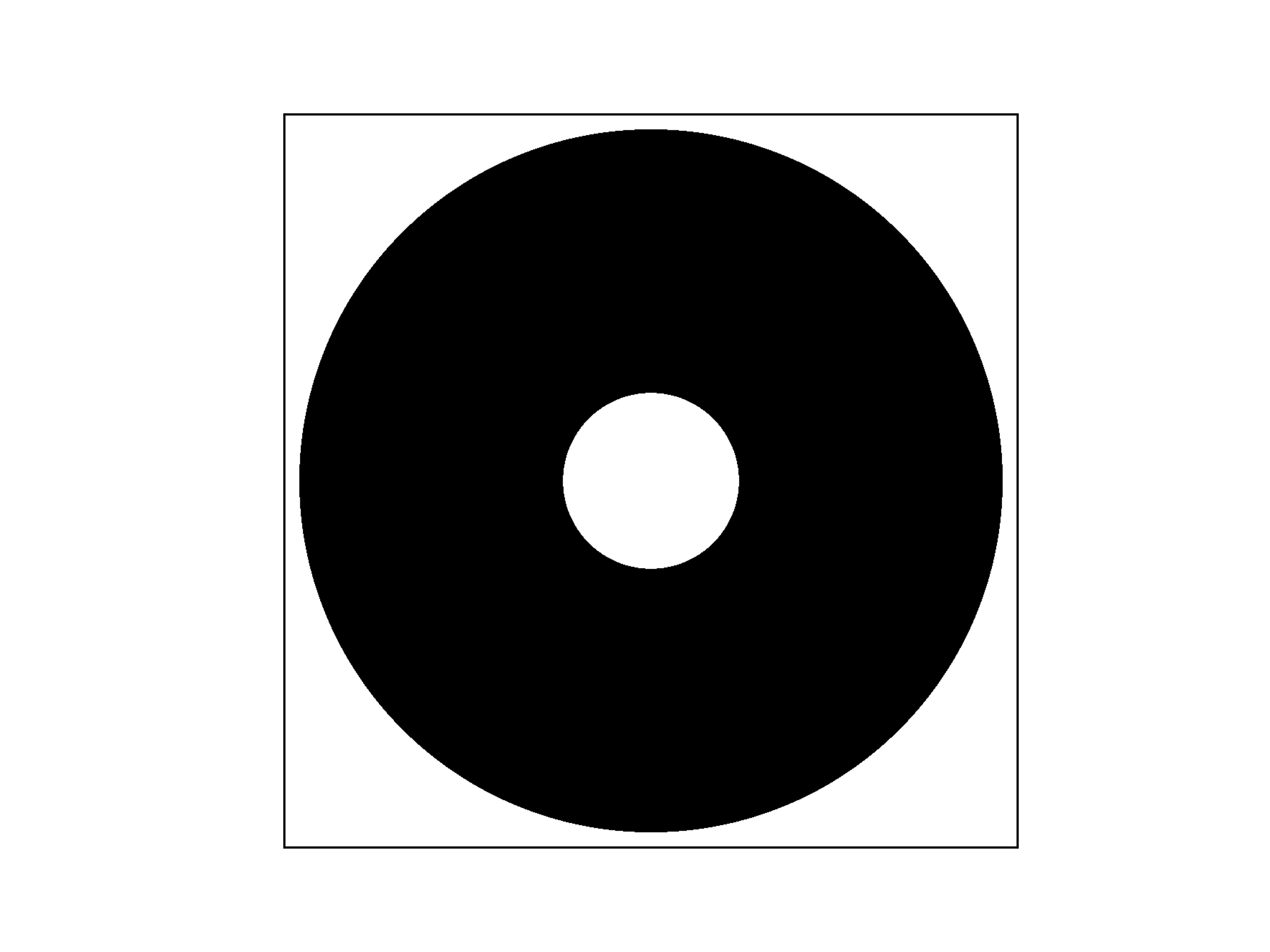}};
     \node[inner sep=0pt] (whitehead) at (0,3)
    {\includegraphics[width=.15\textwidth]{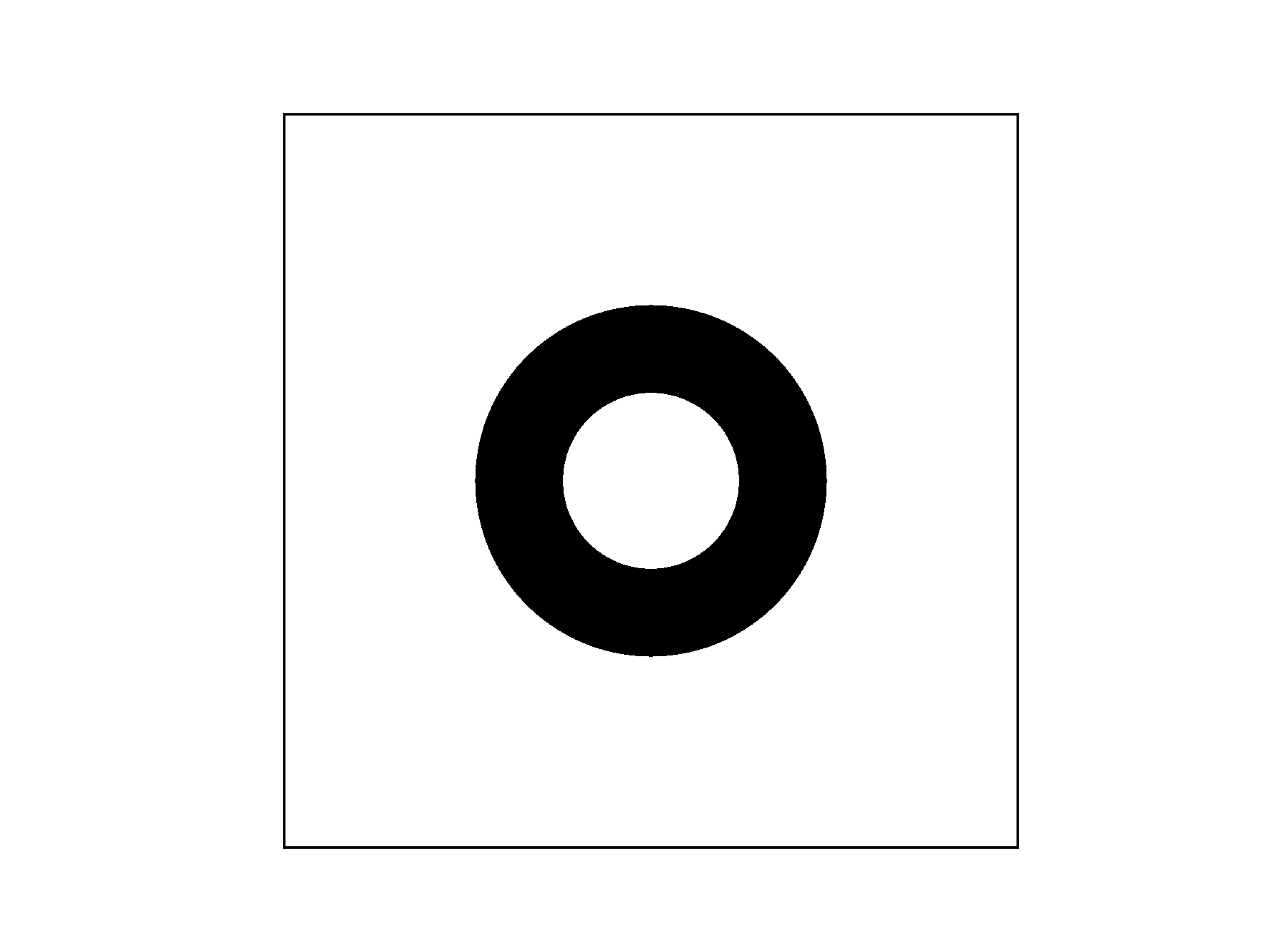}};
    \node[] at (0,1.5) {$\vdots$};
    \node[] at (0,4.5) {$\vdots$};
    \node[] at (6,1.5) {$\vdots$};
    \node[] at (6,4.5) {$\vdots$};
     \node[] at (1.5,0) {$\cdots$};
    \node[] at (4.5,0) {$\cdots$};
    \node[] at (1.5,6) {$\cdots$};
    \node[] at (4.5,6) {$\cdots$};
    \draw[->] (-1.5,-1.5) --(7.5,-1.5);
    \draw[->] (-1.5,-1.5) --(-1.5,7.5);
    \draw[-] (-1.7,3)--(-1.3,3);
    \node[left] at  (-1.7,3) {$1.5$};
    \draw[-] (-1.7,6)--(-1.3,6);
    \node[left] at  (-1.7,6) {$2$};
    \draw[-] (-1.7,0)--(-1.3,0);
    \node[left] at  (-1.7,0) {$1$};
    \draw[-] (3,-1.7)--(3,-1.3);
    \node[below] at  (3,-1.7) {$4.5$};
    \draw[-] (6,-1.7)--(6,-1.3);
    \node[below] at  (6,-1.7) {$6$};
    \draw[-] (0,-1.7)--(0,-1.3);
    \node[below] at  (0,-1.7) {$3$};
    \node[] at (3,-3.3){outer radius };
    \node[above, rotate=90] at (-3,3){inner radius };
\end{tikzpicture}
    \caption{The data consist $400$ annuli with different outer and inner radii.}
    \label{fig:ex2_obj}
\end{figure}
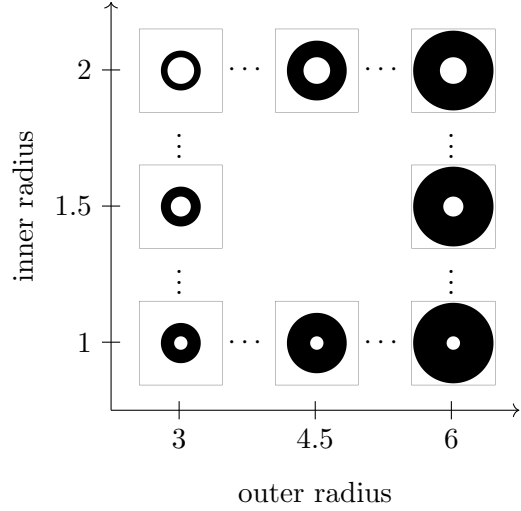

\begin{figure}
    \centering
    \begin{subfigure}[t]{0.49\textwidth}
             \begin{tikzpicture}[scale=0.7]
            \begin{axis}[
                 name=full,
                  xlabel={detector point},
                    ylabel={attenuation},
                ]
                \addplot[
                  only marks,
                  mark=*,
                  mark size=0.1pt,
                  black,
                    const plot,
                ]
                table {datafiles/radon_signal.dat};
                 \end{axis}

        \end{tikzpicture}
        \caption{Radon transform of an annulus.}
    \end{subfigure}
    \begin{subfigure}[t]{0.49\textwidth}
    \begin{tikzpicture}[scale=0.7]
            \begin{axis}[
                 name=full,
                  xlabel={detector point},
                    ylabel={attenuation},
                ]
                \addplot[
                  only marks,
                  mark=*,
                  mark size=0.1pt,
                  black,
                    const plot,
                ]
                table {datafiles/noisy_radon_signal.dat};
                 \end{axis}
        \end{tikzpicture}
        \caption{Noisy Radon transform of an annulus.}
    \end{subfigure}
    \caption{Example of Radon transform of an annulus belonging to $Y_{clean}$ and its noisy version in $Y_{noisy}$. }
    \label{fig:ex2_radon}
\end{figure}
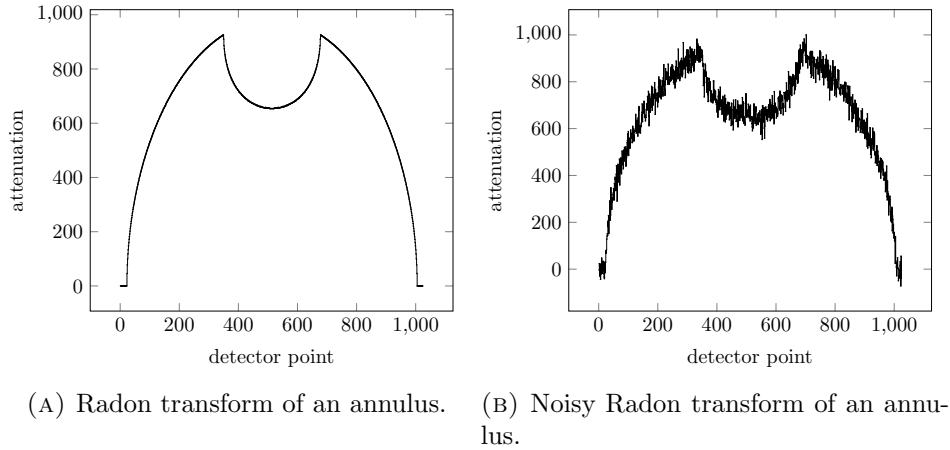

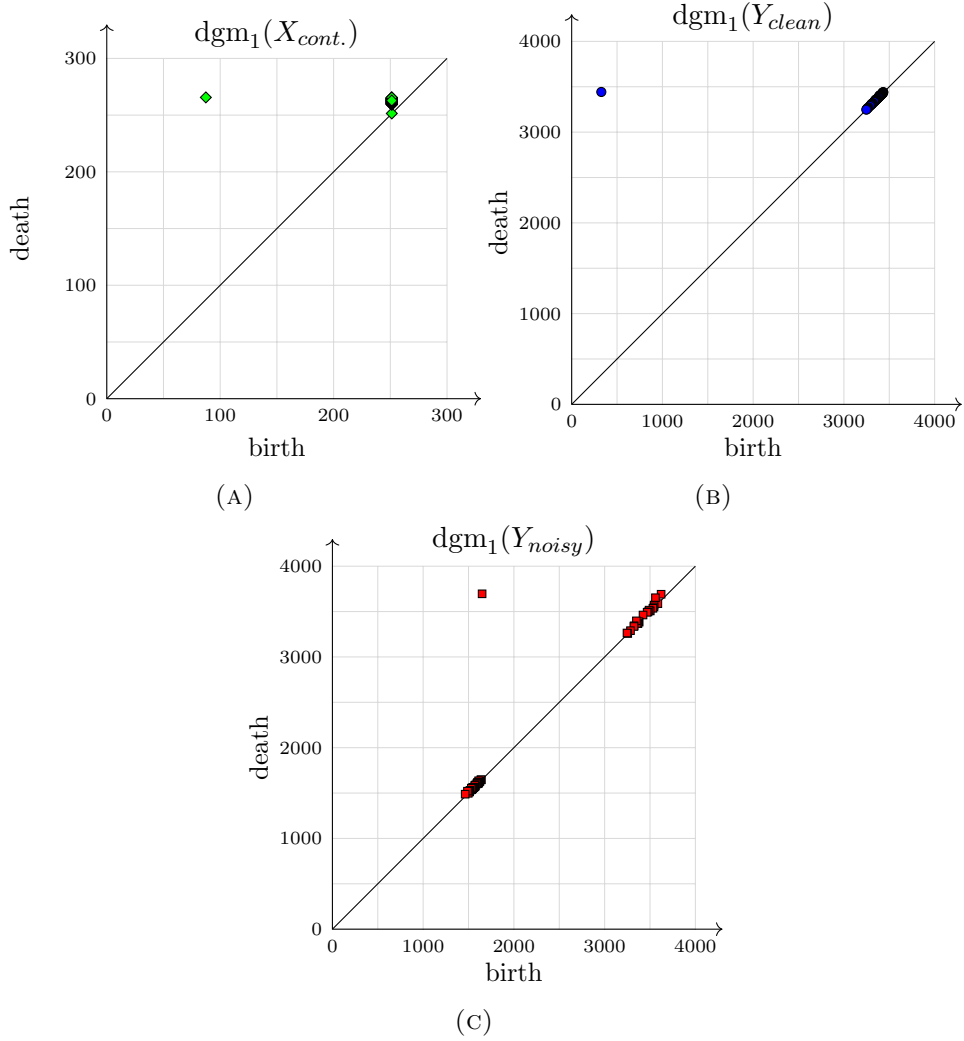
\begin{figure}[!htb]
\centering
     \begin{subfigure}[b]{0.49\textwidth}
         \centering
                 \begin{tikzpicture}[scale=0.15]
                \draw[step=5, gray!30, very thin] (0,0) grid (30,30);
                 \node[below] at (15,-2.5) {\small{birth}}; 
              \node[above,rotate=90] at (-6,15) {\small{death}}; 
                \foreach \x in {0,100,200,300}
                    \draw ({\x/10},0) -- ({\x/10},-0.03) node[below] {\tiny{\x}};
                \foreach \y in {0,100,200,300}
                    \draw (0,{\y/10}) -- (-0.02,{\y/10}) node[left] {\tiny{\y}};
                \draw[->](0,0)--(0,33); 
                \draw[->](0,0)--(33,0); 
                \draw[-](0,0)--(30,30); 
              \pgfplotstableread[col sep=space]{datafiles/image_dgm.dat}\datatable
              \foreach \i in {0,...,88} {
                \pgfplotstablegetelem{\i}{[index]0}\of\datatable
                \let\x\pgfplotsretval
                \pgfplotstablegetelem{\i}{[index]1}\of\datatable
                \let\y\pgfplotsretval
                    \node[draw, fill=green, minimum size=1.5mm, inner sep=0pt, shape=diamond] at ({\x/10},\y/10){};
              }
             \node[above] at (15, 30) {$\dgm_1(X_{cont.})$}; 
            
        \end{tikzpicture}
         \caption{}
         \label{subfig:pershom_diagram_annulus_data}
        \end{subfigure}
              \begin{subfigure}[b]{0.49\textwidth}
         \centering
        \begin{tikzpicture}[scale=0.12]
                \draw[step=5, gray!30, very thin] (0,0) grid (40,40);
                 \node[below] at (20,-2.5) {\small{birth}}; 
              \node[above,rotate=90] at (-6,20) {\small{death}}; 
                \foreach \x in {0,1000,2000,3000,4000}
                    \draw ({\x/100},0) -- ({\x/100},-0.03) node[below] {\tiny{\x}};
                \foreach \y in {0,1000,2000,3000,4000}
                    \draw (0,{\y/100}) -- (-0.02,{\y/100}) node[left] {\tiny{\y}};
                \draw[->](0,0)--(0,43); 
                \draw[->](0,0)--(43,0); 
                \draw[-](0,0)--(40,40); 
              \pgfplotstableread[col sep=space]{datafiles/radon_dgm.dat}\datatable
              \foreach \i in {0,...,59} {
                \pgfplotstablegetelem{\i}{[index]0}\of\datatable
                \let\x\pgfplotsretval
                \pgfplotstablegetelem{\i}{[index]1}\of\datatable
                \let\y\pgfplotsretval
                    \node[draw, fill=blue, minimum size=1.2mm, inner sep=0pt, shape=circle] at ({\x/100},\y/100){};
              }
             \node[above] at (20, 40) {$\dgm_1(Y_{clean})$}; 
            
        \end{tikzpicture}
         \caption{  }
         \label{subfig:pershom_diagram_clean_radon}
     \end{subfigure}
    \hfill
     \begin{subfigure}[b]{0.49\textwidth}
         \centering
        \begin{tikzpicture}[scale=0.12]
                \draw[step=5, gray!30, very thin] (0,0) grid (40,40);
                \node[below] at (20,-2.5) {\small{birth}}; 
              \node[above,rotate=90] at (-6,20) {\small{death}}; 
                \foreach \x in {0,1000,2000,3000,4000}
                    \draw ({\x/100},0) -- ({\x/100},-0.03) node[below] {\tiny{\x}};
                \foreach \y in {0,1000,2000,3000,4000}
                    \draw (0,{\y/100}) -- (-0.02,{\y/100}) node[left] {\tiny{\y}};
                \draw[->](0,0)--(0,43); 
                \draw[->](0,0)--(43,0); 
                \draw[-](0,0)--(40,40); 
              \pgfplotstableread[col sep=space]{datafiles/radon_noise_dgm.dat}\datatable
              \foreach \i in {0,...,62} {
                \pgfplotstablegetelem{\i}{[index]0}\of\datatable
                \let\x\pgfplotsretval
                \pgfplotstablegetelem{\i}{[index]1}\of\datatable
                \let\y\pgfplotsretval
                    \node[draw, fill=red, minimum size=1mm, inner sep=0pt, shape=rectangle] at ({\x/100},\y/100){};
              }
             \node[above] at (20, 40) {$\dgm_1(Y_{noisy})$}; 
        \end{tikzpicture}
         \caption{}
         \label{subfig:pershom_diagram_noisy_diagrams}
     \end{subfigure}
        \caption{The simplified persistence diagrams of dimension one using Euclidean distance. Here $X_{cont.}$ are set of discretized annuli, $Y_{clean}$ is the discretized Radon transforms of these annuli belonging to $X_{cont.}$. The set $Y_{noisy}$ is the noisy Radon transforms, where the relative error between noisy and clean signals are $3-5\%$. In each of the three diagrams, there is exactly one off-diagonal point.}
        \label{fig:radon_example_diagrams}
\end{figure}

The discretization of the cause data, the functions $f_t$, was done in the following way.
First, we took $400$ evenly spaced values in the interval $[-2,2)$. Let denoted these values by $t_i,$ $i=1,\dots,400$, $t_i<t_j, $ when $i<j$. We created a $1024\times 1024$ meshgrid over $[-6.25,6.25]\times[-6.25,6.25]$. Now each $f_{t_i}$ is evaluated at each grid point resulting a binary matrix $\mathsf{M}_{t_i}\in\{0,1\}^{1024\times 1024}$. Some of these binary matrices are presented in Figure \ref{fig:ex2_obj}. Denote that $X_{cont.}:=\{\mathsf{M}_{t_i}\}$. Denote that $X_{vec.}$ is the set whose elements are vectorizations of matrices in $X_{cont.}$.

Furthermore we take (discrete) Radon transform $\mathsf{R}$ of each $\mathsf{M}_{t_i}$; let us denote this set by $Y_{clean}:=\{\mathsf{R}(\mathsf{M}_{t_i})\mid \mathsf{M}_{t_i}\in X_{cont.}\}$. Furthermore, we added some Gaussian noise to elements in $Y_{clean}$ such that the relative error is $3-5\%$. Let denote the noisy data by $Y_{noisy}$. Example of clean and noisy Radon transforms are presented in Figure \ref{fig:ex2_radon}. The indexed persistent diagrams of dimension one for $X_{cont.}$ (using $X_{vec.}$), $Y_{clean}$ and $Y_{noisy}$ are computed using Euclidean distance. The simplified diagrams are presented in Figure \ref{fig:radon_example_diagrams}. The indexed persistence diagram $\Dgm_1(X_{cont.})$ contains 89 diagram points, and one point is clearly off the diagonal in the simplified persistence diagram. The indexed persistence diagram $\Dgm_1(Y_{clean})$ contains 60 points, and the indexed persistence diagram $\Dgm_1(Y_{noisy})$, contains 63 points. Also in the simplified persistence diagrams $\dgm_1(Y_{clean})$ and $\dgm_1(Y_{noisy})$, there are clear off-diagonal points.

\section{Conclusion}\label{sec:conclusion}
The study demonstrated a way to observe structures, including the periodic behavior of a set of causes from noisy measurements. We showed that a persistence diagram can be recovered with error bounds.  
 The focus was on the inverse direction, finding the structure of the continuous model space $X_{cont.}$ from noisy measurements $Y_{noisy}$. Results can also be presented for forward direction, in such a case, one can study what kind of structures can be assumed to be retained (in a topological sense) during the measurement process. Two examples were given, the Radon transform and the conductivity problem, arising from inverse problems, fitting to this problem setting.  
Even though our motivation in this study lay in inverse problems, one could find the results of quasi-interleaving interesting as such.  
\section*{Acknowledgments}
E.K.  was partially supported  by 
the FAME Flagship of Advanced Mathematics for Sensing Imaging and Modelling or the research Council of Finland (grant 359182) and P.P by the Center of Excellence FiRST.
M.L. was partially supported by PDE-Inverse project of the European Research Council of the European Union, the FAME and Finnish Quantum flagships and the grant 336786 of the Research Council of Finland. Views and opinions expressed are those of the authors only and do not necessarily reflect those of the European Union or the other funding organizations. Neither the European Union nor the other funding organizations can be held responsible for them.

\newpage
\bibliographystyle{plain} 
\bibliography{refs} 
\end{document}